\documentclass{article}
\usepackage{amssymb}
\usepackage{amsmath}
\usepackage{CJK}
\usepackage[colorlinks,linkcolor=blue]{hyperref} 
\usepackage{fancyhdr}
\usepackage{tikz-cd}
\usepackage{multicol}
\usepackage{stmaryrd}
\usepackage{wasysym}
\usepackage{amsthm}
\usepackage{geometry}
\usepackage{mathrsfs}
\usepackage{natbib} 
\usepackage{graphicx} 
\usepackage{float} 
\allowdisplaybreaks[4]

\usepackage{physics}
\usepackage{tabularx }
\usepackage[toc,page]{appendix}
\makeatletter
\renewenvironment{abstract}{%
	\if@twocolumn
	\section*{\abstractname}%
	\else 
	\begin{center}%
		{\bfseries \Large\abstractname\vspace{\z@}}
	\end{center}%
	\quotation
	\fi}
{\if@twocolumn\else\endquotation\fi}
\makeatother

\providecommand{\keywords}[1]
{
  
\textbf{{Keywords:}} #1
  
}

\theoremstyle{definition}
\newtheorem{theorem}{Theorem}[section]

\newtheorem{lem}[theorem]{Lemma}
\newtheorem{prop}[theorem]{Proposition}

\newtheorem{defi}[theorem]{Definition}
\newtheorem{assump}[theorem]{Assumption}

\newtheorem{remark}[theorem]{Remark}

\usepackage{hyperref}

\numberwithin{equation}{section}

\newcommand{\sym}[1]{(#1)_{\mbox{\tiny{sym}}}}
\newcommand{\skw}[1]{(#1)_{\mbox{\tiny{skw}}}}
\newcommand{\dom}[1]{\mbox{dom}(#1)}
\newcommand{\divge}[1]{\mbox{div}(#1)}

\newcommand{\tU}{\tilde{\pmb U}}
\makeatletter 
\newcommand\superimpose[2]{%
	\ooalign{$\m@th#1\@firstoftwo#2$\cr
		\hidewidth$\m@th#1\@secondoftwo#2$\hidewidth}%
}
\makeatother 
\newcommand{\threedotsord}[0]{\mathpalette\superimpose{{\mathop:}{\cdot}}} 
\newcommand{\threedotsbin}[0]{\mathbin{\threedotsord}}     

\title{Analysis of a Cahn--Hilliard model for viscoelastoplastic two-phase flows}
\author{Fan Cheng\thanks{Freie Universit\"at Berlin, Institute of Mathematics, Arnimallee 9, 14195 Berlin, Germany, e-mail: fan.cheng@fu-berlin.de}, Robert Lasarzik\thanks{Weierstrass Institute for Applied Analysis and Stochastics, Anton-Wilhelm-Amo-Str. 39, 10117 Berlin, Germany, e-mail: robert.lasarzik@wias-berlin.de}, Marita Thomas\thanks{Freie Universit\"at Berlin, Institute of Mathematics, Arnimallee 9, 14195 Berlin, Germany, e-mail: marita.thomas@fu-berlin.de}}
\date{\today}

\begin{document}
        \maketitle
	\begin{abstract}
    \noindent
    We study a Cahn--Hilliard two-phase model describing the flow of two viscoelastoplastic fluids, which arises in geodynamics. A phase-field variable indicates the proportional distribution of the two fluids in the mixture. The motion of the incompressible mixture is described in terms of the volume-averaged velocity. Besides a volume-averaged Stokes-like viscous contribution, the Cauchy stress tensor in the momentum balance contains an additional volume-averaged internal stress tensor to model the elastoplastic behavior. This internal stress has its own evolution law featuring the nonlinear Zaremba-Jaumann time-derivative and the subdifferential of a non-smooth plastic potential. The well-posedness of this system is studied in two cases: Based on a regularization by stress-diffusion we obtain the existence of Leray-Hopf-type weak solutions. In order to deduce existence results also in the absence of the regularization 
    we introduce the concept of dissipative solutions, which is based on an estimate for the relative energy. 
    We discuss general properties of dissipative solutions and show their existence for the viscoelastoplastic two-phase model in the setting of stress-diffusion. By a limit passage in the relative energy inequality for vanishing stress-diffusion, we conclude an existence result for the non-regularized model.
\end{abstract}
    \newpage
	\tableofcontents
	\section{Introduction}\label{section introduction}
The movement of tectonic plates driven by convective processes within the Earth’s mantle is a generally recognized theory which explains many geological phenomena, e.g., earthquakes, volcanoes, formation of mountains, ocean trenches, mid-ocean ridges, and island arcs, see~\cite{Schubert2001} for details. These tectonic plates form the Earth's lithosphere, consisting of the Earth's crust and the uppermost part of its mantle. On geological time scales of millions of years, the moving and mechanically deforming plates in the mantle are treated as non-Newtonian viscoelastoplastic fluids. 

In this paper, we consider a system of equations which describes a two-phase flow of a incompressible mixture of two viscoelastoplastic fluids arising from geodynamics. In a time interval $(0,T)$, where $T\in(0,\infty]$, and a bounded $C^{2}-$domain $\Omega\subseteq\mathbb{R}^3$, the system is given as:
\begin{subequations}
	\label{two phase system}
	\begin{align}
		\partial_{t} (\rho v)+\divge{v\otimes(\rho v+J)}-\divge{\mathbb{T}}&=f-\varepsilon\divge{\nabla\varphi\otimes\nabla\varphi}&\text{ in }\Omega\times(0,T), 
        \label{fluid equation momentum}\\
		\divge{v}&=0&\text{ in }\Omega\times(0,T),
        \label{fluid equation incompressible}\\
        \mathbb{T}&=\eta(\varphi)S+2\nu(\varphi)\sym{\nabla v}-p\mathbb{I}&\text{ in }\Omega\times(0,T),
        \label{fluid equation cauchy stress tensor}\\
		\overset{\triangledown}{S}+\partial {P}(\varphi;S)-\gamma\Delta S&\ni\eta(\varphi)\sym{\nabla v}&\text{ in }\Omega\times(0,T), \label{fluid equation stress}\\
		\partial_{t}\varphi+v\cdot\nabla\varphi&=\Delta\mu&\text{ in }\Omega\times(0,T), 
        \label{first CHE}\\
		\mu&=\frac{1}{\varepsilon}W^{\prime}(\varphi)-\varepsilon\Delta\varphi,&\text{ in }\Omega\times(0,T). 
        \label{second CHE}		
	\end{align}
The system is complemented by the following boundary and initial conditions 
	\label{BC and ID for TP}
	\begin{align}
		v|_{\partial\Omega}&=0\mbox{ on }\partial\Omega\times(0,T),
        \\
        \gamma \vec{n}\cdot\nabla S|_{\partial\Omega}&=0\mbox{ on }\partial\Omega\times(0,T),
        \\\vec{n}\cdot\nabla\varphi|_{\partial\Omega}=\vec{n}\cdot\nabla\mu|_{\partial\Omega}&=0\mbox{ on }\partial\Omega\times(0,T),
        \\
		(v,S,\varphi)|_{t=0}&=(v_{0},S_{0},\varphi_{0})\mbox{ in }\Omega,
	\end{align}
\end{subequations}
where $\vec{n}$ denotes the outward unit normal vector of $\Omega$.

Equation~\eqref{fluid equation momentum} describes the momentum balance for fluids, phrased in terms of the volume-averaged Eulerian velocity field $v:[0,T]\times\Omega\to\mathbb{R}^{3}$, the mass density $\rho:[0,T]\times\Omega\to\mathbb{R}$, the volume-averaged mass flux $J:[0,T]\times\Omega\to\mathbb{R}^{3}$ which occurs due to the unmatched mass density, the Cauchy stress tensor $\mathbb{T}:[0,T]\times\Omega\to\mathbb{R}^{3\times3}$, an external loading $f:[0,T]\times\Omega\to\mathbb{R}^{3}$, and the Korteweg stress representing the capillarity stress, which is modeled as $\varepsilon \nabla\varphi\otimes\nabla\varphi$ with a parameter $\varepsilon>0$. Equation~\eqref{fluid equation incompressible} states the incompressibility condition. Equation~\eqref{fluid equation cauchy stress tensor} provides the Cauchy stress tensor for viscoelastoplastic fluids, consisting of two parts: a radial part, given by the pressure $p:[0,T]\times\Omega\to\mathbb{R}$, and a deviatoric part, which is in the form of a symmetric matrix with zero trace including the viscous part given as $2\nu(\varphi)\sym{\nabla v}$ and an extra contribution by the internal $S$ given as $\eta(\varphi)S$ where coefficients $\nu$ and $\eta$ are functions of $\varphi$ arising from the viscosity and the elasticity of the fluids. The strain rate $\sym{\nabla v}=\frac{1}{2}(\nabla v + \nabla v^{\top})$, which is the symmetric part of the velocity gradient, describes the relative motion between the particles. 

Equation~\eqref{fluid equation stress} characterizes the evolution law of the volume-averaged internal stress $S:[0,T]\times\Omega\to\mathbb{R}_{\mathrm{sym,Tr}}^{3\times3}$, which takes the form of a Maxwell-type stress-strain relation. Moreover, the rate of the internal stress is controlled by the Zaremba-Jaumann rate $\overset{\triangledown}{S}$ defined as
\begin{align}\label{Zaremba-Jaumann rate}
	\overset{\triangledown}{S}:=\partial_{t} S+v\cdot\nabla S+S\skw{\nabla v}-\skw{\nabla v}S,
\end{align}
which is a notion of time derivative widely used in geophysical models. Here, the spin tensor $\skw{\nabla v}=\frac{1}{2}(\nabla v - \nabla v^{\top})$ denotes the skew-symmetric part of the velocity gradient. To model plastic effects, an additional term $\partial P(\varphi;S)$ is incorporated into equation~\eqref{fluid equation stress}. An example of the plastic potential, which is used in geodynamics for the plastic deformation in lithospheric plates, is given as 
\begin{equation}
    P(\varphi;S):=G(\varphi)P_{1}(S)+(1-G(\varphi))P_{2}(S).
\end{equation}
Here,
\begin{equation}
    P_{i}(S):=\begin{cases}
		\frac{a_{i}}{2}\abs{S}^{2}+b_{i}\abs{S}&\mbox{ if }\abs{S}\leq\sigma_{\mathrm{{yield},i}},\\
		\infty&\mbox{ if }\abs{S}>\sigma_{\mathrm{{yield},i}},
	\end{cases}
\end{equation}
where $a_{i}>0$, $b_{i}>0$ are constants of each pure phase $i$ and $\sigma_{\mathrm{{yield},i}}$ is the yield stress of each pure phase $i$, $i=1,2$, which determines the onset of plastic flow behavior. Moreover, $\partial P(\varphi;S)$ denotes the subdifferential of the convex potential $P(\varphi;\cdot)$ in $S$. It is defined by 
\begin{equation}
    \partial{P}(\varphi;S):=\left\{\xi\in\mathbb{R}_{\mathrm{sym,Tr}}^{3\times3}:\langle \xi,\Tilde{S}-S\rangle_{\mathbb{R}^{3\times3}}+P(\varphi;{S})\leq P(\varphi;\tilde{S})\right\}.
\end{equation}
In addition, the plastic potential is given as 
\begin{equation}
    \mathcal{P}(\varphi;S):=\int_{\Omega}P(\varphi;S)\dd{x}.
\end{equation}

The Cahn-Hilliard type equation~\eqref{first CHE}-~\eqref{second CHE} describes the evolution law of the phase-field variable $\varphi:[0,T]\times\Omega\to\mathbb{R}$ which indicates the presence of each of the two phases. Hence, in the physical sense, we expect, for $(t,x)\in[0,T]\times\Omega$
\begin{equation}
    \varphi(t,x)\begin{cases}
        =-1,&\text{ pure fluid 1,}\\
        \in(-1,1),&\text{ mixture of fluid 1 and fluid 2,}\\
        =1,&\text{ pure fluid 2.}
    \end{cases}
\end{equation}
Moreover, in equation~\eqref{first CHE}-\eqref{second CHE}, $\mu:[0,T]\times\Omega\to\mathbb{R}$ is the associated chemical potential, $W:\mathbb{R}\to[0,\infty]$ is a singular potential, and $\varepsilon>0$ is a small parameter related to the thickness of the interface. 

In particular, the mass density is modeled as 
\begin{equation}
    \rho(\varphi):=\frac{1-\varphi}{2}\rho_{1}+\frac{1+\varphi}{2}\rho_{2},
\end{equation}
where $\rho_{i}>0$ is the constant mass density of each pure fluid $i$, $i=1,2$. Hence, equation~\eqref{first CHE} provides the additional continuity equation
\begin{equation}
    \partial_{t}\rho+\divge{\rho v + J}=0.
\end{equation}

To have an understanding of this system, let us first define the total energy of the system at time $t\in[0,T]$, i.e., the sum of the kinetic energy depending on  $\rho$ and $v$, the elastic energy depending on $S$ and the phase-field energy depending on $\varphi$:
\begin{equation}\label{Energy smooth solution}
	{E}(t):=\mathcal{E}(v(t),S(t),\varphi(t))
    :=\int_{\Omega}
    \frac{\rho(\varphi)}{2}\abs{v(x,t)}^{2}
    +\frac{1}{2}\abs{S(x,t)}^{2}
    +\frac{1}{2}\abs{\nabla\varphi}^{2}+W(\varphi)\dd{x}.
\end{equation}
Assume for now that we have a smooth solution $(v,S,\varphi,\mu)$ of system \eqref{two phase system}. Then we multiply \eqref{fluid equation momentum} by $v$, \eqref{fluid equation stress} by $S$, \eqref{first CHE} by $\mu$ and \eqref{second CHE} by $\partial_{t}\varphi$, integrate over space and time, and perform an integration by parts, so that we obtain the following energy-dissipation balance:
\begin{equation}\label{original total energy balance}
    {E}(t)	+\int_{0}^{t}\int_{\Omega}2\nu(\varphi)\abs{\sym{\nabla v}}^{2}
	+\gamma\abs{\nabla S}^{2}
	+\xi:S 
    +\abs{\nabla \mu}^{2} \dd{x}\dd{\tau}	
    ={E}(0)+\int_{0}^{t}\langle f, v\rangle_{H^{1}}\dd{\tau}
\end{equation}
for all $t\in[0,T]$ and where $\xi\in \partial{P}(\varphi;S)$. From this energy-dissipation balance, we can see that the total energy of the system is dissipated by four parts: a Newtonian viscosity, a quadratic stress diffusion, dissipation due the plastic deformation stemming from the non-smooth potential $P$ and a quadratic term due to phase separation involving the chemical potential $\mu$.

The mathematical challenges associated with analysis of this system stem from the following: First of all, the momentum balance~\eqref{fluid equation momentum} comes with all the difficulties arising from the three-dimensional Navier--Stokes equations. Based on the large body of analytical results on Navier--Stokes equations, cf. e.g.~\cite{constantinNavierStokesEquations1988a,lionsMathematicalTopicsFluid1996,galdiIntroductionNavierStokesInitialBoundary2000b,farwigNewClassWeak2006a,sohrNavierStokesEquationsElementary2001a,roubicekNonlinearPartialDifferential2013,robinsonThreeDimensionalNavierStokes2016a}, we cannot expect a better class of solutions for the velocity field than Leray--Hopf solutions, which were introduced in~\cite{learyhopfsolution1,learyhopfsolution2}. 

Moreover, the stress evolution equation~\eqref{fluid equation stress} includes a nonlinearity in the Zaremba-Jaumann derivative~\eqref{Zaremba-Jaumann rate} of $S$ and a multi-valued derivative due to the non-smoothness of the dissipation potential ${P}$. \cite{lionsGLOBALSOLUTIONSOLDROYD2000} studies the global existence of solutions based on the Zaremba-Jaumann derivative together with a smooth dissipation potential such that $\partial{P}(S)=a S$. In addition, the stress diffusion $-\gamma\Delta S$ with $\gamma>0$ is introduced as a regularization for analytical reasons as in~\cite{eiterLerayHopfSolutionsViscoelastoplastic2022b}. However, in order to get closer to models used in geoscientific applications, e.g.~\cite{moresi_mantle_2002}, we aim to avoid it in our analysis.

Furthermore, this model contains more than one phase, which also gives us the difficulties arising from the coupled Cahn--Hilliard--Navier--Stokes system. We refer to~\cite{elliottCahnHilliardEquation1986a,abelsConvergenceEquilibriumCahn2007} for the study of Cahn--Hilliard equations and to~\cite{abelsExistenceWeakSolutions2009,colliGlobalExistenceWeak2012,modelingofsystem,abelsExistenceWeakSolutions2013} for the study of the Cahn--Hilliard--Navier--Stokes system. In contrast to the above-mentioned works, we additionally face the difficulties of the multi-valued stress evolution~\eqref{fluid equation stress}. A single-phase system of the momentum and stress equations has already been studied in~\cite{eiterLerayHopfSolutionsViscoelastoplastic2022b} and~\cite{eiterWeakstrongUniquenessEnergyvariational2022b}.  

Our goal is to study the existence results for system~\eqref{two phase system} not only in the case $\gamma>0$ but, more importantly, also in the case $\gamma=0$. This will be achieved with the help of an alternative concept of solution: the dissipative solution. The idea of dissipative solutions is that the solutions do not satisfy the equation in the distributional sense anymore, but, instead, one controls the difference between the solutions and smooth test functions satisfying the equations in terms of the relative energy and the relative dissipation. The concept of dissipative solutions was introduced by P.L. Lions~\cite[Sec.~4.4]{lionsMathematicalTopicsFluid1996} in the context of the incompressible Euler equations and his motivation stemmed from the consideration of the singular limit in the Boltzmann equation and identification of its limits~\cite{LionsBoltzman}. Since then, dissipative solutions have been applied in different contexts, for instance to singular limits in the Boltzman equations~\cite{saint}, incompressible viscous electro-magneto-hydrodynamics~\cite{raymond}, equations of viscoelastic diffusion in polymers~\cite{viscoelsticdiff}, the Ericksen--Leslie equations~\cite{diss,approx}, or finite-element approximation of nematic electrolytes~\cite{Lubo}. Moreover, it is also applied to isothermal damped Hamiltonian systems in \cite{zbMATH07432866} and to a viscoelastoplastic single-phase models in~\cite{eiterWeakstrongUniquenessEnergyvariational2022b}, where this notion of solutions is called energy-variational solution. The term dissipative solutions is also used for other solution concepts. On the one hand, it is used in \cite{feireisl2014relative}  in the context of the Navier-Stokes-Fourier system, where it basically denotes a weak solution. On the other hand the term dissipative solution is also used for different measure-valued solution frameworks, see for instance~\cite{breit,basaric}. These concepts are different from the dissipative solutions in~\cite{lionsMathematicalTopicsFluid1996} and we rather refer with this term to the original definition by Lions.

	\section{General notations, preliminaries and assumptions}\label{section preliminary}
In this section, we fix the notation that will be used throughout this work and recall some useful results that will be applied for our analysis. 
\subsection{General notations}
By default, we use Einstein’s summation convention for vectors and tensors. Let $a=(a_j)_{j=1}^{3},b=(b_j)_{j=1}^{3}\in\mathbb{R}^3$ be two vectors, then their inner product is written as $a\cdot b:=a_{j}b_{j}$, and their tensor product is written as $a\otimes b=(a_{j}b_{k})_{j,k=1}^{3}$.  Similarly, let $A=(A_{jk})_{j,k=1}^{3},B=(B_{jk})_{j,k=1}^{3}\in\mathbb{R}^{3\times3}$ be two second-order tensors, the tensor inner product is written as $A:B=A_{jk}B_{jk}$. Besides, for two third-order tensors $C=(C_{jkl})_{j,k,l=1}^{3},D=(D_{jkl})_{j,k,l=1}^{3}\in\mathbb{R}^{3\times3\times3}$, we denote the inner product by $C\threedotsbin D=C_{jkl}D_{jkl}$. The tensor product between a second-order tensor $A\in\mathbb{R}^{3\times3}$ and a vector $a\in\mathbb{R}^{3}$ gives a third-order tensor and it is defined as $(A\otimes a)_{jkl}=(A_{jk}a_{l})_{j,k,l=1}^{3}$. We write the transpose and trace of a matrix $A\in\mathbb{R}^{3\times3}$ in the usual way that $A^{\top}$ and $\mbox{Tr}A$. Moreover, we set the space of the symmetric and trace-free second order tensors as
\begin{align}
    \mathbb{R}_{\mathrm{sym,Tr}}^{3\times3}
    :=\left\{A\in\mathbb{R}^{3\times3}:A=A^{\top},\mbox{Tr}A=0\right\}.
\end{align} 

The point $(x,t)\in\Omega\times(0,T)$ is defined by the spatial variable $x\in\Omega$ and the time variable $t\in(0,T)$. Thus, we write the partial time derivative and partial spatial derivative of a (sufficiently regular) function $u:\Omega\times(0,t)\to\mathbb{R}$ as $\partial_{t} u$ and $\partial_{x_{i}} u$, where $i=1,2,3$. Moreover, $\nabla u$ and $\Delta u$ denote the gradient and Laplace of $u$ with respect to spatial variable. The symmetrized and skew-symmetrized part of $\nabla v$ of a vector field $v:\Omega\to\mathbb{R}^{3}$ is given by 
\begin{align}
	\sym{\nabla v}:=\frac{1}{2}\left(\nabla v+\nabla v^{\top}\right)
	\mbox{ and }
	\skw{\nabla v}:=\frac{1}{2}\left(\nabla v-\nabla v^{\top}\right).
\end{align}
We also write $\divge{v}=\partial_{x_{i}} v_{i}$ as the divergence of vector field $v$. Similarly, for a second-order tensor $S$, the divergence is defined by $\divge{S}=\partial_{x_{k}} S_{jk}$.

~\\
\textbf{Function spaces.} Let $X$ be a Banach space with norm $\lVert\cdot\rVert_{X}$ and dual space $X^{\prime}$. The same notation is used also for $X^{3}$ and $X^{3\times3}$. When the dimension is clear, we simply write $X$ instead of $X^{3}$ or $X^{3\times3}$. We use $\langle x^{\prime},x\rangle_{X}$ to denote the duality pairing of $x^{\prime}\in X^{\prime}$ and $x\in X$. 

The space $C^{\infty}(\Omega)$ denotes the class of smooth functions in $\Omega$ and the space $C_{0}^{\infty}(\Omega)$ denotes the class of smooth functions with compact support in $\Omega$. The Lebesgue spaces and Sobolev spaces are denoted as $L^p(\Omega)$ and $W^{k,p}(\Omega)$ for $p\in[1,\infty]$ and $k\in\mathbb{N}$, in particular, for $p=2$, we write $W^{k,2}(\Omega)=H^{k}(\Omega)$. Moreover, we write $H_{0}^{1}(\Omega)$ as the space of functions in $H^1(\Omega)$ whose boundary value is zero in the trace sense and $H^{-1}(\Omega):=(H_{0}^{1}(\Omega))^{\prime}$ is the dual space.

Now let $(0,T)\subseteq\mathbb{R}$ be an interval, the space $C^{0}(0,T;X)$ consists of the class of continuous functions in time with values in the Banach space $X$. For $p\in[1,\infty]$, the corresponding Lebesgue-Bochner spaces are denoted by $L^{p}(0,T;X)$. Moreover, we write $W^{1,p}(I;X):=\left\{u\in L^{p}(0,T;X):\partial_{t}u\in L^{p}(0,T;X)\right\}$ and $H^{1}(0,T;X)=W^{1,2}(0,T;X)$. The local Lebesgue-Bochner spaces $L_{\mathrm{loc}}^{p}(0,T;\Omega)$ and $H_{\mathrm{loc}}^{1}(0,T;\Omega)$ consist of the class of functions in $L^{p}(J;X)$ and $H^{1}(J;X)$ for every compact subinterval $J\subseteq (0,T)$ respectively. Besides, we will simply write $u(t):=u(\cdot,t)$ for $u$ defined on $\Omega\times I$.
For any $u\in L^{1}(\Omega)$
\begin{align}\label{mean value in Omega}
	u_{\Omega}:=\frac{1}{|\Omega|}\int_{\Omega}u\dd{x}
\end{align} 
is the mean value of $u$ in $\Omega$.

~\\
\textbf{Solenoidal vector fields and symmetric deviatoric fields.} We introduce function spaces for solenoidal (divergence-free) vector fields and symmetric, deviatoric (trace-free) fields. The corresponding classes of smooth functions on $\Omega$ are given by
\begin{subequations}
\begin{align}
	C_{0,\mathrm{div}}^{\infty}(\Omega)&:=\left\{\varphi\in C_{0}^{\infty}(\Omega)^3:\divge{\varphi}=0\text{ in }\Omega\right\},
    \\
    C_{\mathrm{sym,Tr}}^{\infty}(\bar{\Omega})&:=\left\{\psi\in C^{\infty}(\bar{\Omega})^{3\times 3}:\psi=\psi^{\top},\Tr(\psi)=0\text{ in }\Omega\right\}.
\end{align}
We further write the time-dependent solenoidal vector fields and symmetric, deviatoric fields as
\begin{align}
	C_{0,\mathrm{div}}^{\infty}(\Omega\times I)&:=\left\{\Phi\in C_{0}^{\infty}(\Omega\times I)^3:\divge{\Phi}=0\text{ in }\Omega\times I\right\},
    \\
	C_{0,\mathrm{sym,Tr}}^{\infty}(\Omega\times I)&:=\left\{\Psi\in C_{0}^{\infty}(\Omega\times I)^{3\times3}:\Psi=\Psi^{\top},\Tr(\Psi)=0\text{ in }\Omega\times I\right\},
\end{align}
where $I\subseteq[0,\infty)$ is an interval. The corresponding Lebesgue space of space-integrable functions on $\Omega$ are defined by
\begin{align}
	L_{\mathrm{div}}^{2}(\Omega)&:=\left\{v\in L^{2}(\Omega)^{3}:\divge{v} =0\text{ in }\Omega\right\},\\
	L_{\mathrm{sym,Tr}}^{2}(\Omega)&:=\left\{S\in L^{2}(\Omega)^{3\times 3}:S=S^{\top},\Tr(S)=0\text{ in }\Omega\right\}.
\end{align}
The Sobolev spaces obtained as the closure of $C_{0,\mathrm{div}}^{\infty}(\Omega)$ and $C_{0,\mathrm{sym,Tr}}^{\infty}(\Omega)$ with respect to the $H^{1}(\Omega)$-norm are denoted by
\begin{align}
	H_{0,\mathrm{div}}^{1}(\Omega)&:=\left\{v\in H_{0}^{1}(\Omega)^{3}: \divge{v}=0\text{ in }\Omega\right\},
    \\
    H_{\mathrm{sym,Tr}}^{1}(\Omega)&:=\left\{S\in H^{1}(\Omega)^{3\times 3}:S=S^{\top},\Tr(S)=0\text{ in }\Omega\right\}.
\end{align}
\end{subequations}
Notice that all the boundary conditions are identified in the trace sense. \\
\subsection{General assumptions and further notations}
In the following, we collect and discuss the mathematical assumptions on the domain, the given data, and the material parameters.
\begin{assump}[on the domain]\label{assumption on Omega}
    We assume that $\Omega\subseteq\mathbb{R}^3$ is a bounded domain with $C^{2}$-boundary $\partial\Omega$ and write $\vec{n}$ as the outward unit normal vector.
\end{assump}

Also in view of~\eqref{Energy smooth solution}, we make the following hypothesis for the non-smooth dissipation potential $\mathcal{P}$ in~\eqref{fluid equation stress}:
\begin{assump}[on the plastic potential]\label{assumption on the dissipation potential}
	For the plastic potential 
    \begin{equation}\label{integral form of dissipation potential}
    \begin{aligned}
        \mathcal{P}:L^{2}(\Omega)\times L_{\mathrm{sym,Tr}}^{2}(\Omega)&\to[0,\infty]\,\\
        (\varphi;S)&\mapsto\int_{\Omega}P(x,\varphi(x),S(x))\dd{x},
    \end{aligned}
    \end{equation}
    we make the following assumptions: The density $P:\Omega\times\mathbb{R}\times \mathbb{R}_{\mathrm{sym,Tr}}^{3\times 3}\to[0,\infty]$ is proper and measurable with ${P}(x,\varphi,0)=0$ for all $x \in \Omega$ and $\varphi\in\mathbb{R}$. Moreover, 
 \begin{itemize}
    \item for all $x\in\Omega$, the mapping $(y,z)\mapsto P(x,y,z)$ is lower semicontinuous,
     \item for all $(x,y) \in \Omega\times\mathbb{R}$, the mapping $z \mapsto {P}(x,y,z) $ is convex,
     \item for all $(x,z)\in \Omega\times\mathbb{R}_{\mathrm{sym,Tr}}^{3\times 3}$, the mapping $y \mapsto {P}(x,y,z) $ is continuous.
 \end{itemize}
 Besides, the convex partial subdifferential of $\mathcal{P}(\varphi;\cdot)$ at point $S$ is given by
\begin{equation}\label{definition of subdifferential}
	\partial\mathcal{P}(\varphi;S):=\left\{ \xi\in (L_{\mathrm{sym,Tr}}^{2}(\Omega))^{\prime}:\langle \xi,\tilde{S}-S\rangle_{L^{2}(\Omega)}+\mathcal{P}(\varphi;S)\leq\mathcal{P}(\varphi;\tilde{S})\mbox{ for all }\tilde{S}\in L_{\mathrm{sym,Tr}}^{2}(\Omega) \right\}.
\end{equation}
Notice that, by definition, $\partial\mathcal{P}(\varphi;S)=\emptyset$ if $\mathcal{P}(\varphi;S)=\infty$.

\end{assump}

Furthermore, we make the following hypotheses for the initial data and the external loading:
\begin{assump}[on the given data]\label{assumptions on the initial data}
	Assume that $v_{0}\in L_{\mathrm{div}}^{2}(\Omega)$, $S_{0}\in L_{\mathrm{sym,Tr}}^{2}(\Omega)$ and $f\in L_{\mathrm{loc}}^{2}([0,T);H^{-1}(\Omega)^{3})$. Moreover, assume that $\varphi_{0}\in H^{1}(\Omega)$ with $\abs{\varphi_{0}}\leq1$ almost everywhere in $\Omega$ and
	\begin{align*}
		\frac{1}{\abs{\Omega}}\int_{\Omega}\varphi_{0}\dd{x}\in(-1,1).
	\end{align*}
\end{assump}

In order to guarantee the existence of weak solutions, we make the following assumptions on the coefficients and the singular potential. The idea of these assumptions is inspired directly by \cite[][Assumption 3.1]{abelsExistenceWeakSolutions2013} and \cite[][Assumption 1.1]{abelsConvergenceEquilibriumCahn2007}.

\begin{assump}[on the material parameters]\label{TP assumption}
    The dependence of the material parameters on the composition of the mixture, i.e., on $\varphi$ is assumed to be as follows:\\
	(1) The dependence of the mass density $\rho$ on the phase-field variable $\varphi$ is given by 
	\begin{align}\label{assumption on density}
		\rho(\varphi)=\frac{\rho_{1}+\rho_{2}}{2}+\frac{\rho_{2}-\rho_{1}}{2}\varphi
	\end{align}
	where the constant $\rho_{i}>0$ is the mass density of fluid $i$, $i=1,2$.
    \\
	(2) For the viscosity parameter $\nu$ and the elastic modulus $\eta$, we assume that
    $\nu\in C^{0}(\mathbb{R})$, $\eta\in C^{1}(\mathbb{R})$ and 
	\begin{align}\label{assumption on coefficients}
		0<\nu_{1}\leq\nu(\varphi)\leq \nu_{2} \quad 0\leq\eta_{1}\leq\eta(\varphi)\leq\eta_{2},\text{ and }\abs{\eta^{\prime}(\varphi)}\leq C\text{ for all }\varphi\in\mathbb{R}
	\end{align}
	for some positive constants $C,\eta_{1},\eta_{2},\nu_{1},\nu_{2}$. Herein, the constants $\eta_{i},\nu_{i}$ can be viewed as the constants associated with the pure fluid $i$, $i=1,2$.
\end{assump}
\begin{assump}[on the singular potential]
    For the singular potential $W$, we assume that $W\in C([-1,1])\cap C^{3}(-1,1)$ such that
	\begin{align}\label{assumption on double well potential W}
		\lim\limits_{\varphi\to-1}W^{\prime}(\varphi)=-\infty,\,\lim\limits_{\varphi\to1}W^{\prime}(\varphi)=\infty,\, W^{\prime}(0)=0,\,W^{\prime\prime}(\varphi)\geq-\kappa\mbox{ for some }\kappa>0.
	\end{align} 
	Moreover, we extend $W(\varphi)$ to $+\infty$ for $\varphi\in\mathbb{R}\backslash[-1,1]$. Without loss of generality, we also assume that $W(\varphi)\geq0$ for all $\varphi\in[-1,1]$.
\end{assump}
\begin{remark}
	(1) As an example for such a singular potential, one can consider
	\begin{align*}
		W(\varphi)=\frac{1}{2}\left((1+\varphi)\ln(1+\varphi)+(1-\varphi)\ln(1-\varphi)\right)-\frac{\lambda}{2}\varphi^{2},\,\varphi\in[-1,1],
	\end{align*}
    for fixed $\lambda\geq0$.\\
	(2) Since we will show that a solution $\varphi$ takes values in $(-1,1)$, we actually only need the functions $\nu,\eta$ to be defined on this interval and then extend them outside of this interval in a sufficiently smooth way by constants. So, by the assumptions of continuity and continuous differentiability, respectively, the bounds \eqref{assumption on coefficients} are natural.
\end{remark}
\begin{subequations}\label{energies functionals}
\textbf{Energy functionals.} We define the kinetic energy, the elastic energy and the phase-field energy as follows: 
\begin{align}
    &\mathcal{E}_{\mathrm{kin}}:L^{\infty}(\Omega)\times L_{\mathrm{div}}^{2}(\Omega)\to[0,\infty],\,(\varphi,v)\mapsto\int_{\Omega}\rho(\varphi)\frac{\abs{v}^2}{2}\dd{x},
    \\
    &\mathcal{E}_{\mathrm{el}}:L_{\mathrm{sym,Tr}}^{2}(\Omega)\to[0,\infty],\,
    S\mapsto\int_{\Omega}\frac{\abs{S}^2}{2}\dd{x},
    \\
    &\mathcal{E}_{\mathrm{pf}}:H^{1}(\Omega)\to[0,\infty],\,
	\varphi\mapsto\int_{\Omega}\varepsilon\frac{\abs{\nabla\varphi}^2}{2}+\varepsilon^{-1}W(\varphi)\dd{x}.
\end{align}
Then, the total energy of this system is given by
\begin{equation}
	\mathcal{E}_{\mathrm{tot}}(v,S,\varphi)
	:=\mathcal{E}_{\mathrm{kin}}(\varphi,v)+\mathcal{E}_{\mathrm{el}}(S)+\mathcal{E}_{\mathrm{pf}}(\varphi).
	\label{total energy two phase system}
\end{equation}
\end{subequations}
\textbf{Dissipation functionals:} We define the viscous Stokes-type dissipation, the quadratic stress diffusion, and the quadratic diffusion of the chemical potential according to the Cahn--Hilliard equation as follows:
\begin{subequations}\label{dissipation functionals}
\begin{align}
        &\mathcal{D}_{\mathrm{s}}:L^{1}(\Omega)\times H_{0,\mathrm{div}}^{1}(\Omega)\to [0,\infty],\, 
        (\varphi,v)\mapsto\int_{\Omega}2\nu(\varphi)\abs{\sym{\nabla v}}^{2}\dd{x},
        \\
        &\mathcal{D}_{\mathrm{ch}}:H^{1}(\Omega)\to[0,\infty],\,
        \mu\mapsto\int_{\Omega}\abs{\nabla\mu}^{2}\dd{x},
        \\
        &\mathcal{D}_{\mathrm{sd},\gamma}:H_{\mathrm{sym,Tr}}^{1}(\Omega)\to [0,\infty],\,
        S\mapsto\int_{\Omega}\gamma\abs{\nabla S}^{2}\dd{x}.
\end{align}
Then, the sum of Stokes-type dissipation and Cahn--Hilliard dissipation and the total dissipation potential are given by
\begin{align}
    \mathcal{D}_{\mathrm{chs}}&:=\mathcal{D}_{s}(\varphi,v)+\mathcal{D}_{\mathrm{ch}}(\mu),
    \label{dissipation stokes and chemical}\\
    \mathcal{D}_{\mathrm{tot}}(v,S,\varphi,\mu)&:=\mathcal{D}_{\mathrm{chs}}(v,\varphi,\mu)+\mathcal{D}_{\mathrm{sd},\gamma}(S)+\mathcal{P}(\varphi;S).    
    \label{total dissipation}
\end{align}
Note that~\eqref{total dissipation} also involves the plastic potential $\mathcal{P}$, which is specified in more details in~\eqref{integral form of dissipation potential}.
\end{subequations}

\subsection{Reformulation of system~\eqref{two phase system}}
Now, we reformulate the Korteweg stress that appears on the right-hand side of~\eqref{fluid equation momentum} with the aid of~\eqref{second CHE}. Formally, for a smooth solution $\varphi$, by substituting~\eqref{second CHE} in~\eqref{fluid equation momentum}, we find
\begin{equation}
\begin{aligned}
	-\varepsilon\divge{\nabla\varphi\otimes\nabla\varphi}&=-\varepsilon\Delta\varphi\cdot\nabla\varphi-\varepsilon\nabla\left(\frac{\abs{\nabla\varphi}^2}{2}\right)
	\\
	&=\mu\nabla\varphi-\varepsilon^{-1}W^{\prime}(\varphi)\nabla\varphi-\varepsilon\nabla\left(\frac{\abs{\nabla\varphi}^2}{2}\right)
	\\
	&=\mu\nabla\varphi-\varepsilon^{-1}\nabla(W(\varphi))-\varepsilon\nabla\left(\frac{\abs{\nabla\varphi}^2}{2}\right).
\end{aligned}
\end{equation}
This allows us to define a new pressure term 
\begin{align}\label{new pressure term}
	g=p+\varepsilon^{-1}W(\varphi)+\varepsilon\frac{1}{2}\abs{\nabla\varphi}^{2}
\end{align}
to rewrite \eqref{fluid equation momentum} as
\begin{align}\label{momentum balbace for two phase problem alternative}
	\partial_{t} (\rho v)+\divge{(v\otimes(\rho v+J))}-\divge{(\eta(\varphi)S+2\nu(\varphi)\sym{\nabla v})}+\nabla g=\mu\nabla\varphi.
\end{align}
Hence, when testing \eqref{momentum balbace for two phase problem alternative} with divergence-free vectors, the additional pressure term also vanish. Notice that the singular potential is not convex. But thanks to the assumption that $W^{\prime\prime}$ is bounded from below by $-\kappa$ with $\kappa>0$ cf. \eqref{assumption on double well potential W}, we define
\begin{align}\label{convex double well potential}
	W_{\kappa}(\varphi):=W(\varphi)+\frac{\kappa}{2}\varphi^{2}.
\end{align}
Thanks to the bound $W^{\prime\prime}(\varphi)\geq-\kappa$, we find that $W_{\kappa}^{\prime\prime}(\varphi)=W^{\prime\prime}(\varphi)+\kappa\geq0$, for all $\varphi\in[-1,1]$ which ensures that $W_{\kappa}$ is convex. In particular, it holds that $W_{\kappa}^{\prime}(\varphi)=W^{\prime}(\varphi)+\kappa \varphi$. Using $W_{\kappa}$ in \eqref{second CHE}, this is equivalent to 
\begin{equation}\label{chemical potential with convex double well potential}
	\mu+\frac{\kappa}{\varepsilon}\varphi=\varepsilon^{-1}W_{\kappa}^{\prime}(\varphi)-\varepsilon\Delta\varphi.
\end{equation}

Now, we can consider the energy $\mathcal{E}_{\mathrm{pf},\kappa}:L^2(\Omega)\to\mathbb{R}$ with its domain given by
\begin{subequations}
\begin{align}
	\dom{\mathcal{E}_{\mathrm{pf},\kappa}}&:=\{\varphi\in H^1(\Omega):-1\leq \varphi\leq1\mbox{ a.e.}\},
	\label{domain of convex phase-field energy}\\
	\mathcal{E}_{\mathrm{pf},\kappa}(\varphi)&:=\begin{cases}
		\frac{1}{2}\int_{\Omega}\varepsilon\abs{\nabla \varphi}^{2}\dd{x}+\int_{\Omega}\varepsilon^{-1}W_{\kappa}(\varphi)\dd{x}&\mbox{ for }u\in\dom{\mathcal{E}_{\mathrm{pf},\kappa}},\\
		+\infty&\mbox{ otherwise}.
	\end{cases}
	\label{convex phase-field energy}
\end{align}
\end{subequations}
According to~\cite[][Theorem 4.3]{abelsConvergenceEquilibriumCahn2007}, the domain of the subdifferential is given by 
\begin{equation}
	\dom{\partial \mathcal{E}_{\mathrm{pf},\kappa}}=\left\{\varphi\in H^{2}(\Omega):W_{\kappa}^{\prime}(\varphi)\in L^{2}(\Omega),W_{\kappa}^{\prime\prime}(\varphi)\abs{\nabla \varphi}^{2}\in L^{1}(\Omega),\vec{n}\cdot\nabla \varphi|_{\partial\Omega}=0\right\}
\end{equation}
as well as 
\begin{equation}\label{subdifferential of phase-field energy}
    \partial \mathcal{E}_{\mathrm{pf},\kappa}(\varphi)=\left\{-\varepsilon\Delta \varphi+\varepsilon^{-1}W_{\kappa}^{\prime}(\varphi)\right\}\text{ for all }\varphi\in\dom{\partial \mathcal{E}_{\mathrm{pf},\kappa}}.
\end{equation}
In this case, since $\partial \mathcal{E}_{\mathrm{pf},\kappa}$ is single-valued, it coincides with the G\^{a}teaux derivative $\mathrm{D}\mathcal{E}_{\mathrm{pf},\kappa}(\varphi)=-\varepsilon\Delta\varphi+\varepsilon^{-1}W_{\kappa}^{\prime}(\varphi)$. Additionally, we have the following estimate
\begin{equation}
	\lVert \varphi\rVert_{H^{2}}^{2}
	+\lVert W_{\kappa}^{\prime}(\varphi)\rVert_{L^{2}}^{2}
	+\int_{\Omega}W_{\kappa}^{\prime\prime}(\varphi)\abs{\nabla \varphi}^{2} \dd{x}
	\leq C\left(\lVert \mathrm{D} \mathcal{E}_{\mathrm{pf},\kappa}(\varphi)\rVert_{L^{2}}^{2}
	+\lVert \varphi\rVert_{L^{2}}^{2}+1\right)
	\label{estimate on element in subdifferntial domain}
\end{equation}
for all $\varphi\in\dom{\mathrm{D} \mathcal{E}_{\mathrm{pf},\kappa}}$ as well as 
\begin{equation}
	\mu+\frac{\kappa}{\varepsilon}\varphi
    =-\varepsilon\Delta\varphi+\varepsilon^{-1}W_{\kappa}^{\prime}(\varphi)
    =\mathrm{D} \mathcal{E}_{\mathrm{pf},\kappa}(\varphi)
	\label{reformulation of second CH}
\end{equation}
for all $\varphi\in\dom{\mathrm{D} \mathcal{E}_{\mathrm{pf},\kappa}}$. Thus, we have the following relation between the original phase-field energy $\mathcal{E}_{\mathrm{pf}}$ from~\eqref{energies functionals} and the convexified phase-field energy $\mathcal{E}_{\mathrm{pf},\kappa}$ from~\eqref{convex phase-field energy}
\begin{equation}
	\mathcal{E}_{\mathrm{pf}}(\varphi)
	=\mathcal{E}_{\mathrm{pf},\kappa}(\varphi)-\frac{\kappa}{2\varepsilon}\lVert\varphi\rVert_{L^{2}}^{2}.
\end{equation}

Therefore, after this reformulation, system~\eqref{two phase system} can be written in the following formally equivalent form:\begin{subequations}
	\label{two phase system alternative}
	\begin{align}
		\partial_{t} (\rho v)+\divge{v\otimes(\rho v+J)}-\divge{\mathbb{T}}& =f+\mu\nabla\varphi&\mbox{ in }\Omega\times(0,T),
		\label{two phase system alternative velocity}\\
		\divge{v}&=0&\mbox{ in }\Omega\times(0,T),
		\label{two phase system alternative incompressible}\\
        \mathbb{T}&=\eta(\varphi)+2\nu(\varphi)\sym{\nabla v}-g\mathbb{I}&\mbox{ in }\Omega\times(0,T),
        \label{two phase system alternative cauchy stress tensor}\\
		\overset{\triangledown}{S}+\partial {P}(\varphi;S)-\gamma\Delta S&\ni\eta(\varphi)\sym{\nabla v}&\mbox{ in }\Omega\times(0,T),
		\label{two phase system alternative stress}\\
		\partial_{t}\varphi+v\cdot\nabla\varphi&=\Delta\mu&\mbox{ in }\Omega\times(0,T),
		\label{two phase system alternative phase field}\\
		\mu+{\kappa}{\varepsilon^{-1}}\varphi&=-\varepsilon\Delta\varphi+\varepsilon^{-1}W_{\kappa}^{\prime}(\varphi)&\mbox{ in }\Omega\times(0,T),
		\label{two phase system alternative chemical potential}\\	
		v|_{\partial\Omega}&=0&\mbox{ on }\partial\Omega\times(0,T),
        \\
        \gamma \vec{n}\cdot\nabla S|_{\partial\Omega}&=0&\mbox{ on }\partial\Omega\times(0,T),
		\label{two phase system alternative homogeneous boundary}\\
		\vec{n}\cdot\nabla\varphi|_{\partial\Omega}=\vec{n}\cdot\nabla\mu|_{\partial\Omega}&=0&\mbox{ on }\partial\Omega\times(0,T),
		\label{two phase system alternative neumann boundary}\\
		(v,S,\varphi)|_{t=0}&=(v_{0},S_{0},\varphi_{0})&\mbox{ in }\Omega.
		\label{two phase system alternative initial}
	\end{align}
\end{subequations}

\subsection{Weak solutions of system~\eqref{two phase system} with respect to~\eqref{two phase system alternative} for $\gamma>0$}
Observe that the partial subdifferential $\partial \mathcal{P}(\varphi;S)$ may be multi-valued for non-smooth potentials $\mathcal{P}$. By using the definition of the partial subdifferential, we can introduce the variational inequality
\begin{equation}\label{subdifferential inequality}
    \langle \xi,\tilde{S}-S\rangle_{L_{\mathrm{sym,Tr}}^{2}(\Omega)}\leq\mathcal{P}(\varphi;\tilde{S})-\mathcal{P}(\varphi;S),
\end{equation}
that is equivalent to all $\xi\in\partial\mathcal{P}(\varphi;S)$. 

Accordingly, a weak solution of system~\eqref{two phase system alternative} is defined by a weak formulation and an energy estimate of equation~\eqref{two phase system alternative velocity}-\eqref{two phase system alternative cauchy stress tensor}, an evolutionary variational inequality of equation~\eqref{two phase system alternative stress}, and a weak formulation and an energy estimate of equation~\eqref{two phase system alternative phase field}-\eqref{two phase system alternative chemical potential}.

\begin{defi}[Weak solutions for system~\eqref{two phase system alternative}]\label{gerneralized solution TPP}
\begin{subequations}
    Let $\gamma>0$. Let the assumptions~\ref{assumption on Omega}-\ref{TP assumption} hold true. A quadruplet $(v,S,\varphi,\mu)$ is called a weak solution of the two-phase system \eqref{two phase system alternative} if the following properties are satisfied:\\
    1. The quadruplet $(v,S,\varphi,\mu)$ has the following regularity:
	\begin{align*}
		&v\in L_{\mathrm{loc}}^{\infty}([0,T);L_{\mathrm{div}}^{2}(\Omega))\cap L_{\mathrm{loc}}^{2}([0,T);H_{0,\mathrm{div}}^{1}(\Omega)),\\
		&S\in L_{\mathrm{loc}}^{\infty}([0,T);L_{\mathrm{sym,Tr}}^{2}(\Omega))\cap L_{\mathrm{loc}}^{2}([0,T);H_{\mathrm{sym,Tr}}^{1}(\Omega)),\\
		&\varphi\in L_{\mathrm{loc}}^{\infty}([0,T);H^{1}(\Omega))\cap L_{\mathrm{loc}}^{2}([0,T);H^{2}(\Omega)),W^{\prime}(\varphi)\in L_{\mathrm{loc}}^{2}([0,T);L^{2}(\Omega)),\\
		&\mu\in L_{\mathrm{loc}}^{2}([0,T);H^{1}(\Omega))
	\end{align*}
    2. The quadruplet $(v,S,\varphi,\mu)$ satisfies the following weak formulations and energy estimates:\\
    2.1. Weak formulation of the momentum balance:
	\begin{equation}\label{weak velocity TPP}
	    \begin{aligned}
		&\int_{0}^{T}\int_{\Omega}-\rho v\cdot\partial_{t}\Phi-(\rho v\otimes v):\nabla\Phi+2\nu(\varphi)\sym{\nabla v}:\sym{\nabla \Phi}+\eta(\varphi)S:\sym{\nabla\Phi} \dd{x}\dd{t}
		\\
		-&\int_{0}^{T}\int_{\Omega}(v\otimes J):\nabla\Phi+\mu\nabla\varphi\cdot\Phi \dd{x} \dd{t}
		=\int_{\Omega}\rho_{0}v_{0}\cdot\Phi(\cdot,0) \dd{x}
		+\int_{0}^{T}\langle f,\Phi\rangle_{H^{1}}\dd{t}
	\end{aligned}
	\end{equation}
	for all $\Phi\in C_{0,\mathrm{div}}^{\infty}(\Omega\times[0,T))$, and the partial energy inequality
    \begin{equation}\label{partial energy inequality kinetic}
    \begin{aligned}
        &\mathcal{E}_{\mathrm{kin}}(\varphi(t),v(t))
        +\int_{0}^{t}\int_{\Omega}2\nu(\varphi)\abs{\sym{\nabla v}}^{2}\dd{x}\dd{s}
        +\int_{0}^{t}\int_{\Omega}\eta(\varphi)S:\nabla v \dd{x}\dd{s}
    \\
        \leq&\mathcal{E}_{\mathrm{kin}}(\varphi_{0},v_{0})
        +\int_{0}^{t}\langle f, v\rangle_{H^{1}}\dd{s}
        +\int_{0}^{t}\int_{\Omega}\mu(\nabla\varphi\cdot v)\dd{x}\dd{s},
    \end{aligned}
    \end{equation}
    for almost all $t\in(0,T)$.\\
    2.2. The evolutionary variational inequality for the stress
    \begin{align}
		&\frac{1}{2}\lVert S(t)-\tilde{S}(t)\rVert_{L^2}^{2}
		+\int_{0}^{t}\int_{\Omega} \partial_{t}\tilde{S}:(S-\tilde{S}) - v\cdot\nabla S:\tilde{S} - (S\skw{\nabla v}-\skw{\nabla v}S):\tilde{S}\dd{x} \dd{s}
    \nonumber\\
		+&\int_{0}^{t}\mathcal{P}(\varphi;S)-\mathcal{P}(\varphi;\tilde{S}) \dd{s}
		+\int_{0}^{t}\int_{\Omega}\gamma\nabla S:\nabla(S-\tilde{S}) - \eta(\varphi)\sym{\nabla v}:(S-\tilde{S})\dd{x} \dd{s} 
    \nonumber\\
		\leq&\frac{1}{2}\lVert S_{0}-\tilde{S}(0)\rVert_{L^2}^{2}
    \label{generalized solution stress energy estimate TPP}
    \end{align}
    for all $\tilde{S}\in C_{0,\mathrm{sym,Tr}}^{\infty}(\Omega\times[0,T))\cap\dom{\mathcal{P}(\varphi;\cdot)}$, and a.e. $t\in(0,T)$.\\
    2.3. Weak formulation of the phase-field evolution law:
    \begin{equation}
		\int_{0}^{T}\int_{\Omega}-\varphi\cdot\partial_{t}\zeta+(v\cdot\nabla\varphi)\zeta\dd{x} \dd{t}
		=\int_{\Omega}\varphi_{0}\cdot\zeta(\cdot,0) \dd{x}
		-\int_{0}^{T}\int_{\Omega}\nabla\mu:\nabla\zeta \dd{x} \dd{t} 
		\label{weak first CH TPP}
    \end{equation}
    for all $\zeta\in C_{0}^{\infty}([0,T);C^{1}(\bar{\Omega}))$ as well as 
    \begin{equation}
		\mu=\varepsilon^{-1}W^{\prime}(\varphi)-\varepsilon\Delta\varphi 
		\label{weak second CH TPP}	
    \end{equation}
    almost everywhere in $\Omega\times(0,T)$, and the partial energy inequality
    \begin{equation}
        \mathcal{E}_{\mathrm{pf}}(\varphi(t))
        +\int_{0}^{t}\int_{\Omega}\abs{\nabla\mu}^{2}\dd{x}\dd{s}
        +\int_{0}^{t}\int_{\Omega}\mu(\nabla \varphi\cdot v)\dd{x}\dd{s}
        \leq\mathcal{E}_{\mathrm{pf}}(\varphi_{0}),
        \label{partial energy inequality phase field}
    \end{equation}
    for almost all $t\in(0,T)$.
	\end{subequations}
\end{defi}
\begin{remark}
    Notice that the evolutionary variational inequality~\eqref{generalized solution stress energy estimate TPP}, is defined by testing~\eqref{two phase system alternative stress} with $S-\tilde{S}$, integrating over $\Omega\times(0,t)$, using~\eqref{subdifferential inequality} together with an integration by parts in space and time.
    
    Further notice that, by choosing $\tilde{S}\equiv0$ in \eqref{generalized solution stress energy estimate TPP} and summing with \eqref{partial energy inequality kinetic} and \eqref{partial energy inequality phase field}, a weak solution $(v,S,\varphi,\mu)$ satisfies the following total energy-dissipation estimate: 
    \begin{equation}
    \begin{aligned}
	 	&\mathcal{E}_{\mathrm{tot}}(v(t),S(t),\varphi(t))
	 	+\int_{0}^{t}\mathcal{D}_{\mathrm{tot}}(v,S,\varphi,\mu) 
	 	\dd{\tau}
	 	\leq \mathcal{E}_{\mathrm{tot}}(v_{0},S_{0},\varphi_{0})
            +\int_{0}^{t}\langle f, v\rangle_{H^{1}}\dd{\tau}
	 	\label{total energy estimate for generalized solution for TPP}
    \end{aligned} 
    \end{equation}
	 for almost all $t\in[0,T)$.
\end{remark}

	\section{Dissipative solutions}\label{section dissipative solutions}
In this section, we introduce the notion of dissipative solutions for $\gamma\geq0$ and investigate their relation to the weak solutions from Definition~\ref{gerneralized solution TPP} in the case $\gamma>0$. 
\subsection{General concept of dissipative solutions}
\begin{subequations}
The notion of dissipative solutions is based on a relative energy inequality. In order to better explain the concept, we follow the general approach proposed in~\cite{AgLaRo2024}. For this, we consider two reflexive Banach spaces $\mathbb V$ and $\mathbb Y$ with dual spaces $\mathbb{V}^{\prime}$ and $\mathbb{Y}^{\prime}$ such that $ \mathbb Y \subseteq \mathbb V\subseteq \mathbb{V}^{\prime}\subseteq\mathbb Y^{\prime}$ and a general evolutionary PDE on time interval $(0,T)$ of the form
\begin{equation}
\partial_ t \pmb U(t) + A(\pmb U(t)) = \pmb 0 \quad \text{ in }\mathbb Y^\prime \text{ with }\quad\pmb U(0)=\pmb U_0 \quad \text{in }\mathbb V.\,\label{introGEnEq}
\end{equation}
Here, $A\,:\, \mathbb V\to \mathbb Y^\prime$ denotes a differential operator and $\pmb U_0\in \mathbb V$ the initial datum. Let $\mathcal{E}:\mathbb{V}\to[0,\infty]$ and $\Psi:\mathbb{V}\to[0,\infty]$ be energy functional which is assumed to be twice G\^ateaux differentiable and dissipation functionals associated with~\eqref{introGEnEq}. Furthermore, we introduce a space of test functions $\mathfrak{T}$ such that $\mathrm{D}\mathcal{E}(\tU(t)),\mathcal{A}(\tU(t))\in \mathbb{Y}$ for a.e. $t\in(0,T)$ and all $\tU\in\mathfrak{T}$. For any initial value $\pmb U_0\in \mathbb V$, a sufficiently regular solution $\pmb U\in\mathbb{V}$ with the property $\mathrm{D}\mathcal{E}(\pmb U(t))\in\mathbb{Y}$ for all $t\in[0,T]$ of~\eqref{introGEnEq}  formally fulfills the energy-dissipation mechanism 
\begin{equation}
\mathcal E( \pmb U) \Big|_0^t +\int_0^t \Psi (\pmb U)   \dd{\tau} \leq  0 \label{intreoEnGen}
\end{equation}
for all $ t\in [0,T]$. This is obtained by testing~\eqref{introGEnEq} by $\mathrm{D}\mathcal{E}(\pmb U(t))\in\mathbb{Y}$, and it also implies that 
\begin{equation}\label{intro def of dissipation}
    \Psi(\pmb U):=\langle A(\pmb U),\mathrm{D}\mathcal{E}(\pmb U)\rangle_{\mathbb{Y}}.
\end{equation}

Moreover, we introduce a so-called regularity weight
\begin{equation}\label{regularity weight}
    \mathcal{K}:\mathfrak{T}\to[0,\infty]\text{ with }\mathcal{K}(0)=0,    
\end{equation}
which is to be chosen such that both sides of the relative energy inequality remain finite.

In addition, we define the system operator $\mathcal{A}$ directly following~\eqref{introGEnEq} as 
\begin{equation}\label{introSysoper}
    \begin{aligned}
        \mathcal{A}:\mathfrak{T}&\to \mathbb{V}^{\prime},\,\mathcal{A}(\tU):=\partial_{t}\tU+A(\tU).
    \end{aligned}
\end{equation}

We refer to~\cite{AgLaRo2024} for the full list of assumptions. A major assumption in~\cite{AgLaRo2024} is that the energy $\mathcal E$ is convex, which gives the positivity of the first-order Taylor expansion of the energy functional. Following the assertions in~\cite[Prop.~3.6]{AgLaRo2024}, we define the relative energy~$\mathcal{R}$ and the relative dissipation~$\mathcal{W}^{(\mathcal{K})}$ as the first-order Taylor expansion of the energy $\mathcal E$ and the term 
$ \Psi(\pmb U) - \langle A(\pmb U), \mathrm{D} \mathcal E (\tU) \rangle_{\mathbb{Y}} + \mathcal{K}(\tU) \mathcal E(\pmb U)$  given as    
\begin{equation}\label{relen}
\mathcal{R}(\pmb U| \tU) := \mathcal{E}(\pmb U) - \mathcal E(\tU) - \langle \mathrm{D}\mathcal E( \tU) , \pmb U - \tU\rangle_{{\mathbb{V}}} 
\end{equation}
as well as
\begin{equation}\label{reldis}
\begin{aligned}
    \mathcal{W}^{(\mathcal{K})}(\pmb U| \tU) &: =
   \Psi(\pmb U) - \langle A(\pmb U) , \mathrm{D}\mathcal E (\tU)\rangle_{{\mathbb{Y}}}  
   - \langle A(\tU) ,\mathrm{D}^2\mathcal E(\tU)(\pmb U - \tU) \rangle_{{\mathbb{Y}}} 
   + \mathcal{K}(\tU)\mathcal{R}(\pmb U| \tU) 
   \,,   
\end{aligned}
\end{equation}
where we use~\eqref{intro def of dissipation}. Therefore, we can formally compute
\begin{equation*}
\begin{aligned}
    \mathcal R(\pmb U | \tU) \Big|_0^t 
    + \int_0^t\Psi(\pmb U) \dd{\tau} &{}\leq 
    - \int_0^t \langle \partial_t \pmb U, \mathrm{D}\mathcal{E}(\tU)\rangle_{\mathbb{Y}} 
    + \langle  \pmb U- \tU, \mathrm{D}^2\mathcal E(\tU)\partial_t \tU\rangle_{\mathbb{Y}}  \dd{\tau} 
    \\
   &{} = \int_0^t \langle A(\pmb U), \mathrm{D}\mathcal{E}(\tU)\rangle_{\mathbb{Y}} + \langle \pmb U- \tU,\mathrm{D}^2\mathcal E(\tU)A(\tU)  \rangle_{\mathbb{Y}}  \dd{\tau} \\
   &\quad -\int_0^t
   \langle \mathrm{D}^2\mathcal E (\tU)(\pmb U-\tU),\mathcal{A}(\tU)\rangle_{\mathbb{V}}\dd{\tau}
\end{aligned}
\end{equation*}
By rearranging terms on both sides and using~\eqref{introSysoper} and~\eqref{reldis}, we infer 
\begin{equation}
    \mathcal R(\pmb U | \tU) \Big|_0^t 
    + \int_0^t \mathcal{W}^{(\mathcal{K})}(\pmb U| \tU) +\langle \mathrm{D}^2\mathcal E (\tU)(\pmb U-\tU),\mathcal{A}(\tU)\rangle_{\mathbb{V}} \dd{\tau} \leq \int_0^t \mathcal{K}(\tU) \mathcal R( \pmb U|\tU) \dd{\tau}\,.
\end{equation}
Applying Gronwall's inequality provides the relative energy inequality in this general setting as 
\begin{equation}\label{relenineq}
    \begin{aligned}
     \mathcal{R}(\pmb U(t) | \tU(t)) + \int_0^t \mathrm \exp{\left(\int_s^t\mathcal K(\tU) \dd{\tau} \right)}\left[ \mathcal{W}^{(\mathcal{K})}(\pmb U| \tU ) + \langle   \mathrm{D}^2\mathcal E (\tU)(\pmb U-\tU),\mathcal{A}(\tU) \rangle_{\mathbb{V}} \right] \dd{s}  
 \\  \leq \mathcal{R}(\pmb U_0 | \tU(0))\mathrm \exp{\left(\int_0^t\mathcal{K}(\tU)\dd{s} \right) }   \,
\end{aligned}
\end{equation}   
for a.e. $t\in (0,T)$ and smooth test functions $\tU\in\mathfrak{T}$. 
\begin{defi}[Dissipative solutions for general system~\eqref{introGEnEq}]\label{dissipative solution for general system}
    A function $\pmb U:(0,T)\to\mathbb{V}$ is called a dissipative solution for system~\eqref{introGEnEq}, if $\pmb U$ satisfies~\eqref{relenineq} for all sufficiently regular test functions $\tU\in\mathfrak{T}$ and for a.e. $t\in(0,T)$.
\end{defi}
\end{subequations} 

\subsection{Dissipative solutions for system~\eqref{two phase system alternative}}
A key ingredient for the notion of dissipative solutions in~\cite{AgLaRo2024} is the convexity of the system's energy $\mathcal{E}$. This means for our system~\eqref{two phase system alternative} that the total energy should be convex. But note that the kinetic energy $\mathcal{E}_{\mathrm{kin}}$, from~\eqref{energies functionals} as a function of $(\varphi,v)$ is not convex. In order to overcome the non-convexity in the kinetic energy~$\mathcal E_{\mathrm{kin}}$, we use transformation of variables. Expressing the energy in terms of the mass density $\rho$ and the momentum $p$ with $p=\rho v $ instead of mass density $\rho$ and velocity $v$ makes the kinetic energy convex, {i.e.,}
\begin{align}
\label{Ekin-tilde}
    \tilde{\mathcal{E}}_{\mathrm{kin}}(\rho,p):= \int_\Omega \tilde{\mathcal{E}}_{\mathrm{kin}}(\rho , p) \mathrm d x \text{ with }\tilde{\mathcal{E}}_{\mathrm{kin}}(\rho, p) :=   \begin{cases}
        \frac{p^{2}}{2\rho}&\text{ if }\rho>0, \\
        0&\mbox{ if }\rho=0 \text{ and }p = 0 ,\\
        \infty &\text{ otherwise. }
    \end{cases}
\end{align}
Accordingly, we also denote the the total energy that depends on the momentum as a variable by $\tilde{\mathcal{E}}_{\mathrm{tot},\kappa}$. 
Using the general definitions of the relative energy~\eqref{relen}, we find the following expression for the relative kinetic energy
\begin{equation}\label{relative kinetic energy}
    \begin{aligned}        
    \tilde{\mathcal{R}}_{\mathrm{kin}}(\rho,p|\tilde{\rho},\tilde{p})
    &=\frac{1}{2}\int_{\Omega}\frac{p^{2}}{\rho}-\frac{\tilde{p}^{2}}{\tilde{\rho}}-\frac{2\tilde{p}}{\tilde{\rho}}\left(p-\tilde{p}\right)+\frac{\tilde{p}^{2}}{\tilde{\rho}^{2}}\left(\rho-\tilde{\rho}\right) \dd{x}
    =\frac{1}{2}\int_{\Omega}\rho\abs{\frac{p}{\rho}-\frac{\tilde{p}}{\tilde{\rho}}}^{2}\dd{x}
    \\
    &=\int_{\Omega}\rho\frac{\abs{v-\tilde{v}}^{2}}{2} \dd{x}=:
    \mathcal{R}_{\mathrm{kin}}\left( \varphi;v|\tilde{v}\right)
    \end{aligned}
\end{equation}
with $ p = \rho v $ and  $ \tilde{p}  = \tilde{\rho}\tilde{ v} $.
One can see that
\begin{equation*}
    \mathcal{R}_{\mathrm{kin}}\left( \varphi;v|\tilde{v}\right)=\mathcal{E}_{\mathrm{kin}}(\varphi,v-\tilde{v}).
\end{equation*}
Moreover, since the elastic energy is convex, in particular quadratic, we define the relative elastic energy as:
\begin{equation}
    \mathcal{R}_{\mathrm{el}}(S|\tilde{S}):=\int_{\Omega}\frac{|S-\tilde{S}|^{2}}{2}\dd{x}=\mathcal{E}_{\mathrm{el}}(S-\tilde{S}).
\end{equation}
In addition, by \eqref{assumption on double well potential W}, the singular potential $W$ is $\kappa$-convex, i.e., $W_{\kappa}(\varphi)=W(\varphi)+\frac{\kappa}{2}\abs{\varphi}^{2}$ is convex. The gradient part of the phase-field energy is convex, in particular quadratic. According to~\eqref{subdifferential of phase-field energy}, we define the relative phase-field energy as 
\begin{equation}
\begin{aligned}
    \mathcal{R}_{\mathrm{pf},\kappa}(\varphi|\tilde{\varphi})&
    :=\mathcal{E}_{\mathrm{pf},\kappa}({\varphi})
    -\mathcal{E}_{\mathrm{pf},\kappa}(\tilde{\varphi})
    -\mathrm{D}\mathcal{E}_{\mathrm{pf},\kappa}(\tilde{\varphi})(\varphi-\tilde{\varphi})
    \\
    &=\varepsilon\int_{\Omega}\frac{\abs{\nabla\varphi-\nabla\tilde{\varphi}}^{2}}{2}\dd{x}+\varepsilon^{-1}\int_{\Omega}W_{\kappa}(\varphi)-W_{\kappa}(\tilde{\varphi})-W_{\kappa}(\tilde{\varphi})(\varphi-\tilde{\varphi})\dd{x}
\end{aligned}
\end{equation}
for all $\tilde{\varphi}\in \mathrm{dom}(\mathrm{D}\mathcal{E}_{\mathrm{pf},\kappa})$.

Altogether, we define the relative total energy as:
\begin{equation}
    \mathcal{R}\left(v,S,\varphi\big|\tilde{v},\tilde{S},\tilde{\varphi}\right):=
    \mathcal{R}_{\mathrm{kin}}\left( \varphi,v|\tilde{v}\right)
    +\mathcal{R}_{\mathrm{el}}(S|\tilde{S})
    +\mathcal{R}_{\mathrm{pf},\kappa}(\varphi|\tilde{\varphi}).
	\label{relative energy of two phase}
\end{equation}
We introduce the space of admissible test functions as
\begin{equation}\label{test function space for two phase}
	\scalebox{0.88}{$\mathfrak{T}:=\{(v,S,\varphi):
    v\in C_{0,\mathrm{div}}^{\infty}(\Omega\times[0,T)),S\in C_{0,\mathrm{sym,Tr}}^{\infty}(\Omega\times[0,T))\cap\dom{\mathcal{P}(\varphi;\cdot)},\varphi\in C_{0}^{\infty}(\Omega\times[0,T)),\abs{\varphi}<1\}.$}
\end{equation}
Notice that, in our setting, the spaces
\begin{equation}
    \begin{aligned}
        \mathbb{V}&:=L_{\mathrm{div}}^{2}(\Omega)\times L_{\mathrm{sym,Tr}}^{2}(\Omega)\times L^{2}(\Omega),\\
        \mathbb{Y}&:= H_{0,\mathrm{div}}^{1}(\Omega)\times H_{\mathrm{sym,Tr}}^{1}(\Omega)\times H^{1}(\Omega).
    \end{aligned}
\end{equation}

Now, the derivation of the relative energy inequality for system~\eqref{two phase system alternative} mostly  follows the general approach of~\eqref{relenineq}. 
However, in the definition~\eqref{introSysoper} of the system operator $\mathcal{A}_{\gamma}$ and~\eqref{reldis} of the relative dissipation $\mathcal{W}_{\gamma}^{(\mathcal{K})}$, we have to slightly deviate from the general form for the following reasons: The first difference arises from the non-smoothness of the plastic potential $\mathcal P$, which compels us to treat the multi-valued subdifferential $\partial\mathcal{P}(\varphi;S)$ separately, instead of directly including it in the system operator $\mathcal{A}_{\gamma}$ and in the relative dissipation $\mathcal{W}_{\gamma}^{(\mathcal{K})}$. More precisely, for this term we shall keep the variational inequality and add it to the relative energy inequality, see~\eqref{relative energy estimate for two phase problem} below. Another difference stems from the dependence of $\rho$ on $\varphi,$ which requires to use the momentum $p$ as a variable in the kinetic energy for convexity reasons, cf.\ \eqref{Ekin-tilde}, and for compensation, additional terms are created in the relative dissipation, see~\eqref{relative dissipation not explicit}.  
\par
In the following, we introduce $\mathcal{A}_{\gamma}$ and $\mathcal{W}_{\gamma}^{(\mathcal{K})}$ for system~\eqref{two phase system alternative} and discuss the difference to~\eqref{reldis} in more detail in Remark~\ref{why differ from general approach}.\\
\textbf{System operator $\mathcal{A}_{\gamma}$.} \begin{subequations}\label{system operator for two phase} Due to the non-smoothness of the dissipation potential $\mathcal{P}$, the multi-valued subdifferential $\partial\mathcal{P}(\varphi;S)$ will not be included in the system operator $\mathcal{A}_{\gamma}$. For the remaining terms we  follow~\eqref{introSysoper} and define the system operator 
\begin{equation}
    \mathcal{A}_{\gamma}=\left( \mathcal{A}^{(1)},\mathcal{A}_{\gamma}^{(2)},\mathcal{A}^{(3)}\right)^{\top}:\mathfrak{T}\to (H_{0,\mathrm{div}}^{1}(\Omega))^{\prime}\times(H_{\mathrm{sym,Tr}}^{1}(\Omega))^{\prime}\times(L^{2}(\Omega))^{\prime}    
\end{equation}
from~\eqref{two phase system alternative} as follows: 

\begin{equation}\label{operator A for two phase problem part one}
    \begin{aligned}
	\langle\mathcal{A}^{(1)}(\tilde{v},\tilde{S},\tilde{\varphi}),\Phi\rangle_{H_{0,\mathrm{div}}^{1}}
    :=
        &\int_{\Omega}\partial_{t}(\tilde{\rho}\tilde{v})\cdot\Phi
        -\tilde{\rho}\tilde{v}\otimes\tilde{v}:\nabla\Phi-\tilde{v}\otimes\tilde{J}:\nabla\Phi
        +\eta(\tilde{\varphi})\tilde{S}:\nabla\Phi
        \dd{x}
	\\
        &+\int_{\Omega}2\nu(\tilde{\varphi})\sym{\nabla\tilde{v}}:\nabla\Phi-\tilde{\mu}\nabla\tilde{\varphi}\cdot\Phi  \dd{x}-\langle f, \Phi\rangle_{H^{1}}
\end{aligned}
\end{equation}
for all $\Phi\in H_{0,\mathrm{div}}^{1}(\Omega)$;
\begin{equation}\label{operator A for two phase problem part two}
    \begin{aligned}
	\langle\mathcal{A}_{\gamma}^{(2)}(\tilde{v},\tilde{S},\tilde{\varphi}),\Psi\rangle_{H_{\mathrm{sym,Tr}}^{1}}:=
        &\int_{\Omega}\partial_{t}\tilde{S}:\Psi+\tilde{v}\cdot\nabla\tilde{S}:\Psi+\left( \tilde{S}\skw{\nabla\tilde{v}}-\skw{\nabla\tilde{v}}\tilde{S}\right):\Psi \dd{x}
	\\
	   &+\int_{\Omega}\gamma\nabla\tilde{S}\threedotsbin\nabla\Psi-\eta(\tilde{\varphi})\sym{\nabla\tilde{v}}:\Psi \dd{x}
\end{aligned}
\end{equation}
for all $\Psi\in H_{\mathrm{sym,Tr}}^{1}(\Omega)$;
\begin{align}
	\langle\mathcal{A}^{(3)}(\tilde{v},\tilde{S},\tilde{\varphi}),\zeta \rangle_{L^{2}}:=&\int_{\Omega}\partial_{t}\tilde{\varphi}\zeta+\tilde{v}\cdot\nabla\tilde{\varphi}\zeta-\Delta\tilde{\mu}\zeta \dd{x},
	\label{operator A for two phase problem part three}
\end{align}
for all $\zeta\in L^{2}(\Omega)$, and where 
\begin{align}
	\tilde{\mu}:=-\varepsilon\Delta \tilde{\varphi}+\varepsilon^{-1}W^{\prime}(\tilde{\varphi}).
\end{align}
\end{subequations}
\textbf{Relative dissipation.} Due to the non-smoothness of the plastic potential, we also do not include the multi-valued subdifferential $\partial\mathcal{P}(\varphi;S)$ into the relative dissipation. For the remaining terms and for a given regularity weight $\mathcal{K}$ as in \eqref{regularity weight}, we now introduce the relative dissipation following~\eqref{reldis}. However, as already mentioned, in order to ensure the convexity of the total energy, hence the positive definiteness of its second derivative, we have to use $\tilde{\mathcal{E}}_{\mathrm{tot},\kappa},$ cf.\  \eqref{Ekin-tilde} and below. To compensate this change of variables, two additional terms are created in the definition of the relative dissipation:
\begin{equation}\label{relative dissipation not explicit}
    \begin{aligned}
    \mathcal{W}_{\gamma}^{(\mathcal{K})}\left(v,S,\varphi|\tilde{v},\tilde{S},\tilde{\varphi}\right)
    :=&\mathcal{D}_{\mathrm{chs}}(v,\varphi,\mu)+\mathcal{D}_{\mathrm{sd},\gamma}(S)
    -\langle A_{\gamma}(v,S,\varphi,\mu),\mathrm{D}\tilde{\mathcal{E}}_{\mathrm{tot},\kappa}(\tilde{p},\tilde{S},\tilde{\varphi})\rangle_{\mathbb{Y}}
    \\
    &-\left\langle \mathcal{A}_{\gamma}(\tilde{v},\tilde{S},\tilde{\varphi}),\mathrm{D}^{2}\tilde{\mathcal{E}}_{\mathrm{tot},\kappa}(\tilde{p},\tilde{S},\tilde{\varphi})\begin{pmatrix}
             p  -\tilde {p} 
             \\ 
             S-\tilde {S} 
             \\ 
             \varphi - \tilde {\varphi} 
        \end{pmatrix} \right\rangle_{\mathbb{Y}}
    \\
    &+\mathcal{K}(\tilde{v},\tilde{S},\tilde{\varphi})\mathcal{R}(v,S,\varphi|\tilde{v},\tilde{S},\tilde{\varphi})\\
    &+\left\langle \mathcal{A}_\gamma (\tilde{ v},\tilde S,\tilde \varphi),\begin{pmatrix}
        \frac{\rho-\tilde\rho}{\tilde \rho }(v-\tilde v)\\ 0 \\ - \frac{\rho_2-\rho_1}{2}\tilde{v} \frac{\rho-\tilde\rho}{\tilde \rho }(v-\tilde v)
    \end{pmatrix}\right\rangle_{\!\!\!\mathbb{Y}}
    \\
    &+\int_{\Omega}(\rho-\tilde{\rho})(v-\tilde{v})\cdot\partial_{t}\tilde{v}\dd{x}\,,
    \end{aligned}
\end{equation}
where $\mathcal{A}_{\gamma}$ is defined as in~\eqref{system operator for two phase} and, using~\eqref{energies functionals}, we set 
\begin{equation*}
    \begin{aligned}
        \left\langle A_{\gamma}(v,S,\varphi,\mu),\mathrm{D}\tilde{\mathcal{E}}_{\mathrm{tot},\kappa}(\tilde{p},\tilde{S},\tilde{\varphi})\right\rangle_{\mathbb{Y}}:=
        \left\langle
        \begin{pmatrix}
            A^{(1)}(v,S,\varphi)\\
            A_{\gamma}^{(2)}(v,S,\varphi)\\
            A^{(3)}(v,\varphi)
        \end{pmatrix},
        \begin{pmatrix}
            \tilde{v}\\
            \tilde{S}\\
            \tilde{\mu}+\frac{\kappa}{\varepsilon}\tilde{\varphi}-\frac{\tilde{v}^{2}}{2}\cdot\frac{\rho_{2}-\rho_{1}}{2}
        \end{pmatrix}
        \right\rangle_{\!\!\!\mathbb{Y}}.
    \end{aligned}
\end{equation*}
In particular,
\begin{equation}
    \begin{split}
    \langle A^{(1)}(v,S,\varphi,\mu),\tilde{v}\rangle_{H_{0,\mathrm{div}}^{1}}:=&
        \int_{\Omega}
        -{\rho}{v}\otimes{v}:\nabla\tilde{v}-{v}\otimes{J}:\nabla\tilde{v}
        +\eta({\varphi}){S}:\nabla\tilde{v}
        \dd{x}
	\\
        &+\int_{\Omega}2\nu({\varphi})\sym{\nabla{v}}:\nabla\tilde{v}-{\mu}\nabla{\varphi}\cdot\tilde{v}  \dd{x}-\langle f, \tilde{v}\rangle_{H^{1}}\,,
    \label{opartor A without}
    \end{split}
\end{equation}
\begin{equation}
    \begin{split}
    \langle A_{\gamma}^{(2)}(v,S,\varphi,\mu),\tilde{S}\rangle_{H_{\mathrm{sym,Tr}}^{1}}:=&
        \int_{\Omega}
        -{S}\otimes{v}\threedotsbin\nabla\tilde{S}
        +\left( {S}\skw{\nabla{v}}-\skw{\nabla{v}}{S}\right):\tilde{S} \dd{x}
	\\
	   &+\int_{\Omega}
       \gamma\nabla{S}\threedotsbin\nabla\tilde{S}
       -\eta({\varphi})\sym{\nabla{v}}:\tilde{S} \dd{x}
    \\
    =&
    \int_{\Omega}
    -{S}\otimes{v}\threedotsbin\nabla\tilde{S}
    + 2{S}\skw{\nabla{v}}:\tilde{S} 
    +
       \gamma\nabla{S}\threedotsbin\nabla\tilde{S}
       -\eta({\varphi})\sym{\nabla{v}}:\tilde{S} \dd{x},
       \label{operator A without subdifferential}
\end{split}
\end{equation}
where we have used the fact that $\tilde{S}$ is symmetric, hence 
\begin{align*}
    \left( {S}\skw{\nabla{v}}-\skw{\nabla{v}}{S}\right):\tilde{S}=2{S}\skw{\nabla{v}}:\tilde{S},
\end{align*}
as well as
\begin{align}
    \langle A^{(3)}(v,S,\varphi,\mu),\tilde{\mu}+\frac{\kappa}{\varepsilon}\tilde{\varphi}-\frac{\rho_{2}-\rho_{1}}{2}\frac{\tilde{v}^{2}}{2}\rangle_{H^{1}}:=&
    \int_{\Omega}
        {v}\cdot\nabla{\varphi}(\tilde{\mu}+\frac{\kappa}{\varepsilon}\tilde{\varphi}-\frac{\rho_{2}-\rho_{1}}{2}\frac{\tilde{v}^{2}}{2})\dd{x}
    \nonumber\\
        &+\int_{\Omega}\nabla{\mu}:\nabla(\tilde{\mu}+\frac{\kappa}{\varepsilon}\tilde{\varphi}-\frac{\rho_{2}-\rho_{1}}{2}\frac{\tilde{v}^{2}}{2}) \dd{x}\,.
    \label{operator A without time derivative}
\end{align}
By collecting all the terms from~\eqref{relative dissipation not explicit}--\eqref{operator A without time derivative}, using the definition of the dissipation potentials from \eqref{dissipation functionals}, by adding and subtracting additional terms in order to create differences of the solution and the test function, and by applying integration by parts, we arrive at 
\begin{equation}
    \begin{aligned}    \mathcal{W}_{\gamma}^{(\mathcal{K})}\left(v,S,\varphi|\tilde{v},\tilde{S},\tilde{\varphi}\right):=
        &\int_{\Omega}\gamma\abs{\nabla S-\nabla\tilde{S}}^{2}\dd{x}
    +\int_{\Omega}\frac{1}{2}\abs{\nabla\mu-\nabla\tilde{\mu}}^{2}\dd{x}
    \\
    &+\int_{\Omega}2\nu(\varphi)\abs{\sym{\nabla v}-\sym{\nabla\tilde{v}}}^{2}
    +2( \nu(\varphi)-\nu(\tilde{\varphi}))\sym{\nabla\tilde{v}}:\left({\nabla v}-{\nabla\tilde{v}}\right)\dd{x}
    \\
        &+\int_{\Omega}\frac{1}{2}\abs{\nabla \mu}^{2}
    -\frac{1}{2}\abs{\nabla\tilde{\mu}}^{2}+ \Delta\tilde{\mu}\left(-\Delta(\varphi-\tilde{\varphi})+W^{\prime\prime}(\tilde{\varphi})(\varphi-\tilde{\varphi})\right)\dd{x}
    \\
        &+\int_{\Omega}(v-\tilde{v})\otimes(\rho v-\tilde{\rho}\tilde{v} + J-\tilde{J}):\nabla\tilde{v}
        +(\rho-\tilde{\rho})(v-\tilde{v})\cdot\partial_{t}\tilde{v}\dd{x}
    \\
        &-\int_{\Omega}(\eta(\varphi)-\eta(\tilde{\varphi}))(S-\tilde{S}):\nabla\tilde{v}
		+(\eta(\varphi)-\eta(\tilde{\varphi}))\tilde{S}:({\nabla v}-{\nabla\tilde{v}})\dd{x}
	\\
		&-\int_{\Omega}(S-\tilde{S})\otimes(v-\tilde{v})\threedotsbin\nabla\tilde{S} 
        +2 
        (S-\tilde{S})\skw{\nabla v-\nabla\tilde{v}}
        :{\tilde{S}} \dd{x}
	\\
		&-\int_{\Omega}\tilde{\mu}(\nabla\varphi-\nabla\tilde{\varphi})\cdot(v-\tilde{v})
		-\varepsilon(\nabla\varphi-\nabla\tilde{\varphi})\otimes(\nabla\varphi-\nabla\tilde{\varphi}):\nabla\tilde{v} \dd{x}
	\\
		&+\int_{\Omega}\frac{\kappa}{\varepsilon}(\nabla\mu-\nabla\tilde{\mu})\cdot(\nabla\varphi-\nabla\tilde{\varphi})
		+\frac{\kappa}{\varepsilon}(v-\tilde{v})\cdot\nabla\tilde{\varphi}(\varphi-\tilde{\varphi}) \dd{x}
	\\
		&+\mathcal{K}(\tilde{v},\tilde{S},\tilde{\varphi})\mathcal{R}(v,S,\varphi|\tilde{v},\tilde{S},\tilde{\varphi})
    \\
    =&\mathcal{D}_{\mathrm{sd},\gamma}(S-\tilde{S})
    +\mathcal{W}_{0}^{(\mathcal{K})}(v,S,\varphi|\tilde{v},\tilde{S},\tilde{\varphi}).
    \label{dissipation function like quantity for two phase}
\end{aligned}
\end{equation}

\begin{remark}\label{why differ from general approach}
    In order to see the connection of~\eqref{relative dissipation not explicit} and~\eqref{reldis}, in particular, how the two additional terms arise in \eqref{relative dissipation not explicit}, we calculate the second derivative of the adapted energy $ {\mathcal{E}}_{\text{tot},\kappa} (p , S , \varphi):= {\mathcal{E}}_{\text{kin}} (\varphi, p ) + \mathcal{E}_{\text{el}}(S) + \mathcal E_{\text{pf},\kappa}(\varphi)$ such that we observe  that
    \begin{equation*}
        \begin{aligned}
        \mathrm{D}^2{\mathcal{E}}_{\text{tot},\kappa}( \tilde{p} ,\tilde{S} , \tilde{\varphi}) 
        \begin{pmatrix}
             p  -\tilde {p} 
             \\ 
             S-\tilde {S} 
             \\ 
             \varphi - \tilde {\varphi} 
        \end{pmatrix} 
        =& \begin{pmatrix}
            \frac{1}{\tilde{\rho}} I & 0 & -\frac{\rho_2-\rho_1}{2\tilde{\rho}^{2}} \tilde {p} 
            \\
            0 & I & 0 
            \\
            -\frac{\rho_2-\rho_1}{2\tilde{\rho}^{2}} \tilde {p}  & 0 & - \Delta + W''(\tilde{\varphi}) + \frac{\kappa}{\varepsilon} {+\frac{|\tilde p|^2}{\tilde{\rho}^3}\rho'(\tilde{\varphi})^2}
        \end{pmatrix} \begin{pmatrix}
             p  - \tilde{p}  
             \\ 
             S-\tilde{S} 
             \\ 
             \varphi - \tilde{\varphi} 
        \end{pmatrix} 
        \\
        =&\begin{pmatrix}
            \frac{\rho}{\tilde{\rho}}(v-\tilde{v})
             \\ 
             S-\tilde{S} 
             \\ 
            (- \Delta + W''(\tilde{\varphi}) + \frac{\kappa}{\varepsilon} ) (\varphi - \tilde{\varphi})
            -\frac{\rho_{2}-\rho_{1}}{2}\frac{\rho}{\tilde{\rho}}{\tilde{v}\cdot(v-\tilde{v})}
        \end{pmatrix}. 
    \end{aligned}
    \end{equation*}
    By following the abstract approach of~\eqref{relenineq} and by  adding and subtracting the term
    \begin{equation*}
        \left \langle \mathcal{A}_\gamma (\tilde{ v},\tilde S,\tilde \varphi),\begin{pmatrix}
        \frac{\rho-\tilde\rho}{\tilde \rho }(v-\tilde v)\\ 0 \\ - \frac{\rho_2-\rho_1}{2}\tilde{v} \frac{\rho-\tilde\rho}{\tilde \rho }(v-\tilde v)
    \end{pmatrix}\right \rangle_{\!\!\!\mathbb{Y}},
    \end{equation*}
    we arrive at the formulation
    \begin{equation}
        \left\langle \mathcal{A}_{\gamma}(\tilde{v},\tilde{S},\tilde{\varphi}),
			\begin{pmatrix}
				v-\tilde{v}\\
				S-\tilde{S}\\
				-\Delta(\varphi-\tilde{\varphi})+W^{\prime\prime}(\tilde{\varphi})(\varphi-\tilde{\varphi})+\frac{\kappa}{\varepsilon}(\varphi-\tilde{\varphi})-\frac{\rho_{2}-\rho_{1}}{2}(v-\tilde{v})\cdot\tilde{v}
			\end{pmatrix}
		\right\rangle_{\!\!\!\mathbb{Y}} .
    \end{equation}
    Notice that this substitution gives an extra term
    \begin{align*}
    &\int_{\Omega}\partial_{t}(\tilde{\rho}\tilde{v})
    (-\frac{\rho}{\tilde{\rho}}(v-\tilde{v})+(v-\tilde{v}))+\partial_{t}\tilde{\varphi}(\frac{\rho_{2}-\rho_{1}}{2}\frac{\rho}{\tilde{\rho}}\tilde{v}(v-\tilde{v})-\frac{\rho_{2}-\rho_{1}}{2}\tilde{v}(v-\tilde{v}))\dd{x}
    \\
    =&\int_{\Omega}-(\rho-\tilde{\rho})(v-\tilde{v})\partial_{t}\tilde{v}\dd{x},
\end{align*}
which also contributes to $\mathcal{W}_{\gamma}^{(\mathcal{K})}$ in \eqref{dissipation function like quantity for two phase} and thus, to the definition of dissipative solutions.

Following here the original definition of dissipative solutions according to Lions~\cite[Sec.~4.4]{lionsMathematicalTopicsFluid1996}, the solution fulfills the relative energy inequality for all smooth enough test functions. 
There are other solvability concepts that are coined dissipative in the literature. For instance certain measure-valued solutions (e.g.,~\cite{basaric}) or essentially weak solutions~\cite{feireisl2014relative}. 
\end{remark}

Now, we can introduce the dissipative solution for system~\eqref{two phase system alternative} according to Definition~\ref{dissipative solution for general system}. 

\begin{defi}[Dissipative solution for system~\eqref{two phase system alternative}]\label{dissipative solution for two phase problem}
    Let the assumptions~\ref{assumption on Omega}-\ref{TP assumption} hold true. Let $\mathcal{K}$ be a regularity weight satisfying~\eqref{regularity weight}. A quadruplet $(v,S,\varphi,\mu)$ is called a dissipative solution of type $\mathcal{K}$ of the two-phase system \eqref{two phase system alternative} if the following properties are satisfied:\\
    \begin{subequations}
    1. The quadruplet $(v,S,\varphi,\mu)$ has the following regularity:
    \begin{equation}
    \begin{aligned}    
		v&\in L_{\mathrm{loc}}^{\infty}([0,T),L_{\mathrm{div}}^{2}(\Omega))\cap L_{\mathrm{loc}}^{2}([0,T),H_{0,\mathrm{div}}^{1}(\Omega)),\\
		S&\in L_{\mathrm{loc}}^{\infty}([0,T),L_{\mathrm{sym,Tr}}^{2}(\Omega)),\\
		\varphi&\in L_{\mathrm{loc}}^{\infty}([0,T);H^{1}(\Omega))\cap L_{\mathrm{loc}}^{2}([0,T);H^{2}(\Omega)), W^{\prime}(\varphi)\in L_{\mathrm{loc}}^{2}([0,T);L^{2}(\Omega)),\\
		\mu&\in L_{\mathrm{loc}}^{2}([0,T);H^{1}(\Omega)).
    \end{aligned}
    \end{equation}
    2. With the relative energy $\mathcal{R}$ from~\eqref{relative energy of two phase}, the system operator $\mathcal{A}_{\gamma}$ from~\eqref{system operator for two phase}, and the relative dissipation $\mathcal{W}_{\gamma}^{(\mathcal{K})}$ from~\eqref{dissipation function like quantity for two phase}, the quadruplet $(v,S,\varphi,\mu)$ satisfies the following relative energy-dissipation estimate:
    \begin{equation}\label{relative energy estimate for two phase problem}
    \begin{aligned}
			&\mathcal{R}(v(t),S(t),\varphi(t)|\tilde{v}(t),\tilde{S}(t),\tilde{\varphi}(t))
	\\
			+\int_{0}^{t}\Big(& \left\langle \mathcal{A}_{\gamma}(\tilde{v},\tilde{S},\tilde{\varphi}),
			\begin{pmatrix}
				v-\tilde{v}\\
				S-\tilde{S}\\
				-\Delta(\varphi-\tilde{\varphi})+W^{\prime\prime}(\tilde{\varphi})(\varphi-\tilde{\varphi})+\frac{\kappa}{\varepsilon}(\varphi-\tilde{\varphi})-\frac{\rho_{2}-\rho_{1}}{2}(v-\tilde{v})\cdot\tilde{v}
			\end{pmatrix}
			\right\rangle_{\!\!\!\mathbb Y} 
		\\
			&+\mathcal{P}(\varphi;S)-\mathcal{P}(\varphi;\tilde{S})+\mathcal{W}_{\gamma}^{(\mathcal{K})}(v,S,\varphi|\tilde{v},\tilde{S},\tilde{\varphi})\Big)
			\exp\left(\int_{s}^{t} \mathcal{K}(\tilde{v},\tilde{S},\tilde{\varphi})\dd{\tau} \right) \dd{s}
		\\
			\leq &\mathcal{R}(v_{0},S_{0},\varphi_{0}|\tilde{v}(0),\tilde{S}(0),\tilde{\varphi}(0)) \exp\left(\int_{0}^{t} \mathcal{K}(\tilde{v},\tilde{S},\tilde{\varphi}) \dd{s} \right)
    \end{aligned}
    \end{equation}
    for all $(\tilde{v},\tilde{S},\tilde{\varphi})\in\mathfrak{T}$ and for a.e. $t\in(0,T)$.
    \end{subequations}
\end{defi}

   


\begin{remark}
    One can still obtain an inequality for the relative energy even without the regularity weight term, but this term turns out to be crucial when carrying out the limit passage $\gamma\to0$ for the relative dissipation. Recall that, when we pass to the limit $\gamma\to0$, we will lose the control on the term $\nabla S$ which is essential for the limit passage of the nonlinear terms $S\skw{\nabla v}-\skw{\nabla v}S$. However, with the help of the chosen regularity weights, $\mathcal{W}_{\gamma}^{(\mathcal{K})}$ can be convex and even continuous and therefore weakly lower semicontinuous. This will allow us to perform the limit passage.

    For notational simplicity, we assume without loss of generality that $\varepsilon=1$.
\end{remark}

\subsection{Properties of the dissipative solutions}
We first provide a lemma that will be useful for passing to the limit in energy inequalities.
\begin{lem}[{\cite[][Lemma 2.2]{eiterWeakstrongUniquenessEnergyvariational2022b}}]\label{result related to energy estimate}
	Let $g_0\in\mathbb{R}$. Let $f\in L^{1}(0,T)$ and $g\in L^{\infty}(0,T)$. Then the following two inequalities are equivalent:
	\begin{equation}\label{result related to energy estimate 1}
		-\int_{0}^{T}\phi^{\prime}(t)g(t)\dd{t}+\int_{0}^{T}\phi(t)f(t)\dd{t}\leq g_0,\mbox{ for all }\phi\in\tilde{C}([0,T]),
	\end{equation}
	where 
    \begin{equation}\label{tilde C space}
        \tilde{C}([0,T]):=\left\{\phi\in C^{1}([0,T]):\phi\geq0\mbox{ on }[0,T],\phi^{\prime}\leq0,\phi(0)=1,\phi(T)=0\right\},    
    \end{equation}
    and
	\begin{equation}\label{result related to energy estimate 2}
		g(t)+\int_{0}^{t}f(s)\dd{s}\leq g_0\mbox{ for a.e. }t\in(0,T).
	\end{equation} 
\end{lem}
\begin{proof}
	See \cite[][Lemma 2.2]{eiterWeakstrongUniquenessEnergyvariational2022b}. Notice that this proof remains valid for $g$ negative.
\end{proof}

We now investigate how dissipative solutions from Definition \ref{dissipative solution for two phase problem} and weak solutions from Definition \ref{gerneralized solution TPP} are connected. The first result shows that a weak solution is also a dissipative solution for arbitrary regularity weight $\mathcal{K}$, given that $\gamma>0$.

\begin{prop}\label{generalized solution is dissipative solution for TPP}
	For $\gamma>0$, let $(v,S,\varphi,\mu)$ be a weak solution in the sense of Definition \ref{gerneralized solution TPP} and let $\mathcal{K}$ satisfy \eqref{regularity weight}. Then $(v,S,\varphi,\mu)$ satisfies the relative energy estimate \eqref{relative energy estimate for two phase problem} for a.e. $t\in(0,T)$, any given regularity weight $\mathcal{K}$ and all $(\tilde{v},\tilde{S},\tilde{\varphi})\in \mathfrak{T}$. Hence, $(v,S,\varphi,\mu)$ is also a dissipative solution.
\end{prop}
\begin{proof}
	Let $(\tilde{v},\tilde{S},\tilde{\varphi})\in \mathfrak{T}$ and for a.e. $t\in(0,T)$, 
    let $\phi\in \tilde{C}([0,t])$ cf.~\eqref{tilde C space}. Since $(v,S,\varphi,\mu)$ is a weak solution in the sense of Definition \ref{gerneralized solution TPP}, it satisfies the weak formulation \eqref{weak velocity TPP}. Besides, since $\tilde{v}\in C_{0,\mathrm{div}}^{\infty}(\Omega\times[0,T))$ is admissible test function for \eqref{weak velocity TPP}. Therefore, we choose the test functions in the weak formulation to be $-\phi\tilde{v}$, in order to obtain that
	\begin{equation}\label{weak formulation tested by velocity}
	    \begin{aligned}
		-&\int_{0}^{t}\phi^{\prime} \int_{\Omega}-\rho v\cdot \tilde{v} \dd{x} \dd{s}
		\\
		+&\int_{0}^{t}\phi \int_{\Omega}\rho v\cdot\partial_{t}\tilde{v}+(\rho v\otimes v):\nabla\tilde{v}-2\nu(\varphi)\sym{\nabla v}:\nabla\tilde{v} \dd{x}\dd{s}
		\\
		-&\int_{0}^{t}\phi\int_{\Omega}\eta(\varphi)S:\nabla\tilde{v}-(v\otimes J):\nabla\tilde{v}-\mu\nabla\varphi\cdot\tilde{v} \dd{x} \dd{s}
        +\int_{0}^{t}\phi\langle f, \tilde{v}\rangle_{H^{1}}\dd{s}
		=\int_{\Omega}-\rho_{0}v_{0}\cdot\tilde{v}(0) \dd{x}.
	\end{aligned}
	\end{equation}
	Furthermore, with the help of Lemma \ref{result related to energy estimate}, the partial energy inequalities \eqref{partial energy inequality kinetic} and \eqref{partial energy inequality phase field} can be equivalently written as
    \begin{equation}\label{weak partial energy inequality kinetic}
        \begin{aligned}
        -&\int_{0}^{t}\phi^{\prime}\int_{\Omega}\frac{\rho(t)}{2}\abs{v(t)}^{2}\dd{x}\dd{s}
        +\int_{0}^{t}\phi\int_{\Omega}2\nu(\varphi)\abs{\sym{\nabla v}}^{2}
        +\eta(\varphi)S:\nabla v 
        -\mu(\nabla\varphi\cdot v) \dd{x}\dd{s}
        \\
        \leq&\int_{\Omega}\frac{\rho_{0}}{2}\abs{v_{0}}^{2} \dd{x}
        +\int_{0}^{t}\phi\langle f, v\rangle_{H^{1}}\dd{s},
    \end{aligned}
    \end{equation}
    as well as 
     \begin{equation}\label{weak partial energy inequality phase field}
         \begin{aligned}
        -&\int_{0}^{t}\phi^{\prime}\int_{\Omega}\frac{\abs{\nabla\varphi(t)}^{2}}{2}+W(\varphi(t))\dd{x}\dd{s}
        +\int_{0}^{t}\phi\int_{\Omega}\abs{\nabla\mu}^{2}
        +\mu(\nabla \varphi\cdot v)\dd{x}\dd{s}
        \\
        \leq&\int_{\Omega}\frac{\abs{\nabla\varphi_{0}}^{2}}{2}+W(\varphi_{0})\dd{x},
    \end{aligned}
    \end{equation}
    Besides, by Lemma \ref{result related to energy estimate}, the evolutionary variational inequality for the stress \eqref{generalized solution stress energy estimate TPP} can be also transformed into the following
    \begin{equation}\label{weak relative energy inequality stress}
        \begin{aligned}
		-&\int_{0}^{t}\phi^{\prime}\int_{\Omega}\frac{1}{2}\abs{ S(t)-\tilde{S}(t)}^{2} \dd{x}\dd{s}
        \\
		&+\int_{0}^{t}\phi\int_{\Omega} \partial_{t}\tilde{S}:(S-\tilde{S}) - v\cdot\nabla S:\tilde{S} - (S\skw{\nabla v}-\skw{\nabla v}S):\tilde{S}\dd{x} \dd{s}
		\\
		&+\int_{0}^{t}\phi(\mathcal{P}(\varphi;S)-\mathcal{P}(\varphi;\tilde{S})) \dd{s}
		+\int_{0}^{t}\phi\int_{\Omega}\gamma\nabla S\threedotsbin\nabla(S-\tilde{S}) - \eta(\varphi)\sym{\nabla v}:(S-\tilde{S})\dd{x} \dd{s} 
		\\
		\leq&\int_{\Omega}\frac{1}{2} \abs{S_{0}-\tilde{S}(0)}^{2}\dd{x}
	\end{aligned}
    \end{equation}
	Moreover, by integration-by-parts in time and space, we calculate the terms in the relative phase-field energy as follows:
    \begin{equation}\label{weak formualtion tested phase field energy}
        \begin{aligned}
		-&\int_{0}^{t}\phi^{\prime} \int_{\Omega}-\nabla\varphi:\nabla\tilde{\varphi}-W^{\prime}(\tilde{\varphi})(\varphi-\tilde{\varphi})-2W(\tilde{\varphi}) \dd{x} \dd{s}
		\\
		=&\int_{\Omega}-\nabla\varphi_{0}:\nabla\tilde{\varphi}(0)-W^{\prime}(\tilde{\varphi}(0))(\varphi_{0}-\tilde{\varphi}(0))-2W(\tilde{\varphi}(0)) \dd{x}
		\\
		&-\int_{0}^{t}\phi \int_{\Omega}
        \nabla\partial_{t}\tilde{\varphi}:\nabla\varphi
        -\partial_{t}\varphi\Delta\tilde{\varphi}
        +W^{\prime}(\tilde{\varphi})(\partial_{t}\varphi-\partial_{t}\tilde{\varphi})
        +W^{\prime\prime}(\tilde{\varphi})\partial_{t}\tilde{\varphi}(\varphi-\tilde{\varphi})
        +2W^{\prime}(\tilde{\varphi})\partial_{t}\tilde{\varphi}
        \dd{x} \dd{s},
	\end{aligned}
    \end{equation}
	where we used that $\phi(0)=1$ and $\phi(t)=0$ for $\phi\in\tilde{C}([0,t])$ and the Neumann boundary condition \eqref{two phase system alternative neumann boundary}.
	Besides, by testing \eqref{weak first CH TPP} with $\phi\tilde{\mu}$ and an integration by parts in time, we obtain
	\begin{equation}
		\int_{0}^{t}\phi \int_{\Omega}\partial_{t}\varphi\tilde{\mu}+v\cdot\nabla\varphi\tilde{\mu}+\nabla\mu:\nabla\tilde{\mu} \dd{x} \dd{s}=0.
		\label{weak formulation tested by mu}
	\end{equation}
	By testing~\eqref{weak second CH TPP} with $-\phi\partial_{t}\tilde{\varphi}$ and integrating in time and space, we derive
	\begin{equation}
		-\int_{0}^{t}\phi \int_{\Omega}\mu\partial_{t}\tilde{\varphi}-W^{\prime}(\varphi)\partial_{t}\tilde{\varphi}+\Delta\varphi\partial_{t}\tilde{\varphi} \dd{x} \dd{s} =0.
		\label{weak formualtion tested by time derivative varphi}
	\end{equation}
	Summing~\eqref{weak formulation tested by mu} and~\eqref{weak formualtion tested by time derivative varphi} implies 
	\begin{equation}
	    \begin{aligned}
		-&\int_{0}^{t}\phi \int_{\Omega}\nabla\mu:\nabla\tilde{\mu} \dd{x} \dd{s}
		\\
		=&\int_{0}^{t}\phi \int_{\Omega}\partial_{t}\varphi\tilde{\mu}+v\cdot\nabla\varphi\tilde{\mu}-\mu\partial_{t}\tilde{\varphi}+W^{\prime}(\varphi)\partial_{t}\tilde{\varphi}+\nabla\varphi\cdot\nabla\partial_{t}\tilde{\varphi} \dd{x} \dd{s}.
		\label{weak formulation nabla mu}
	\end{aligned}
	\end{equation}
	In addition, we can calculate the contribution of non-convexity of the singular potential, i.e.
	\begin{equation}\label{weak formulation nonconvexity not yet finished}
	    \begin{aligned}
		-&\kappa\int_{0}^{t}\phi^{\prime} \int_{\Omega}\frac{\abs{\varphi-\tilde{\varphi}}^{2}}{2} \dd{x} \dd{s}
		\\
		=&\kappa\int_{\Omega}\frac{\abs{\varphi_{0}-\tilde{\varphi}(0)}^{2}}{2} \dd{x}
		+\kappa\int_{0}^{t}\phi \int_{\Omega}(\varphi-\tilde{\varphi})(\partial_{t}\varphi-\partial_{t}\tilde{\varphi}) \dd{x} \dd{s}.
	\end{aligned}
	\end{equation}
	Notice that, with the help of~\eqref{weak first CH TPP}, the second term on the right-hand side of~\eqref{weak formulation nonconvexity not yet finished} can be written as
	\begin{equation*}
	    \begin{aligned}
		&\int_{\Omega}-\kappa(\varphi-\tilde{\varphi})(\partial_{t}\varphi-\partial_{t}\tilde{\varphi}) \dd{x}
		\\
		=&\int_{\Omega}\kappa(\nabla\mu-\nabla\tilde{\mu})\cdot(\nabla\varphi-\nabla\tilde{\varphi})\dd{x}
		\\
		&+\int_{\Omega}\kappa\left( v\cdot(\nabla\varphi-\nabla\tilde{\varphi})+(v-\tilde{v})\cdot\nabla\tilde{\varphi}\right)(\varphi-\tilde{\varphi})\dd{x}
		\\
		&+\langle \mathcal{A}^{(3)}(\tilde{v},\tilde{S},\tilde{\varphi}),\kappa(\varphi-\tilde{\varphi})\rangle
	\end{aligned}
	\end{equation*}
	Also, since $v$ is a divergence-free function, thanks to integration by parts, we know
	\begin{equation*}
		\int_{\Omega}v\cdot(\nabla\varphi-\nabla\tilde{\varphi})(\varphi-\tilde{\varphi})\dd{x}
		=0.
	\end{equation*}
	Inserting these two equalities into~\eqref{weak formulation nonconvexity not yet finished} yields
	\begin{equation}\label{weak formulation of non convexity}
	    \begin{aligned}
		-&\kappa\int_{0}^{t}\phi^{\prime} \int_{\Omega}\frac{\abs{\varphi-\tilde{\varphi}}^{2}}{2} \dd{x} \dd{s}
		+\int_{0}^{t}\phi\int_{\Omega}\kappa(\nabla\mu-\nabla\tilde{\mu})\cdot(\nabla\varphi-\nabla\tilde{\varphi})+\kappa(v-\tilde{v})\cdot\nabla\tilde{\varphi}(\varphi-\tilde{\varphi})\dd{x} \dd{s}
		\\
		+&\int_{0}^{t}\phi\langle \mathcal{A}^{(3)}(\tilde{v},\tilde{S},\tilde{\varphi}),\kappa(\varphi-\tilde{\varphi})\rangle\dd{s}
		=\kappa\int_{\Omega}\frac{\abs{\varphi_{0}-\tilde{\varphi}(0)}^{2}}{2} \dd{x}.
	\end{aligned}
	\end{equation}
	Now, observe that $v-\tilde{v}$ is an admissible test function for \eqref{operator A for two phase problem part one}, $S-\tilde{S}$ is admissible for \eqref{operator A for two phase problem part two} and $-\Delta(\varphi-\tilde{\varphi})+W^{\prime\prime}(\tilde{\varphi})(\varphi-\tilde{\varphi})$ is admissible for~\eqref{operator A for two phase problem part three}. Moreover, notice that $-\frac{\rho_{2}-\rho_{1}}{2}\tilde{v}\cdot(v-\tilde{v})$ is also admissible for \eqref{operator A for two phase problem part three}. Hence, choosing them as the test functions for each operator and summing up~\eqref{weak formulation tested by velocity}-\eqref{weak formualtion tested phase field energy}, \eqref{weak formulation nabla mu}, and \eqref{weak formulation of non convexity} yields that
	\begin{equation}\label{relative energy disssipation estimate not yet finished}
	    \begin{aligned}
		-&\int_{0}^{t}\phi^{\prime} \mathcal{R}(v(t),S(t),\varphi(t)|\tilde{v}(t),\tilde{S}(t),\tilde{\varphi}(t)) \dd{s}
		\\
		&+\int_{0}^{t}\phi \left\langle \mathcal{A}_{\gamma}(\tilde{v},\tilde{S},\tilde{\varphi}),\begin{pmatrix}
			v-\tilde{v}\\
			S-\tilde{S}\\
			-\Delta(\varphi-\tilde{\varphi})+W^{\prime\prime}(\tilde{\varphi})(\varphi-\tilde{\varphi})+\kappa(\varphi-\tilde{\varphi})-\frac{\rho_{2}-\rho_{1}}{2}\tilde{v}\cdot(v-\tilde{v})
		\end{pmatrix}
		\right\rangle_{\!\!\!\mathbb{Y}} \dd{s}
		\\
		&+\int_{0}^{t}\phi\left(\mathcal{P}(\varphi;S)-\mathcal{P}(\varphi;\tilde{S})\right) \dd{s}
		+\int_{0}^{t}\phi \widetilde{\mathcal{W}}(v,S,\varphi|\tilde{v},\tilde{S},\tilde{\varphi}) \dd{s}
		\\
		\leq&\mathcal{R}(v_{0},S_{0},\varphi_{0}|\tilde{v}(0),\tilde{S}(0),\tilde{\varphi}(0)),
	\end{aligned}
	\end{equation}
	where $\widetilde{\mathcal{W}}$ is given by
    \begin{equation}
    \begin{aligned}
        \widetilde{\mathcal{W}}(v,S,\varphi|\tilde{v},\tilde{S},\tilde{\varphi})=\mathcal{W}_{\gamma}^{(\mathcal{K})}\left(v,S,\varphi|\tilde{v},\tilde{S},\tilde{\varphi}\right)-\mathcal{K}(\tilde{v},\tilde{S},\tilde{\varphi})\mathcal{R}(v,S,\varphi|\tilde{v},\tilde{S},\tilde{\varphi}).
    \end{aligned}
    \end{equation}
    Finally, by choosing $\phi(t)=\psi(t)\exp\left(-\int_{0}^{t}\mathcal{K}(\tilde{v},\tilde{S},\tilde{\varphi})\dd{\tau}\right)$ for $\psi\in\tilde{C}([0,t])$, we arrive at the relative energy estimate \eqref{relative energy estimate for two phase problem}.
\end{proof}

The next two results state that the velocity component and the phase-field variable of a dissipative solution of type $\mathcal{K}$ are also a weak solution of the momentum balance~\eqref{fluid equation momentum} and a weak solution of the evolution law~\eqref{first CHE}, respectively, under certain assumptions on $\mathcal{K}$.

\begin{prop}\label{energy-variational solution satisfies the weak formulation of Cahn-Hilliard equation}
    Suppose that the regularity weight $\mathcal{K}$ satisfies $\mathcal{K}(0,0,\tilde{\varphi})=0$ for all $\tilde{\varphi}\in C_{0}^{\infty}(\Omega\times[0,T))$ such that $\tilde{\varphi}\in(-1,1)$. Moreover, assume that $\partial_{t}\varphi\in L^{2}(0,T;H^{-1}(\Omega))$. Then a dissipative solution of type $\mathcal{K}$ is a weak solution of the phase-field evolutionary law, i.e. 
    \eqref{weak first CH TPP}
    is satisfied for all $\zeta$ such that $\zeta\in L^{2}([0,T);H^{1}(\Omega))$, $\partial_{t}\zeta\in L^{2}([0,T);L^{2}(\Omega))$ and $\zeta
    (T)=0$.
\end{prop}

\begin{proof}
    From the assumption on $\mathcal{K}$, we have that $\exp(\int_{0}^{t}\mathcal{K}(\tilde{v},0,0)\dd{\tau})=e^{0}=1$ for all $\tilde{v}\in C_{0,\mathrm{div}}^{\infty}(\Omega\times[0,T))$. Setting $\tilde{v}\equiv0$ and $\tilde{S}\equiv0$ in \eqref{relative energy estimate for two phase problem} and with the fact that $\mathcal{P}(\varphi;0)\equiv0$, we obtain    
    \begin{equation}\label{relative energy disssipation estimate with zero velocity and stress}
        \begin{aligned}
			&\mathcal{R}(v(t),S(t),\varphi(t)|0,0,\tilde{\varphi}(t))
			+\int_{0}^{t}\mathcal{P}(\varphi;S)+\mathcal{W}_{\gamma}^{(\mathcal{K})}(v,S,\varphi|0,0,\tilde{\varphi})\dd{s}
                \\
			&+\int_{0}^{t} \left\langle \mathcal{A}_{\gamma}(0,0,\tilde{\varphi}),
			\begin{pmatrix}
				v\\
				S\\
				-\Delta(\varphi-\tilde{\varphi})+W^{\prime\prime}(\tilde{\varphi})(\varphi-\tilde{\varphi})+\kappa(\varphi-\tilde{\varphi})
			\end{pmatrix}
			\right\rangle_{\!\!\!\mathbb{Y}} 
			 \dd{s}
			\\
			\leq &\mathcal{R}(v_{0},S_{0},\varphi_{0}|0,0,\tilde{\varphi}(0)) 
    \end{aligned}
    \end{equation}
    for all $\tilde{\varphi}\in C_{0}^{\infty}(\Omega\times[0,T))$ with $\tilde{\varphi}\in(-1,1)$ and a.e. $t\in(0,T)$. In view of \eqref{operator A for two phase problem part one}-\eqref{operator A for two phase problem part three}, we have
    \begin{subequations}
        \begin{equation}
            \langle\mathcal{A}^{(1)}(0,0,\tilde{\varphi}),v\rangle_{H_{0,\mathrm{div}}^{1}}=
            -\langle f, v\rangle_{H^{1}}
            -\int_{\Omega}\tilde{\mu}\nabla\tilde{\varphi}\cdot v \dd{x},
            \label{operator A part one with zero velocity and stress}
        \end{equation}
        and
        \begin{equation}
            \langle\mathcal{A}_{\gamma}^{(2)}(0,0,\tilde{\varphi}),S\rangle_{H_{\mathrm{sym,Tr}}^{1}}=0,
            \label{operator A part two with zero velocity and stress}
        \end{equation}
        and
        \begin{equation}\label{operator A part three with zero velocity and stress}
            \begin{aligned}
            &\langle\mathcal{A}^{(3)}(0,0,\tilde{\varphi}),-\Delta(\varphi-\tilde{\varphi})+W^{\prime\prime}(\tilde{\varphi})(\varphi-\tilde{\varphi})+\kappa(\varphi-\tilde{\varphi}) \rangle_{L^{2}}
            \\
            =&\int_{\Omega}\partial_{t}\tilde{\varphi}\left(-\Delta(\varphi-\tilde{\varphi})+W^{\prime\prime}(\tilde{\varphi})(\varphi-\tilde{\varphi})+\kappa(\varphi-\tilde{\varphi})\right)\dd{x}
            \\
            &-\int_{\Omega}\Delta\tilde{\mu}\left(-\Delta(\varphi-\tilde{\varphi})+W^{\prime\prime}(\tilde{\varphi})(\varphi-\tilde{\varphi})+\kappa(\varphi-\tilde{\varphi})\right) \dd{x},
        \end{aligned}
        \end{equation}
    \end{subequations}
    Moreover, by \eqref{dissipation function like quantity for two phase}, it is
    \begin{equation}\label{dissipation function like quantity with zero velocity and stress}
        \begin{aligned}
        \mathcal{W}_{\gamma}^{(\mathcal{K})}(v,S,\varphi|0,0,\tilde{\varphi})
		=&\int_{\Omega}
		2\nu(\varphi)\abs{\sym{\nabla v}}^{2}
        +\gamma\abs{\nabla S}^{2}
        +\abs{\nabla\mu}^{2}\dd{x}
		\\
		&-\int_{\Omega}\nabla\mu\cdot\nabla\tilde{\mu}
        +\tilde{\mu}(\nabla\varphi-\nabla\tilde{\varphi})\cdot v\dd{x}
		\\
		&+\int_{\Omega}\Delta\tilde{\mu}(-\Delta(\varphi-\tilde{\varphi})+W^{\prime\prime}(\tilde{\varphi})(\varphi-\tilde{\varphi})) \dd{x}
		\\
		&+\int_{\Omega}\kappa(\nabla\mu-\nabla\tilde{\mu})\cdot(\nabla\varphi-\nabla\tilde{\varphi})
		+\kappa v\cdot\nabla\tilde{\varphi}(\varphi-\tilde{\varphi}) \dd{x}.
    \end{aligned}
    \end{equation}
    Inserting~\eqref{operator A part one with zero velocity and stress}-\eqref{operator A part three with zero velocity and stress} and~\eqref{dissipation function like quantity with zero velocity and stress} into~\eqref{relative energy disssipation estimate with zero velocity and stress} and applying Lemma~\ref{result related to energy estimate} results in the estimate
    \begin{equation}
        \begin{aligned}
        -&\int_{0}^{T}\phi^{\prime}\int_{\Omega}\rho\frac{\abs{v}^{2}}{2}+\frac{\abs{S}^{2}}{2}+\frac{\abs{\nabla\varphi-\nabla\tilde{\varphi}}^{2}}{2}+W(\varphi)-W(\tilde{\varphi})-W^{\prime}(\tilde{\varphi})(\varphi-\tilde{\varphi})+\frac{\kappa}{2}\abs{\varphi-\tilde{\varphi}}^{2} \dd{x}\dd{t}
        \\
        &+\int_{0}^{T}\phi\int_{\Omega}
		2\nu(\varphi)\abs{\sym{\nabla v}}^{2}
        +\gamma\abs{\nabla S}^{2}
        +\abs{\nabla\mu}^{2}\dd{x}\dd{t}
        -\int_{0}^{T}\phi\langle f,v\rangle_{H^{1}}\dd{t}
        +\int_{0}^{T}\phi\mathcal{P}(\varphi;S)\dd{t}
        \\
        &+\int_{0}^{T}\phi\int_{\Omega}\partial_{t}\tilde{\varphi}\left(-\Delta(\varphi-\tilde{\varphi})+W^{\prime\prime}(\tilde{\varphi})(\varphi-\tilde{\varphi})+\kappa(\varphi-\tilde{\varphi})\right)\dd{x}\dd{t}
        \\
        &-\int_{0}^{T}\phi\int_{\Omega}\nabla\mu\cdot\nabla\tilde{\mu}
        +\tilde{\mu}\nabla\varphi\cdot v 
        -\kappa\nabla\mu\cdot(\nabla\varphi-\nabla\tilde{\varphi})
		-\kappa v\cdot\nabla\varphi(\varphi-\tilde{\varphi})\dd{x}\dd{t}
        \\
        \leq&\int_{\Omega}\frac{\abs{\nabla\varphi_{0}-\nabla\tilde{\varphi}(0)}^{2}}{2}+W(\varphi_{0})-W(\tilde{\varphi}(0))-W^{\prime}(\tilde{\varphi}(0))(\varphi_{0}-\tilde{\varphi}(0))+\frac{\kappa}{2}\abs{\varphi_{0}-\tilde{\varphi}(0)}^{2} \dd{x}
        \\
        &+\int_{\Omega}\rho_{0}\frac{\abs{v_{0}}^{2}}{2}+\frac{\abs{S_{0}}^{2}}{2}\dd{x},
    \end{aligned}
    \end{equation}
    for all $\phi\in\tilde{C}([0,T])$. 
    Notice that, by integration by parts in time, rearranging the terms in the above inequality and exploiting cancellations in terms that depend on $\tilde{\mu}+\kappa\tilde{\varphi}$, we then get
    \begin{equation}\label{simplied relative energy estimate when velocity and stress are zero}
        \begin{aligned}
        -&\int_{0}^{T}\phi^{\prime}\int_{\Omega}\rho\frac{\abs{v}^{2}}{2}
        +\frac{\abs{S}^{2}}{2}\dd{x}\dd{t}
        -\int_{\Omega}\rho_{0}\frac{\abs{v_{0}}^{2}}{2}
        +\frac{\abs{S_{0}}^{2}}{2}\dd{x}
        \\
        +&\int_{0}^{T}\phi\int_{\Omega}
		2\nu(\varphi)\abs{\sym{\nabla v}}^{2}
        +\gamma\abs{\nabla S}^{2}\dd{x}\dd{t}
        -\int_{0}^{T}\phi\langle f, v\rangle_{H^{1}}\dd{t}
        +\int_{0}^{T}\phi\mathcal{P}(\varphi;S)\dd{t}
        \\
        +&\int_{0}^{T}\phi\langle \partial_{t}\varphi,\mu+\kappa\varphi\rangle_{H^{1}}\dd{t}
        +\int_{0}^{T}\phi\int_{\Omega}\kappa v\cdot\nabla\varphi\varphi
        +\nabla\mu(\nabla\mu+\kappa\nabla\varphi)\dd{x}\dd{t}
        \\
        -&\int_{0}^{T}\phi\langle\partial_{t}\varphi,\tilde{\mu}+\kappa\tilde{\varphi}\rangle_{H^{1}}\dd{t}
        -\int_{0}^{T}\phi\int_{\Omega}v\cdot\nabla\varphi(\tilde{\mu}+\kappa\tilde{\varphi}) 
        +\nabla\mu\cdot(\nabla\tilde{\mu}+\kappa\nabla\tilde{\varphi})
		\dd{x}\dd{t}
        \leq0.
    \end{aligned}
    \end{equation}
    First, notice that in \eqref{simplied relative energy estimate when velocity and stress are zero}, the regularity of $\tilde{\varphi}$ can be reduced such that $\mathrm{D}\mathcal{E}_{\mathrm{pf},\kappa}(\tilde{\varphi})=\tilde{\mu}+\kappa\tilde{\varphi}\in L^{2}([0,T);H^{1}(\Omega))$. Hence, choose $\tilde{\varphi}_{\alpha}$ such that 
    $\mathrm{D}\mathcal{E}_{\mathrm{pf},\kappa}(\tilde{\varphi}_{\alpha})=\alpha \mathrm{D}\mathcal{E}_{\mathrm{pf},\kappa}(\tilde{\varphi})\in L^{2}([0,T);H^{1}(\Omega))$ with $\alpha>0$ (where the existence of such $\tilde{\varphi}_{\alpha}$ can be guaranteed by the surjectivity of the subdifferential of a proper, convex, lower semicontinuous and coercive functional, see~\cite[][Chapter 2.2]{zbMATH05948485} for details) and multiply both sides of~\eqref{simplied relative energy estimate when velocity and stress are zero} by $\frac{1}{\alpha}$. This gives
    \begin{equation}
        \begin{aligned}
        -&\frac{1}{\alpha}\int_{0}^{T}\phi^{\prime}\int_{\Omega}
        \rho\frac{\abs{v}^{2}}{2}
        +\frac{\abs{S}^{2}}{2} \dd{x}\dd{t}
        -\int_{\Omega}\rho_{0}\frac{\abs{v_{0}}^{2}}{2}
        +\frac{\abs{S_{0}}^{2}}{2}\dd{x}
        \\
        +&\frac{1}{\alpha}\int_{0}^{T}\phi\int_{\Omega}
		2\nu(\varphi)\abs{\sym{\nabla v}}^{2}
        +\gamma\abs{\nabla S}^{2}\dd{x}\dd{t}
        -\int_{0}^{T}\phi\langle f, v\rangle_{H^{1}}\dd{t}
        +\int_{0}^{T}\phi\mathcal{P}(\varphi;S)\dd{t}
        \\
        +&\frac{1}{\alpha}\int_{0}^{T}\phi\langle \partial_{t}\varphi,\mu+\kappa\varphi\rangle\dd{t}
        +\int_{0}^{T}\phi\int_{\Omega}\kappa v\cdot\nabla\varphi\varphi
        +\nabla\mu(\nabla\mu+\kappa\nabla\varphi)\dd{x}
        \\
        -&\int_{0}^{T}\phi\langle\partial_{t}\varphi,\tilde{\mu}+\kappa\tilde{\varphi}\rangle_{H^{1}}\dd{t}
        -\int_{0}^{T}\phi\int_{\Omega}v\cdot\nabla\varphi(\tilde{\mu}+\kappa\tilde{\varphi}) 
        +\nabla\mu\cdot(\nabla\tilde{\mu}+\kappa\nabla\tilde{\varphi})
		\dd{x}\dd{t}
        \leq0.
     \end{aligned}
    \end{equation}
    Letting $\alpha\to\infty$ gives
    \begin{equation}
        -\int_{0}^{T}\phi\langle\partial_{t}\varphi,\tilde{\mu}+\kappa\tilde{\varphi}\rangle_{H^{1}}\dd{t}
        -\int_{0}^{T}\phi\int_{\Omega}v\cdot\nabla\varphi(\tilde{\mu}+\kappa\varphi) 
        +\nabla\mu\cdot(\nabla\tilde{\mu}+\kappa\nabla\varphi)
		\dd{x}\dd{t}
        \leq0
    \end{equation}
    for all $\phi\in\tilde{C}([0,T])$. By Lemma~\ref{result related to energy estimate}, we conclude
    \begin{equation}
        \int_{0}^{T}\langle\partial_{t}\varphi,\tilde{\mu}+\kappa\tilde{\varphi}\rangle_{H^{1}}\dd{t}
        +\int_{0}^{T}\int_{\Omega}v\cdot\nabla\varphi(\tilde{\mu}+\kappa\varphi) 
        +\nabla\mu\cdot(\nabla\tilde{\mu}+\kappa\nabla\varphi)
		\dd{x}\dd{t}
        \geq0.
    \end{equation}
    This inequality is linear with respect to $\tilde{\mu}+\kappa\tilde{\varphi}$. Thus, choosing $\tilde{\varphi}_{-}$ such that $\mathrm{D}\mathcal{E}_{\mathrm{pf},\kappa}(\tilde{\varphi}_{-})=-\mathrm{D}\mathcal{E}_{\mathrm{pf},\kappa}(\tilde{\varphi})$ and integrating by parts in time yields the desired equality.
\end{proof}

\begin{prop}\label{energy-variational solution satisfies the weak formulation of momentum balance}
     Suppose that the regularity weight $\mathcal{K}$ satisfies $\mathcal{K}(\tilde{v},0,0)=0$ for all $\tilde{v}\in C_{0,\mathrm{div}}^{\infty}(\Omega\times[0,T))$. Moreover, assume that $\partial_{t}\varphi\in L^{2}([0,T);H^{-1}(\Omega
     ))$ and that the weak formulation
    \eqref{weak first CH TPP}
     holds true  for all $\zeta$ such that $\zeta\in L^{2}([0,T);H^{1}(\Omega))$, $\partial_{t}\zeta
     \in L^{2}([0,T);L^{2}(\Omega))$ and $\zeta(T)=0$. Then a dissipative solution of type $\mathcal{K}$ is a weak solution of the momentum balance, i.e. \eqref{weak velocity TPP} is satisfied for all $\Phi\in C_{0,\mathrm{div}}^{\infty}(\Omega\times[0,T))$.
\end{prop}

\begin{proof}
    This proposition can be shown by following the idea of the proof of Proposition~\ref{energy-variational solution satisfies the weak formulation of Cahn-Hilliard equation}. See also~\cite[][Proposition 4.3]{eiterWeakstrongUniquenessEnergyvariational2022b}.
\end{proof}

Next, we establish that dissipative solutions are transitive with respect to the regularity weights, given that the regularity weights satisfy a certain monotonicity.

\begin{prop}\label{monotonicity of energy variational solution}
	Let $(v,S,\varphi,\mu)$ be a dissipative solution of type $\mathcal{K}$. Let $\mathcal{K}$, $\mathcal{L}$ be regularity weights with the property $\mathcal{K}(\tilde{v},\tilde{S},\tilde{\varphi})\leq\mathcal{L}(\tilde{v},\tilde{S},\tilde{\varphi})$ for all $(\tilde{v},\tilde{S},\tilde{\varphi})\in \mathfrak{T}$ and a.e. $t\in(0,T)$. Then $(v,S,\varphi,\mu)$ is also a dissipative solution of type $\mathcal{L}$.
\end{prop}

\begin{proof}
    This proposition can be proved by applying Lemma~\ref{result related to energy estimate} to~\eqref{relative energy estimate for two phase problem} with the test function
    \begin{equation*}
        \Phi(t):=\phi(t)\exp(-\int_{0}^{t}\mathcal{L}(\tilde{v},\tilde{S},\tilde{\varphi})-\mathcal{K}(\tilde{v},\tilde{S},\tilde{\varphi})\dd{\tau})
    \end{equation*}
    for $\phi\in\tilde{C}([0,T])$ and using Lemma \ref{result related to energy estimate} again to cancel $\phi$. See also~\cite[][Proposition 4.4]{eiterWeakstrongUniquenessEnergyvariational2022b}. Notice that $\Phi$ is only a valid test function for $\mathcal{K}(\tilde{v},\tilde{S},\tilde{\varphi})\leq\mathcal{L}(\tilde{v},\tilde{S},\tilde{\varphi}) $.
\end{proof}
    \section{Global existence result for the regularized two phase system $\gamma>0$}\label{section Global existence result for the regularized two phase system}

\subsection{Implicit time discretization}\label{Implicit time discretization}
In this section, we will use an implicit time discretization to show the existence of weak solutions.

To start with, we first define another dissipation potential $\widetilde{\mathcal{P}}$ as
\begin{align}
    \begin{split}
    \widetilde{\mathcal{P}}:L^{2}(\Omega)\times H_{\mathrm{sym,Tr}}^{1}(\Omega)&\to[0,+\infty]
    \\
    \widetilde{\mathcal{P}}(\varphi;S)&:=\begin{cases}
        \mathcal{P}(\varphi;S)&(\varphi,S)\in L^{2}(\Omega)\times H_{\mathrm{sym,Tr}}^{1}(\Omega)\cap\dom{\mathcal{P}},\\
        +\infty&\mbox{otherwise}.
    \end{cases}
    \end{split}
    \label{restriction of dissipation potential}
\end{align} 
$\widetilde{\mathcal{P}}$ can be viewed as the restriction of $\mathcal{P}$ in $L^{2}(\Omega)\times H_{\mathrm{sym,Tr}}^{1}(\Omega)$. Notice that $\widetilde{\mathcal{P}}$ is proper with $\widetilde{\mathcal{P}}(\varphi;0)=0$ for all $\varphi \in L^2(\Omega)$. Moreover, 
     for all $\varphi \in L^2(\Omega)$, the mapping $S \mapsto \widetilde{\mathcal{P}}(\varphi; S) $ is convex and lower semicontinuous in $H_{\mathrm{sym,Tr}}^{1}(\Omega)$. 
We write $\dom{\partial\widetilde{\mathcal{P}}(\varphi;\cdot)}$ to represent the domain of the convex partial subdifferential.

To begin with the time discretization, let $h=\frac{T}{N}$ for $N \in \mathbb{N}$ and let $t_{0}=0$, $t_{k}=kh$, $t_{N}=T$ and assume that the initial data $v_{0}, S_{0}, \varphi_{0}$ satisfy Assumption \ref{assumptions on the initial data}. For all $k=0,1,\ldots,N-1$, let $v_{k}\in L_{\mathrm{div}}^{2}(\Omega)$, $S_{k}\in L_{\mathrm{sym,Tr}}^{2}(\Omega)$, $\varphi_{k}\in H^1(\Omega)$ with $W^{\prime}(\varphi_{k})\in L^2(\Omega)$ and $\rho_{k}=\frac{1}{2}(\rho_{1}+\rho_{2})+\frac{1}{2}(\rho_{2}-\rho_{1})\varphi_{k}$.  Moreover, let $f_{k+1}=h^{-1}\int_{t_{k}}^{t_{k+1}}f\dd{\tau}$. We determine $(v_{k+1},S_{k+1},\varphi_{k+1},\mu_{k+1})$ based on the given data with
\begin{subequations}\label{time discrete problem}
\begin{align*}
	J_{k+1}=-\frac{\rho_{2}-\rho_{1}}{2}\nabla\mu_{k+1},
\end{align*}
i.e., we aim to find some $(v_{k+1},S_{k+1},\varphi_{k+1},\mu_{k+1})$ such that 
\begin{equation}
	v_{k+1}\in H_{0,\mathrm{div}}^{1}(\Omega),S_{k+1}\in H_{\mathrm{sym,Tr}}^{1}(\Omega)\cap\dom{\partial\widetilde{\mathcal{P}}(\varphi_{k};\cdot)},
	\varphi_{k+1}\in \dom{\mathrm{D} \mathcal{E}_{\mathrm{pf},\kappa}},\mu_{k+1}\in H_{\vec{n}}^{2}(\Omega),
\end{equation}
where 
\begin{equation}
    H_{\vec{n}}^{2}(\Omega):=\{u\in H^{2}(\Omega):\vec{n}\cdot\nabla u|_{\partial\Omega}=0\},    
\end{equation}
satisfying:\\
1. The weak formulation of discrete momentum balance:
\begin{equation}
    \begin{aligned}
	\langle \frac{\rho_{k+1} v_{k+1}-\rho_{k}v_{k}}{h}+ \divge{\rho_{k}v_{k+1}\otimes v_{k+1}},\Phi \rangle_{L^{2}}&
	\\
	+\langle 2\nu(\varphi_{k})\sym{\nabla v_{k+1}},\sym{\nabla\Phi} \rangle_{L^{2}} 
	-\langle \divge{\eta(\varphi_{k})S_{k+1}},\Phi \rangle_{L^{2}}&
	\\
	+\langle (\divge{J_{k+1}}-\frac{\rho_{k+1}-\rho_{k}}{h}-v_{k+1}\cdot\nabla\rho_{k})\frac{v_{k+1}}{2},\Phi \rangle_{L^{2}}&
	\\
    +\langle J_{k+1}\cdot\nabla v_{k+1}
    -\mu_{k+1}\nabla\varphi_{k},\Phi \rangle_{L^{2}}&=\langle f_{k+1},\Phi\rangle_{H^{1}}
	\label{discrete velocity}
\end{aligned}
\end{equation}
for all $\Phi\in C_{0,\mathrm{div}}^{\infty}(\Omega)$.\\
2. The weak formulation of the discrete evolution law for the stress:
\begin{equation}
\begin{aligned}
	\langle \frac{S_{k+1}-S_{k}}{h} 
    + (v_{k+1}\cdot\nabla S_{k+1}),\Psi \rangle_{L^{2}}&
	\\
	+\langle (S_{k+1}\skw{\nabla v_{k+1}}-\skw{\nabla v_{k+1}}S_{k+1}),\Psi \rangle_{L^{2}}& 
	\\
	+\langle \xi_{k+1}^{k},\Psi\rangle_{H_{\mathrm{sym,Tr}}^{1}}
    +\langle \gamma\nabla S_{k+1},\nabla\Psi \rangle_{L^{2}}
    &=\langle \eta(\varphi_{k})\sym{\nabla v_{k+1}},\Psi\rangle_{L^{2}} 
	\label{discrete stress}
\end{aligned}
\end{equation}
for all $\Psi\in C_{\mathrm{sym,Tr}}^{\infty}(\bar{\Omega})$, and here $\xi_{k+1}^{k}\in\partial\widetilde{\mathcal{P}}(\varphi_{k};S_{k+1})\subseteq (H_{\mathrm{sym,Tr}}^{1}(\Omega))^{\prime}$.\\
3. The discrete evolution law for the phase-field variable:
\begin{equation}
	\frac{\varphi_{k+1}-\varphi_{k}}{h}+v_{k+1}\cdot\nabla\varphi_{k}=\Delta\mu_{k+1}, 
	\label{discrete first CH}
\end{equation}
as well as
\begin{equation}
	\mu_{k+1}+\kappa\frac{\varphi_{k+1}+\varphi_{k}}{2}=-\Delta\varphi_{k+1}+W_{\kappa}^{\prime}(\varphi_{k+1}), 
	\label{discrete second CH}
\end{equation}
almost everywhere in $\Omega$.
\end{subequations}
\begin{remark}
	(1) Integrating \eqref{discrete first CH} in space and performing an integration by parts gives 
	\begin{align}
		\int_{\Omega}\varphi_{k+1} \dd{x}
		&=\int_{\Omega}\varphi_{k}\dd{x}+h\int_{\Omega}-v_{k+1}\cdot\nabla\varphi_{k}+\Delta\mu_{k+1} \dd{x}
		=\int_{\Omega}\varphi_{k}\dd{x},
	\end{align}
	where we made use of the Neumann boundary condition \eqref{two phase system alternative neumann boundary} and the fact that $v_{k+1}$ is divergence-free. Moreover, this equality implies that $\int_{\Omega}\varphi_{k}\dd{x}=\int_{\Omega}\varphi_{0}\dd{x}$ is a constant, i.e., the total volume conserved.\\
	(2) By multiplying \eqref{discrete first CH} with $-\frac{1}{2}(\rho_{2}-\rho_{1})$, we obtain that
	\begin{align}\label{evolution of discrete rho}
		-\frac{\rho_{k+1}-\rho_{k}}{h}-v_{k+1}\nabla\rho_{k}=\divge{J_{k+1}}.
	\end{align}
	Notice that $\divge{v_{k+1}\otimes J_{k+1}}=(\divge{J_{k+1}})v_{k+1}+J_{k+1}\cdot\nabla v_{k+1}$. Inserting \eqref{evolution of discrete rho} into \eqref{discrete velocity} results in
	\begin{equation}
	    \begin{aligned}
		\langle \frac{\rho_{k+1} v_{k+1}-\rho_{k}v_{k}}{h}+\divge{\rho_{k}v_{k+1} \otimes v_{k+1}},\Phi \rangle_{L^{2}}&
	   \\
	   +\langle 2\nu(\varphi_{k})\sym{\nabla v_{k+1}},\sym{\nabla\Phi}\rangle_{L^{2}}
	   -\langle \divge{\eta(\varphi_{k})S_{k+1}},\Phi\rangle_{L^{2}}&
	   \\
	   +\langle \divge{v_{k+1}\otimes J_{k+1}},\Phi\rangle_{L^{2}}
            -\langle f_{k+1},\Phi\rangle_{H^{1}}  &
            =\langle \mu_{k+1}\nabla\varphi_{k},\Phi\rangle_{L^{2}}, 
		\label{equivalent discrete velocity}
	\end{aligned}
	\end{equation}
    which is the direct weak formulation of the momentum balance \eqref{two phase system alternative velocity} with time discretization.
\end{remark}

Before we prove the existence of solutions for the time discrete problem \eqref{time discrete problem}, let us first deduce an estimate for the terms in the Cahn--Hilliard part. This inequality uses a similar method as in the proof of \cite[][Lemma 4.2]{abelsExistenceWeakSolutions2013}.
\begin{lem}\label{estimate from discrete second CH}
	Assume that $\varphi_{k+1}\in\dom{\mathrm{D} \mathcal{E}_{\mathrm{pf},\kappa}}$ and $\mu_{k+1}\in H^{1}(\Omega)$ are solutions to \eqref{discrete second CH} for given $\varphi_{k}\in H^{2}(\Omega)$ satisfying $\abs{\varphi_{k}}\leq1$ in $\Omega$ and
	\begin{align*}
		\frac{1}{\abs{\Omega}}\int_{\Omega}\varphi_{k+1} \dd{x}
		=\frac{1}{\abs{\Omega}}\int_{\Omega}\varphi_{k} \dd{x}
		\in(-1,1).
	\end{align*}
	Then there exists a positive constant $C$ depending on $\int_{\Omega}\varphi_{k}\dd{x}$, such that
	\begin{align}
		\lVert W_{\kappa}^{\prime}(\varphi_{k+1})\rVert_{L^2}+\abs{\int_{\Omega}\mu_{k+1} \dd{x}}
		&\leq C\left(\lVert\nabla\mu_{k+1}\rVert_{L^2}+\lVert\nabla\varphi_{k+1}\rVert_{L^2}^2+\lVert\nabla\varphi_{k}\rVert_{L^2}^{2}+1\right)
		\label{estimate on the discrete version of the derivative of double well potential and mean value of mu}\\
		\lVert\mathrm{D} \mathcal{E}_{\mathrm{pf},\kappa}(\varphi_{k+1})\rVert_{L^2}
		&\leq C\left(\lVert\mu_{k+1}\rVert_{L^2}+1\right)
		\label{estimate on the discrete version of cahn hilliard subdifferential}
	\end{align}
\end{lem}
\begin{proof}
	Recall from \eqref{mean value in Omega} that $u_{\Omega}$ denotes the mean value of a function $u$ in $\Omega$. Hence, for $\varphi_{k+1}$, we write $\varphi_{k+1,\Omega}$ for the mean value. First, we test \eqref{discrete second CH} with $\zeta=(\varphi_{k+1}-{\varphi}_{k+1,\Omega})$ to obtain that
	\begin{align}
		&\int_{\Omega}\mu_{k+1}(\varphi_{k+1}-{\varphi}_{k+1,\Omega})\dd{x}
		+\int_{\Omega}\kappa\frac{\varphi_{k+1}+\varphi_{k}}{2}(\varphi_{k+1}-{\varphi}_{k+1,\Omega})\dd{x}
		\nonumber\\
		=-&\int_{\Omega}\Delta\varphi_{k+1}\cdot(\varphi_{k+1}-{\varphi}_{k+1,\Omega})\dd{x}
		+\int_{\Omega}W_{\kappa}^{\prime}(\varphi_{k+1})(\varphi_{k+1}-{\varphi}_{k+1,\Omega})\dd{x}.
		\label{tested discrete second CH}
	\end{align}
	Writing $\mu_{0}=\mu_{k+1}-{\mu}_{k+1,\Omega}$, we can see that  
	\begin{align}
		\int_{\Omega}\mu_{k+1}(\varphi_{k+1}-{\varphi}_{k+1,\Omega})\dd{x}
		=\int_{\Omega}\mu_{0}\varphi_{k+1} \dd{x}.
		\label{change mean value of phi to mean value of mu}
	\end{align}
    In view of the homogeneous Neumann boundary condition, notice that
	\begin{align}
		\abs{ -\int_{\Omega}\Delta\varphi_{k+1}(\varphi_{k+1}-{\varphi}_{k+1,\Omega})\dd{x}}
		=\abs{ \int_{\Omega}\nabla\varphi_{k+1}\cdot\nabla\varphi_{k+1} \dd{x}}=\lVert\nabla\varphi_{k+1}\rVert_{L^2}^{2}.
	\end{align}
	Besides, by the assumption $W_{\kappa}^{\prime}\in C^1(-1,1)$ and $\lim\limits_{s\to\pm1}W_{\kappa}^{\prime}(s)=\pm\infty$ in~\eqref{assumption on double well potential W}, we obtain that
	\begin{align}
		W_{\kappa}^{\prime}(\varphi_{k+1})(\varphi_{k+1}-\varphi_{k+1,\Omega})\geq C\abs{W_{\kappa}^{\prime}(\varphi_{k+1})}-\tilde{C}
		\label{pointwise lower bound on convex double well potential}
	\end{align}
	for all $\varphi_{k+1}\in[-1,1]$, see \cite[][Lemma 4.2]{abelsExistenceWeakSolutions2013} for details. Inserting \eqref{change mean value of phi to mean value of mu}-\eqref{pointwise lower bound on convex double well potential} into \eqref{tested discrete second CH} yields 
	\begin{align}
		\int_{\Omega}\abs{W_{\kappa}^{\prime}(\varphi_{k+1})}\dd{x}
		\leq& C(\lVert\mu_{0}\rVert_{L^2}\lVert\varphi_{k+1}\rVert_{L^2}
		+\int_{\Omega}\kappa\frac{\varphi_{k+1}+\varphi_{k}}{2}(\varphi_{k+1}-\varphi_{k+1,\Omega})\dd{x}
		+\lVert\nabla\varphi_{k+1}\rVert_{L^2}^{2}
		+1)
		\nonumber\\
		\leq& C\left(\lVert\nabla\mu\rVert_{L^2}+\lVert\nabla\varphi_{k+1}\rVert_{L^2}^{2}+\lVert\nabla\varphi_{k}\rVert_{L^2}^{2}+1\right),
	\end{align}
	where we used the fact that $\abs{\varphi_{k+1}},\abs{\varphi_{k}}\leq1$ and Poinar\'e's inequality. Next, by directly integrating \eqref{discrete second CH}, we can see that
	\begin{align*}
		\int_{\Omega}\mu_{k+1} \dd{x}
		+\int_{\Omega}\kappa\frac{\varphi_{k+1}+\varphi_{k}}{2}\dd{x}
		=\int_{\Omega}-\Delta\varphi_{k+1} \dd{x}
		+\int_{\Omega}W_{\kappa}^{\prime}(\varphi_{k+1})\dd{x}.
	\end{align*}
	This implies
	\begin{align}
		\abs{ \int_{\Omega}\mu_{k+1} \dd{x}}
		&\leq\int_{\Omega}\abs{W_{\kappa}^{\prime}(\varphi_{k+1})}\dd{x}
		+\int_{\Omega}\kappa \abs{\frac{\varphi_{k+1}+\varphi_{k}}{2}}\dd{x}
		\nonumber\\
		&\leq C\left(\lVert\nabla\mu_{k+1}\rVert_{L^2}+\lVert\nabla\varphi_{k+1}\rVert_{L^2}^{2}+\lVert\nabla\varphi_{k}\rVert_{L^2}^{2}+1\right),
		\label{estimate on the mean value of discrete mu}
	\end{align}
	where we used integration by parts,  Poincar\'e's inequality, and the fact that
	\begin{align*}
		\frac{1}{\abs{\Omega}}\int_{\Omega}\varphi_{k+1}\dd{x}=\frac{1}{\abs{\Omega}}\int_{\Omega}\varphi_{k}\dd{x}.
	\end{align*}  
	Finally, since $\mathrm{D} \mathcal{E}_{\mathrm{pf},\kappa}(\varphi_{k+1})=-\Delta\varphi_{k+1}+W_{\kappa}^{\prime}(\varphi_{k+1})=\mu_{k+1}+\frac{\kappa}{2}(\varphi_{k+1}+\varphi_{k})$, we obtain
	\begin{align}
		\lVert \mathrm{D} \mathcal{E}_{\mathrm{pf},\kappa}(\varphi_{k+1})\rVert_{L^2}
		\leq\lVert \mu_{k+1}\rVert_{L^2}+\frac{\kappa}{2}(\lVert\varphi_{k+1}\rVert_{L^2}+\lVert\varphi_{k}\rVert_{L^2})
		\leq C\left(\lVert\mu_{k+1}\rVert_{L^2}+1\right).
	\end{align}
	Besides, using \eqref{estimate on element in subdifferntial domain}, we deduce that
	\begin{align*}
		\lVert W_{\kappa}^{\prime}(\varphi_{k+1})\rVert_{L^2}^{2}
		\leq C\left(\lVert 
        \mathrm{D} \mathcal{E}_{\mathrm{pf},\kappa}(\varphi_{k+1})\rVert_{L^2}^2+\lVert\varphi_{k+1}\rVert_{L^2}^{2}+1\right)
		\leq C\left(\lVert\mu_{k+1}\rVert_{L^2}+1\right)^2,
	\end{align*}
	which implies
	\begin{align}
		\lVert W_{\kappa}^{\prime}(\varphi_{k+1})\rVert_{L^2}&\leq C\left(\lVert\mu_{k+1}\rVert_{L^2}+1\right)
		\leq C\left(\lVert\nabla\mu_{k+1}\rVert_{L^2}+\abs{\int_{\Omega}\mu_{k+1} \dd{x}}+1\right).
		\label{L2 upper bound on convex double well potential}
	\end{align}
	Therefore, combining \eqref{L2 upper bound on convex double well potential} with \eqref{estimate on the mean value of discrete mu}  yields the desired inequality, that is 
	\begin{align*}
		\lVert W_{\kappa}^{\prime}(\varphi_{k+1})\rVert_{L^2}+\abs{\int_{\Omega}\mu_{k+1} \dd{x}}
		\leq&C\left(\lVert\nabla\mu_{k+1}\rVert_{L^2}+\lVert\nabla\varphi_{k+1}\rVert_{L^2}^2+\lVert\nabla\varphi_{k}\rVert_{L^2}^{2}+1\right).
	\end{align*}
\end{proof}

Now, we show the existence of solutions to the time discrete problem \eqref{time discrete problem}. We adapt the proof of \cite[][Lemma 4.3]{abelsExistenceWeakSolutions2013} to our case. Notice that in~\cite[][Lemma 4.3]{abelsExistenceWeakSolutions2013},  the extra stress tensor $S$ is not present, while, below the stress tensor $S$ will be the main difficulty because of the set-valued subdifferential.
\begin{lem}[Existence of solutions to the time discrete problem]\label{Existence of solution to the time discrete problem}
	For $k\in\{0,1,\ldots,N-1\}$, let $v_{k}\in L_{\mathrm{div}}^{2}(\Omega)$, $S_{k}\in L_{\mathrm{sym,Tr}}^{2}(\Omega)$, $\varphi_{k}\in H^{2}(\Omega)$ with $\abs{\varphi_{k}}\leq1$ and $\rho_{k}=\frac{\rho_{2}-\rho_{1}}{2}\varphi_{k}+\frac{\rho_{2}+\rho_{1}}{2}$ be given, let $\widetilde{\mathcal{P}}$ be as in \eqref{restriction of dissipation potential}, and set
    \begin{equation}\label{fixed point theorem preimage space}
        X:=H_{0,\mathrm{div}}^{1}(\Omega)\times H_{\mathrm{sym,Tr}}^{1}(\Omega)\cap\dom{\partial\widetilde{\mathcal{P}}(\varphi_{k};\cdot)}\times\dom{\mathrm{D}\mathcal{E}_{\mathrm{pf},\kappa}}\times H_{\vec{n}}^{2}(\Omega).
    \end{equation}
    Then there exists a quadruplet $(v_{k+1},S_{k+1},\varphi_{k+1},\mu_{k+1})\in X$ solving \eqref{discrete velocity}-\eqref{discrete second CH}. Moreover, this solution satisfies the energy dissipation estimate
    \begin{equation}
        \begin{aligned}
            &\mathcal{E}_{\mathrm{tot}}(v_{k+1},S_{k+1},\varphi_{k+1})
            +h\mathcal{D}_{\mathrm{chs}}(v_{k+1},\varphi_{k},\mu_{k+1})+h\mathcal{D}_{\mathrm{sd},\gamma}(S_{k+1})+h\langle \xi_{k+1}^{k},S_{k+1} \rangle_{H_{\mathrm{sym,Tr}}^{1}}
            \\
            \leq&\mathcal{E}_{\mathrm{tot}}(v_{k},S_{k},\varphi_{k})
        +h\langle f_{k+1}, v_{k+1}\rangle_{H^{1}}.
		\label{discrete total energy estimate}
        \end{aligned}
    \end{equation}
\end{lem}
\begin{proof}
	\textbf{Step 1: A priori estimate~\eqref{discrete total energy estimate}.}
	Let $(v_{k+1},S_{k+1},\varphi_{k+1},\mu_{k+1})\in  H_{0,\mathrm{div}}^{1}(\Omega)\times H_{\mathrm{sym,Tr}}^{1}(\Omega)\cap \dom{\partial\widetilde{\mathcal{P}}(\varphi_{k};\cdot)}\times\dom{\mathrm{D} \mathcal{E}_{\mathrm{pf},\kappa}}\times H_{\vec{n}}^{2}(\Omega)$ be a solution to the time discrete problem \eqref{time discrete problem}. Observe that $v_{k+1}\in H_{0,\mathrm{div}}^{1}(\Omega)$ is a suitable test function for \eqref{discrete velocity}. With this in mind, we calculate that
	\begin{equation}
		\int_{\Omega}\left((\divge{J_{k+1}})\frac{v_{k+1}}{2}+J_{k+1}\cdot\nabla v_{k+1}\right)v_{k+1}\dd{x}
		=\int_{\Omega}\divge{J_{k+1}\frac{|v_{k+1}|^2}{2}}\dd{x}=0,
		\label{eqaulity involving discrete divergence of J k+1}
	\end{equation}
	as well as
	\begin{equation}
        \begin{aligned}
		&\int_{\Omega}\left(\divge{\rho_{k}v_{k+1}\otimes v_{k+1}}-(\nabla\rho_{k}\cdot v_{k+1})\frac{v_{k+1}}{2}\right)v_{k+1} \dd{x}
		\\
		=&\int_{\Omega}\left(\divge{\rho_{k}v_{k+1}\otimes v_{k+1}}-\divge{\rho_{k}v_{k+1}}\frac{v_{k+1}}{2}\right)v_{k+1} \dd{x}
		\\
		=&\int_{\Omega}\left(\divge{\rho_{k}v_{k+1}}|v_{k+1}|^2+\rho_{k}v_{k+1}\cdot\nabla(\frac{|v_{k+1}|^2}{2})-\divge{\rho_{k}v_{k+1}}\frac{|v_{k+1}|^2}{2}\right)\dd{x}
		\\
		=&\int_{\Omega}\divge{\rho_{k}v_{k+1}\frac{|v_{k+1}|^2}{2}}\dd{x}=0,
		\label{eqaulity involving discrete rho k+1 and v k+1 otimes v k+1}
	   \end{aligned}
    \end{equation}
	where we used integration-by-parts and the homogeneous boundary condition \eqref{two phase system alternative homogeneous boundary}. Moreover, notice that
	\begin{equation}
    \begin{aligned}
		&\frac{1}{h}(\rho_{k+1} v_{k+1}-\rho_{k}v_{k})v_{k+1}
		\\
		=&\frac{1}{h}(\rho_{k+1}-\rho_{k})\abs{v_{k+1}}^{2}
		+\frac{1}{h}\rho_{k}(v_{k+1}-v_{k})v_{k+1}
		\\
		=&\frac{1}{h}(\rho_{k+1}-\rho_{k})\abs{v_{k+1}}^{2}
		+\frac{1}{h}\rho_{k}\frac{\abs{v_{k+1}}^2}{2}
		-\frac{1}{h}\rho_{k}\frac{\abs{v_{k}}^2}{2}
		+\frac{1}{h}\rho_{k}\frac{\abs{v_{k+1}-v_{k}}^2}{2}
		\\
		=&\frac{1}{h}\frac{\rho_{k+1}\abs{v_{k+1}}^2}{2}-\frac{1}{h}\frac{\rho_{k}\abs{v_{k}}^2}{2}+\frac{1}{h}\frac{(\rho_{k+1}-\rho_{k})\abs{v_{k+1}}^2}{2}+\frac{1}{h}\frac{\rho_{k}\abs{v_{k+1}-v_{k}}^2}{2}.
		\label{equality involving difference bewteen discrete momemtum k+1 and momentum k}
	\end{aligned}
    \end{equation}
	Thanks to \eqref{eqaulity involving discrete divergence of J k+1}, \eqref{eqaulity involving discrete rho k+1 and v k+1 otimes v k+1} and \eqref{equality involving difference bewteen discrete momemtum k+1 and momentum k}, by testing \eqref{discrete velocity} with $\Phi=v_{k+1}$, we obtain
	\begin{equation}
	\begin{aligned}
		&\int_{\Omega}\frac{1}{h}\frac{\rho_{k+1}|v_{k+1}|^{2}}{2} \dd{x}
		+\int_{\Omega}\frac{1}{h}\frac{\rho_{k}|v_{k+1}-v_{k}|^{2}}{2} \dd{x}
		\\
		&+\int_{\Omega}2\nu(\varphi_{k})|\sym{\nabla v_{k+1}}|^{2} \dd{x}
		+\int_{\Omega}\eta(\varphi_{k})S_{k+1}:\sym{\nabla v_{k+1}} \dd{x} 
		\\
		=&\int_{\Omega}\mu_{k+1}(\nabla\varphi_{k}\cdot v_{k+1}) \dd{x}
		+\int_{\Omega}\frac{1}{h}\frac{\rho_{k}|v_{k}|^2}{2}\dd{x}
        +\langle f_{k+1}, v_{k+1}\rangle_{H^{1}}.
		\label{discrete velocity energy}
	\end{aligned}
	\end{equation}
	Next, observe that $S_{k+1}\in H_{\mathrm{sym,Tr}}^{1}(\Omega)$ is a suitable test function for \eqref{discrete stress}. For the first term in~\eqref{discrete stress}, we calculate that 
    \begin{equation}
		\frac{1}{h}(S_{k+1}-S_{k}):S_{k+1}
		=\frac{1}{h}\frac{\abs{S_{k+1}}^2}{2}-\frac{1}{h}\frac{\abs{S_{k}}^2}{2}+\frac{1}{h}\frac{\abs{S_{k+1}-S_{k}}^2}{2}.
		\label{equality involving differnece between discete stress k+1 and stress k}
	\end{equation}
	With the help of \eqref{equality involving differnece between discete stress k+1 and stress k} and due to the fact that many terms cancel out, since $v$ is divergence-free and $S$ is symmetric, we also obtain 
    \begin{equation}
    \begin{aligned}
		&\int_{\Omega}\frac{1}{h}\frac{|S_{k+1}|^2}{2}\dd{x}
		+\int_{\Omega}\frac{1}{h}\frac{|S_{k+1}-S_{k}|}{2}\dd{x}
		+\langle \xi_{k+1}^{k}:S_{k+1} \rangle_{H_{\mathrm{sym,Tr}}^{1}}
		+\int_{\Omega}\gamma|\nabla S_{k+1}|^{2} \dd{x}
		\\
		=&\int_{\Omega}\eta(\varphi_{k})\sym{\nabla v_{k+1}}:S_{k+1} \dd{x}
		+\int_{\Omega}\frac{1}{h}\frac{|S_{k}|^2}{2}\dd{x}.
		\label{discrete stess energy}
	\end{aligned}
    \end{equation}
	Further, observe that $\mu_{k+1}\in H_{\vec{n}}^{2}(\Omega)$ and $\frac{1}{h}(\varphi_{k+1}-\varphi_{k})\in H^{2}(\Omega)$ are suitable test functions for \eqref{discrete first CH} and \eqref{discrete second CH}, respectively. With this test, we obtain
	\begin{equation}
		\int_{\Omega}\frac{1}{h}(\varphi_{k+1}-\varphi_{k})\mu_{k+1} \dd{x}
		+\int_{\Omega}(v_{k+1}\cdot\nabla\varphi_{k})\mu_{k+1} \dd{x}
		=-\int_{\Omega}|\nabla\mu_{k+1}|^{2} \dd{x}, 
		\label{discrete firsr CH energy}
	\end{equation}
	as well as
	\begin{equation}
	\begin{aligned}
		&\int_{\Omega}\frac{1}{h}\nabla\varphi_{k+1}\cdot(\nabla\varphi_{k+1}-\nabla\varphi_{k})\dd{x}
		+\int_{\Omega}W_{0}^{\prime}(\varphi_{k+1})\frac{\varphi_{k+1}-\varphi_{k}}{h}\dd{x}
		\\
		=&\int_{\Omega}\mu_{k+1}\frac{\varphi_{k+1}-\varphi_{k}}{h}\dd{x}
		+\int_{\Omega}\kappa\frac{\varphi_{k+1}^{2}-\varphi_{k}^{2}}{2h}\dd{x}.
		\label{discrete second CH energy}
	\end{aligned}
	\end{equation}
	Now summing up \eqref{discrete velocity energy}-\eqref{discrete second CH energy}, we derive
    \begin{equation*}
    \begin{aligned}
		&\int_{\Omega}\frac{1}{h}\frac{\rho_{k}\abs{v_{k}}^2}{2} \dd{x}
		+\int_{\Omega}\frac{1}{h}\frac{\abs{S_{k}}^2}{2} \dd{x}
        +\langle f_{k+1},v_{k+1}\rangle_{H^{1}}
		\\
		=&\int_{\Omega}\frac{1}{h}\frac{\rho_{k+1}\abs{v_{k+1}}^2}{2} \dd{x}
		+\int_{\Omega}\frac{1}{h}\frac{\rho_{k}\abs{v_{k+1}-v_{k}}^2}{2} \dd{x}
		+\int_{\Omega}\frac{1}{h}\frac{\abs{S_{k+1}}^2}{2} \dd{x}
		+\int_{\Omega}\frac{1}{h}\frac{\abs{S_{k+1}-S_{k}}^2}{2} \dd{x}
		\\
		&+\int_{\Omega}2\nu(\varphi_{k})\abs{\sym{\nabla v_{k+1}}}^2 \dd{x}
		+\int_{\Omega}\gamma\abs{\nabla S_{k+1}}^2 \dd{x}
		+\langle \xi_{k+1}^{k},S_{k+1} \rangle_{H_{\mathrm{sym,Tr}}^{1}}
		+\int_{\Omega}\abs{\nabla\mu_{k+1}}^2 \dd{x}
		\\
		&+\int_{\Omega}W_{\kappa}^{\prime}(\varphi_{k+1})\frac{\varphi_{k+1}-\varphi_{k}}{h} \dd{x}
		-\int_{\Omega}\kappa\frac{\varphi_{k+1}^2-\varphi_{k}^2}{2h} \dd{x}
		+\frac{1}{h}\int_{\Omega}\nabla\varphi_{k+1}\cdot(\nabla\varphi_{k+1}-\nabla\varphi_{k}) \dd{x}
		\\
		\geq&\int_{\Omega}\frac{1}{h}\frac{\rho_{k+1}\abs{v_{k+1}}^2}{2} \dd{x}
		+\int_{\Omega}\frac{1}{h}\frac{\rho_{k}\abs{v_{k+1}-v_{k}}^2}{2} \dd{x}
		+\int_{\Omega}\frac{1}{h}\frac{\abs{S_{k+1}}^2}{2} \dd{x}
		+\int_{\Omega}\frac{1}{h}\frac{\abs{S_{k+1}-S_{k}}^{2}}{2} \dd{x}
		\\
		&+\int_{\Omega}2\nu(\varphi_{k})\abs{\sym{\nabla v}}^{2} \dd{x}
		+\int_{\Omega}\gamma\abs{\nabla S_{k+1}}^{2} \dd{x}
		+\langle\xi_{k+1}^{k}:S_{k+1} \rangle_{H_{\mathrm{sym,Tr}}^{1}}
		+\int_{\Omega}\abs{\nabla\mu_{k+1}}^{2} \dd{x}
		\\
		&+\frac{1}{h}\int_{\Omega}W_{\kappa}(\varphi_{k+1})-\kappa\frac{\varphi_{k+1}^{2}}{2} \dd{x}
		-\frac{1}{h}\int_{\Omega}W_{\kappa}(\varphi_{k})-\kappa\frac{\varphi_{k}^{2}}{2} \dd{x}
		\\
		&+\frac{1}{h}\int_{\Omega}\frac{\abs{\nabla\varphi_{k+1}-\nabla\varphi_{k}}^{2}}{2} \dd{x}
		+\frac{1}{h}\int_{\Omega}\frac{\abs{\nabla\varphi_{k+1}}^{2}}{2} \dd{x}
		-\frac{1}{h}\int_{\Omega}\frac{\abs{\nabla\varphi_{k}}^{2}}{2} \dd{x},
	\end{aligned}
    \end{equation*}
	where we have used the convexity of $W_{\kappa}$, i.e.,
	\begin{equation*}
		W_{\kappa}^{\prime}(\varphi_{k+1})(\varphi_{k+1}-\varphi_{k})
		\geq W_{\kappa}(\varphi_{k+1})
		-W_{\kappa}(\varphi_{k}),
	\end{equation*}
	and
	\begin{equation*}
		\nabla\varphi_{k+1}\cdot(\nabla\varphi_{k+1}-\nabla\varphi_{k})
		=\frac{\abs{\nabla\varphi_{k+1}}^{2}}{2}
		-\frac{\abs{\nabla\varphi_{k}}^2}{2}
		+\frac{\abs{\nabla\varphi_{k+1}-\nabla\varphi_{k}}^2}{2}.
	\end{equation*}
	Multiplying both sides by $h$ and rearranging terms to the left-hand side results in \eqref{discrete total energy estimate}.\\
	\paragraph{Step 2: Existence result via Schaefer's fixed-point theorem~\cite[][Chapter 9.2.2, Theorem 2]{evansPartialDifferentialEquations2010}.} Suppose that 
    \begin{equation*}
        \mathscr{K}_{k}:\tilde{Y}\to\tilde{Y}
    \end{equation*}
    is a continuous and compact mapping. Assume further that the set
    \begin{equation*}
        \left\{u\in\tilde{Y}:u=\lambda\mathscr{K}_{k}(u)\text{ for some }0\leq\lambda\leq1\right\}
    \end{equation*}
    is bounded. Then $\mathscr{K}_{k}$ has a fixed point. In order to apply Schaefer's fixed point theorem, we will determine $\tilde{Y}$ and the operator $\mathscr{K}_{k}$ based on the discrete weak formulation~\eqref{time discrete problem}. For this, using $X$ from~\eqref{fixed point theorem preimage space} and the space
	\begin{equation}\label{fixed point theorem image space}
	    Y:=\left(H_{0,\mathrm{div}}^{1}(\Omega)\right)^{\prime}
        \times\left(H_{\mathrm{sym,Tr}}^{1}(\Omega)\right)^{\prime}
        \times L^{2}(\Omega)
        \times L^{2}(\Omega),
	\end{equation}
    we define operators $\mathscr{L}_{k},\mathscr{F}_{k}:X\to Y$ as follows:
	\begin{equation}\label{operator L_k}
	    \mathscr{L}_{k}:X\to Y,\,w:=(v,S,\varphi,\mu)\mapsto\mathscr{L}_{k}(w):=\begin{pmatrix}
			L_{k}^{v}(v)\\
			L_{k}^{s}(S)\\
			-\Delta\mu+\int_{\Omega}\mu \dd{x}\\
			\varphi+\mathrm{D} \mathcal{E}_{\mathrm{pf},\kappa}(\varphi)
		\end{pmatrix}
	\end{equation}
	where $L_{k}^{v}(v)$ and $L_{k}^{s}(S)$ are defined in the weak sense, i.e.,
	\begin{align}
		\langle L_{k}^{v}(v),\Phi\rangle
        &=\langle 2\nu(\varphi_{k})\sym{\nabla v},\sym{\nabla\Phi} \rangle_{L^{2}}
        -\langle f_{k+1},\Phi \rangle_{H^{1}}
		\\
		\langle L_{k}^{s}(S),\Psi\rangle
        &=\langle \gamma S, \Psi \rangle_{H^{1}}  
        +\langle \xi^{k},\Psi \rangle_{H_{\mathrm{sym,Tr}}^{1}} 
        \text{ with }\xi^{k}\in\partial\widetilde{\mathcal{P}}(\varphi_{k};S)
	\end{align}
    for all test functions $\Phi\in H_{0,\mathrm{div}}^{1}(\Omega)$ and $\Psi\in H_{\mathrm{sym,Tr}}^{1}({\Omega})$, while the third and forth entries in\eqref{operator L_k} are identified pointwise. We further introduce $\mathscr{F}_{k}$ as follows 
	\begin{equation}
	    \mathscr{F}_{k}:X\to Y,\,w:=(v,S,\varphi,\mu)\mapsto\mathscr{F}_{k}(w),
	\end{equation}
    where
    \begin{equation}\label{operator F_k}
    \scalebox{0.93}{$\mathscr{F}_{k}(w):=\begin{pmatrix}
			-\frac{\rho v-\rho_{k}v_{k}}{h}-\divge{\rho_{k}v\otimes v}+\mu\nabla\varphi_{k}-\left(\divge{J}-\frac{\rho-\rho_{k}}{h}-v\cdot\nabla\rho_{k}\right)\frac{v}{2}-J\cdot\nabla v-\divge{\eta(\varphi_{k})S}
            \\
			-\frac{S-S_{k}}{h}-v\cdot\nabla S-S\skw{\nabla v}+\skw{\nabla v}S+\eta(\varphi_{k})\sym{\nabla v}+\gamma S
            \\
			-\frac{\varphi-\varphi_{k}}{h}-v\cdot\nabla\varphi_{k}+\int_{\Omega}\mu \dd{x}
            \\
			\varphi+\mu+\kappa\frac{\varphi+\varphi_{k}}{2}
		\end{pmatrix}.$}
    \end{equation}
	From these two definitions, we can see that $w=(v_{k+1},S_{k+1},\varphi_{k+1},\mu_{k+1})$ is a weak solution to \eqref{discrete velocity}-\eqref{discrete second CH} if and only if
	\begin{align}
		\mathscr{L}_{k}(w)-\mathscr{F}_{k}(w)=0.
		\label{weak solution in operator sense}
	\end{align}
    \paragraph{Properties of $\mathscr{L}_{k}$.}
    
	Now we want to prove the invertibility of the operator $\mathscr{L}_{k}$. For the first entry, we can derive the invertibility and continuity of the inverse with help of the Lax-–Milgram theorem. To show the invertibility, for all $\tilde{f}\in (H_{0,\mathrm{div}}^{1}(\Omega))^{\prime}$, we want to prove the existence of a unique $v\in H_{0,\mathrm{div}}^{1}(\Omega)$ such that $-\divge{2\nu(\varphi_{k})\sym{\nabla v}}-f_{k+1}=\tilde{f}$. Since $f_{k+1}\in H^{-1}(\Omega)\subseteq(H_{0,\mathrm{div}}^{1}(\Omega))^{\prime}$. There holds $\bar{f}:=\tilde{f}+f_{k+1}\in (H_{0,\mathrm{div}}^{1}(\Omega))^{\prime}$. Observing that the operator $-\divge{2\nu(\varphi_{k})\cdot}:H_{0,\mathrm{div}}^{1}(\Omega)\to(H_{0,\mathrm{div}}^{1}(\Omega))^{\prime}$ induces a continuous, coercive bilinear form, the Lax--Milgram theorem yields the invertibility and continuity of the inverse. For the second entry, notice that it can be viewed as the sum of a maximal monotone operator and the duality map. Hence, we can conclude the invertibility by Minty–-Browder Theorem see~\cite[][Theorem 2.2]{zbMATH05662098}. To see the continuity of the inverse, let $F_{n}=-\gamma\Delta S_{n}+\xi_{n}^{k}+\gamma S_{n}$, $F=-\gamma\Delta S+\xi^{k}+\gamma S$ and $F_{n}\to F$ in $(H^{1}(\Omega))^{\prime}$. For all $n\in\mathbb{N}$, observe that
    \begin{align*}
        \gamma\norm{S_{n}-S}_{H^{1}}^{2}
        &\leq \gamma\langle S_{n}- S, S_{n}- S \rangle_{H^{1}}
        +\langle \xi_{n}^{k}-\xi^{k} ,S_{n}-S \rangle_{H_{\mathrm{sym,Tr}}^{1}}
        \\
        &=\langle F_{n}-F , S_{n}-S \rangle_{H_{\mathrm{sym,Tr}}^{1}}
        \leq \frac{1}{2\gamma}\norm{F_{n}-F}_{(H_{\mathrm{sym,Tr}}^{1})^{\prime}}^{2}+\frac{\gamma}{2}\norm{S_{n}-S}_{H^{1}}^{2}.
    \end{align*} 
    By rearranging terms, one can see that the inverse operator is continuous. For the third entry, let us consider the following elliptic equation
	\begin{equation}
		\begin{cases}
			-\Delta u+\int_{\Omega}u \dd{x}=g&\mbox{ in }\Omega,
			\\
			\vec{n}\cdot\nabla u|_{\partial\Omega}=0&\mbox{ on }\partial\Omega,
		\end{cases}
		\label{operator equation}
	\end{equation} 
	where $g\in L^2(\Omega)$ is a given function. The invertibility of the operator represented by the third entry is equivalent to the existence of a unique weak solution $u\in H_{\vec{n}}^{2}(\Omega):=\{u\in H^{2}(\Omega):\vec{n}\cdot\nabla u|_{\partial\Omega}=0\}$ for any given $f\in L^{2}(\Omega)$ and this can be guaranteed by~\cite[][Chapter 2]{ellipticboundary}. Moreover, one can also derive that
    \begin{align}
		\lVert\mu\rVert_{H^{2}}\leq C\left(\lVert\mu\rVert_{H^{1}}+\lVert g\rVert_{L^2}\right).
		\label{estimate coming from regularity}
	\end{align}
    This gives the continuity of the inverse operator.
    For the last component of $\mathscr{L}_{k}$, notice that $\mathrm{D} \mathcal{E}_{\mathrm{pf},\kappa}$ is a maximal monotone operator. Again, by Minty–-Browder Theorem, we have the invertibility. Moreover, we want to derive the continuity of the inverse operator. To do so, we interpret the inverse operator as a mapping $L^{2}(\Omega)\to H^{2-s}(\Omega)$ for arbitrary $0<s<1/4$. Let $F_{k}=u_{k}+\mathrm{D} \mathcal{E}_{\mathrm{pf},\kappa}(u_{k})$ and $F=u+\mathrm{D} \mathcal{E}_{\mathrm{pf},\kappa}(u)$ be given. Assume $F_{k}\to F$ in $L^{2}(\Omega)$, then
	\begin{align*}
		\lVert u_{k}-u\rVert_{L^2}^{2}+\lVert \nabla u_{k}-\nabla u\rVert_{L^2}^{2}
        &\leq \lVert u_{k}-u\rVert_{L^2}^{2}+\langle \mathrm{D} \mathcal{E}_{\mathrm{pf},\kappa}(u_{k})-\mathrm{D} \mathcal{E}_{\mathrm{pf},\kappa}(u),u_{k}-u\rangle_{L^{2}}
		\\
		&\leq\lVert u_{k}+\mathrm{D} \mathcal{E}_{\mathrm{pf},\kappa}(u_{k})
        -u-\mathrm{D} \mathcal{E}_{\mathrm{pf},\kappa}(u)\rVert_{L^2}\cdot\lVert u_{k}-u\rVert_{L^2}
		\\
		&\leq \frac{1}{2}\lVert F_{k}-F\rVert_{L^{2}}^{2}+\frac{1}{2}\lVert u_{k}-u\rVert_{L^2}^{2}.
	\end{align*}
	This shows that $u_{k}\to u$ in $H^{1}(\Omega)$. Besides, due to \eqref{estimate on element in subdifferntial domain}, $(u_{k})_{k}$ is bounded in $H^{2}(\Omega)$. Then, by interpolation, we have an inequality of the form 
	\begin{align*}
		\lVert u_{k}-u\rVert_{H^{2-s}}\leq C\lVert u_{k}-u\rVert_{H^{2}}^{1-s}\lVert u_{k}-u\rVert_{H^{1}}^{s},
	\end{align*}
	which implies that $u_{k}\to u$ in $H^{2-s}(\Omega)$.
    
	Altogether, we now have the invertibility of $\mathscr{L}_{k}:X\to Y$ and write the inverse operator as $\mathscr{L}_{k}^{-1}:Y\to X$. But for the continuity and even compactness of the inverse operator, we need to introduce two refined Banach spaces:
	\begin{subequations}
    \begin{align}
		\tilde{X}&:=H_{0,\mathrm{div}}^{1}(\Omega)
        \times H_{\mathrm{sym,Tr}}^{1}(\Omega)
        \times H^{2-s}(\Omega)
        \times H_{\vec{n}}^{2}(\Omega),
		\label{fixed point theorem refined preimage space}\\
		\tilde{Y}&:= L^{\frac{3}{2}}(\Omega)^{3}\times L^{\frac{3}{2}}(\Omega)^{3\times 3}\times W^{1,\frac{3}{2}}(\Omega)\times H^{1}(\Omega).
        \label{fixed point theorem refined image space}
	\end{align}
    \end{subequations}
	where $0<s<1/4$. From above arguments, we know that $\mathscr{L}_{k}^{-1}:Y\to\tilde{X}$ is continuous. Since $\tilde{Y}\hookrightarrow\hookrightarrow Y$, the restriction $\mathscr{L}_{k}^{-1}:\tilde{Y}\to\tilde{X}$ is compact.
    \paragraph{Properties of $\mathscr{F}_{k}$.}
	Now, let us consider the operator $\mathscr{F}_{k}$. We want to derive that $\mathscr{F}_{k}:\tilde{X}\to\tilde{Y}$ is continuous and that it maps bounded sets to bounded sets. To this end, let $(v,S,\varphi,\mu)\in\tilde{X}$, and we deduce the following estimates for the different components of $\mathscr{F}_{k}$:\\
    We first discuss the terms in the first line of $\mathscr{F}_{k}$. Since $v\in H^{1}(\Omega)\hookrightarrow L^{6}(\Omega)$, we obtain
	\begin{align*}
		\lVert \rho v\rVert_{L^{\frac{3}{2}}}&\leq \lVert \rho\rVert_{L^{2}}\lVert v\rVert_{L^{6}}\leq C\lVert v\rVert_{H^{1}}\left(\lVert\varphi\rVert_{L^{2}}+1\right).
	\end{align*}
	Notice that $\divge{\rho_{k}v\otimes v}$ contains terms of the form $\rho_{k}(\partial_{x_{l}}v_{i})v_{j}$ and $(\partial_{x_{l}}\rho_{k})v_{i}v_{j}$ for $i,j,l=1,\ldots,3$. Besides, we have $\rho_{k}\in L^{\infty}(\Omega)\cap H^{2}(\Omega)$. Hence, $\partial_{x_{l}}\rho_{k},v_{i}\in H^{1}(\Omega)\hookrightarrow L^{6}(\Omega)$. Hence, we obtain
	\begin{align*}
		&\lVert \rho_{k}(\partial_{x_{l}}v_{i})v_{j}\rVert_{L^{\frac{3}{2}}}\leq \lVert \rho_{k}\rVert_{L^{\infty}}\lVert \partial_{x_{l}}v_{i}\rVert_{L^2}\lVert v_{j}\rVert_{L^{6}},\\
		&\lVert (\partial_{x_{l}}\rho_{k})v_{i}v_{j}\rVert_{L^{\frac{3}{2}}}\leq C\lVert \partial_{x_{l}}\rho_{k} \rVert_{L^{6}}\lVert v_{i} \rVert_{L^{6}}\lVert v_{j} \rVert_{L^{6}},
	\end{align*}
	and thus,
	\begin{align*}
		\lVert \divge{\rho_{k} v\otimes v}\rVert_{L^{\frac{3}{2}}}\leq C_{k}\lVert v\rVert_{H^{1}}^{2}.
	\end{align*}
	Since $\mu\in H^{2}(\Omega)$ and since $\varphi_{k}\in H^{2}(\Omega)$ implies that $\partial_{x_{l}}\varphi_{k}\in H^{1}(\Omega)\hookrightarrow L^{6}(\Omega)$, we also obtain 
	\begin{align*}
		\lVert \mu\nabla\varphi_{k}\rVert_{L^{\frac{3}{2}}}\leq \lVert\nabla\varphi_{k}\rVert_{L^{6}}\lVert\mu\rVert_{L^{2}}=C_{k}\lVert\mu\rVert_{L^{2}}.
	\end{align*}
	Note that $\divge{J}v=\divge{\frac{\rho_{2}-\rho_{1}}{2}\nabla\mu}v$ consists of terms of the form $\frac{\rho_{2}-\rho_{1}}{2}(\partial_{x_{i}}\partial_{x_{i}}\mu) v_{l}$. Moreover, we have $\mu\in H^{2}(\Omega)$ and $v_{l}\in H^{1}(\Omega)\hookrightarrow L^{6}(\Omega)$. Thus, we obtain
	\begin{align*}
		\lVert \divge{J}v\rVert_{L^{\frac{3}{2}}}\leq C\lVert\mu\rVert_{H^{2}}\lVert v\rVert_{L^{6}}\leq C\lVert\mu\rVert_{H^{2}}\lVert v\rVert_{H^{1}}.
	\end{align*}
	Similarly, $J\cdot\nabla v$ has terms of the form $\frac{\rho_{2}-\rho_{1}}{2}\partial_{x_{i}}\mu\partial_{x_{j}}v$. Since $\partial_{x_{i}}\mu\in H^{1}(\Omega)\hookrightarrow L^{6}(\Omega)$ and $v\in H^{1}(\Omega)$, we obtain 
	\begin{align*}
		\lVert J\nabla v\rVert_{L^{\frac{3}{2}}}
		\leq C\lVert \nabla\mu\rVert_{H^{1}}\lVert \nabla v\rVert_{L^{2}}
		\leq C\lVert \mu\rVert_{H^{2}}\lVert \nabla v\rVert_{H^{1}}.
	\end{align*}
	Observe that $\divge{\eta(\varphi_{k})S}$ contains terms of the form that $\eta^{\prime}(\varphi_{k})\partial_{x_{l}}\varphi_{k} S_{ij}$ and $\eta(\varphi_{k})\partial_{x_{l}} S_{ij}$. Since $\partial_{x_{l}}\varphi_{k}\in H^{1}(\Omega)\hookrightarrow L^{6}(\Omega)$, $S\in H^{1}(\Omega)$ and $|\eta(\varphi_{k})|$, $|\eta^{\prime}(\varphi)|$ are bounded by assumption \eqref{assumption on coefficients}, we obtain 
	\begin{align*}
		&\lVert \eta^{\prime}(\varphi_{k})\partial_{x_{l}}\varphi_{k} S_{ij}\rVert_{L^{\frac{3}{2}}}\leq C\lVert \partial_{x_{l}}\varphi_{k}\rVert_{L^{6}}\lVert S_{ij}\rVert_{L^{2}},\\
		&\lVert \eta(\varphi_{k})\partial_{x_{l}} S_{ij}\rVert_{L^{\frac{3}{2}}}\leq C\lVert \partial_{x_{l}} S_{ij}\rVert_{L^{2}},
	\end{align*}
	and therefore also
	\begin{align*}
		&\lVert \divge{\eta(\varphi_{k})S}\rVert_{L^{\frac{3}{2}}}\leq C_{k}\lVert S\rVert_{H^1}.
	\end{align*}	
    This finishes the estimates for the terms in the first line of $\mathscr{F}_{k}$ and we turn to the terms in the second line. Since $S\in H^{1}(\Omega)$, we directly obtain
	\begin{align*}
		\lVert S\rVert_{L^{\frac{3}{2}}}\leq C\lVert S\rVert_{H^{1}}.
	\end{align*}
	Thanks to $v\in H^{1}(\Omega)\hookrightarrow L^{6}(\Omega)$, we find
	\begin{align*}
		\lVert v\cdot\nabla S\rVert_{L^{\frac{3}{2}}}\leq C\lVert v\rVert_{L^{6}}\lVert \nabla S\rVert_{L^{2}}\leq C\lVert v\rVert_{H^{1}}\lVert S\rVert_{H^{1}}.
	\end{align*}
	Moreover, due to $S\in H^{1}(\Omega)\hookrightarrow L^{6}(\Omega)$, we have
	\begin{align*}
		\lVert S\skw{\nabla v}-\skw{\nabla v}S\rVert_{L^{\frac{3}{2}}}\leq C\lVert \nabla v\rVert_{L^{2}}\lVert S\rVert_{L^{6}}\leq C\lVert v\rVert_{H^{1}}\lVert S\rVert_{H^{1}}.
	\end{align*}
	Since $\abs{\eta}$ is bounded by \eqref{assumption on coefficients} and $v\in H^{1}(\Omega)$, we also get
	\begin{align*}
		\lVert \eta(\varphi_{k})\sym{\nabla v}\rVert_{L^{\frac{3}{2}}}\leq C\lVert \sym{\nabla v}\rVert_{L^{2}}\leq C\lVert v\rVert_{H^{1}}.
	\end{align*}
	Now, we discuss the estimates for the terms in the third line of $\mathscr{F}_{k}$. By H\"older inequality, we directly have
	\begin{align*}
		\lVert \varphi \rVert_{W^{1,\frac{3}{2}}}\leq C\lVert\varphi\rVert_{H^{1}}.
	\end{align*}
	Also notice that, since $\nabla\varphi_{k}\in H^{1}(\Omega)$ and $v\in H^{1}(\Omega)\hookrightarrow L^{6}(\Omega)$, we have
	\begin{align*}
		\lVert v\cdot\nabla\varphi_{k}\rVert_{L^{\frac{3}{2}}}
		\leq \lVert \nabla\varphi_{k}\rVert_{L^2}\lVert v\rVert_{L^{6}}
		\leq C_{k}\lVert v\rVert_{H^{1}}.
	\end{align*} 
	Moreover, The derivative of $v\cdot\nabla\varphi_{k}$ consists of terms of the form $\partial_{x_{i}} v_{j}\partial_{x_{l}}\varphi_{k}$ and $v_{j}\partial_{x_{i}}\partial_{x_{l}}\varphi_{k}$. Since $\partial_{x_{i}} v_{j}\in L^{2}(\Omega)$, $\partial_{x_{l}}\varphi_{k}\in H^{1}(\Omega)\hookrightarrow L^{6}(\Omega)$, $v_{j}\in H^{1}(\Omega)\hookrightarrow L^{6}(\Omega)$ and $\partial_{x_{i}}\partial_{x_{l}}\varphi_{k}\in L^{2}(\Omega)$, we arrive at
	\begin{align*}
		&\lVert\partial_{x_{i}} v_{j}\partial_{x_{l}}\varphi_{k}\rVert_{L^{\frac{3}{2}}}
		\leq \lVert \partial_{x_{i}} v_{j}\rVert_{L^{2}}
		\lVert \partial_{x_{l}}\varphi_{k}\rVert_{L^{6}},\\
		&\lVert v_{j}\partial_{x_{i}}\partial_{x_{l}}\varphi_{k}\rVert_{L^{\frac{3}{2}}}
		\leq \lVert v_{j}\rVert_{L^{6}}
		\lVert \partial_{x_{i}}\partial_{x_{l}}\varphi_{k}\rVert_{L^{2}}.
	\end{align*}
	Therefore, we also conclude 
	\begin{align*}
		&\lVert v\cdot\nabla\varphi_{k}\rVert_{W^{1,\frac{3}{2}}}\leq C_{k}\lVert v\rVert_{H^{1}},
	\end{align*} 
    where $C_{k}$ is a positive constant depending on $k$. Since $\mu\in H^{2}(\Omega)$, we obtain
	\begin{align*}
		\lVert \int_{\Omega}\mu \dd{x}\rVert_{W^{1,\frac{3}{2}}}
		\leq C\abs{ \int_{\Omega}\mu \dd{x}}
		\leq C\lVert \mu\rVert_{L^{2}}.
	\end{align*}
	For the last line of $\mathscr{F}_{k}$, the estimate is direct, since
	\begin{align*}
		\lVert \varphi\rVert_{H^{1}}=\lVert \varphi\rVert_{H^{1}}\mbox{ and }\lVert\mu\rVert_{H^{1}}\leq\lVert\mu\rVert_{H^{2}}.
	\end{align*}
	Altogether, from the above discussion, we can see $\mathscr{F}_{k}:\tilde{X}\to\tilde{Y}$ is continuous. Moreover, for $w=(v,S,\varphi,\mu)$ bounded in $\tilde{X}$, also $\mathscr{F}_{k}(w)$ is bounded in $\tilde{Y}$ i.e. $\mathscr{F}_{k}$ maps bounded sets to bounded sets.
    \paragraph{Definition of the operator $\mathscr{K}_{k}:\tilde{Y}\to\tilde{Y}$.}
	Recall~\eqref{weak solution in operator sense}. In order to apply Schaefer's fixed point theorem,see~\cite[][Chapter 9.2.2, Theorem 2]{evansPartialDifferentialEquations2010} for details, we need to introduce a new operator $\mathscr{K}_{k}$ whose image space and preimage space coincide. To this end, using $\tilde{Y}$ from~\eqref{fixed point theorem refined image space}, we define the operator as follows:
	\begin{equation}
		\mathscr{K}_{k}:\tilde{Y}\to\tilde{Y},\,
		u\mapsto\mathscr{F}_{k}\circ\mathscr{L}_{k}^{-1}(u),
		\label{operator K_k}
	\end{equation}
	which is feasible by the invertibility of $\mathscr{L}_{k}$. With the help of this operator, we can rewrite \eqref{weak solution in operator sense} as
	\begin{align}
		u-\mathscr{K}_{k}(u)=0\mbox{ }\Longleftrightarrow u=\mathscr{K}_{k}(u),
	\end{align}
	where $u=\mathscr{L}_{k}(w)$ for $w\in\tilde{X}$. Since we already showed that $\mathscr{L}_{k}^{-1}$ is compact and $\mathscr{F}_{k}$ is continuous, then $\mathscr{K}_{k}$ is also continuous and compact on $\tilde{Y}$. 
    \paragraph{Boundedness of $\mathscr{K}_{k}$ in $\tilde{X}$.} In order to apply Schaefer's fixed point theorem, it remains to show that
	\begin{align}
		\left\{u\in\tilde{Y}:u=\lambda\mathscr{K}_{k}(u)\mbox{ for some }0\leq\lambda\leq1\right\}
		\label{the set which should be bounded in order to apply fixed point theorem}
	\end{align}
	is bounded. To this end, let $u\in\tilde{Y}$ and $0\leq\lambda\leq1$ satisfy $u=\lambda\mathscr{K}_{k}(u)$. Again by the invertibility of $\mathscr{L}_{k}$, we find $w=\mathscr{L}_{k}^{-1}(u)$ satisfying
	\begin{equation*}
		\mathscr{L}_{k}(w)-\lambda\mathscr{F}_{k}(w)=0.
	\end{equation*}
	By the definition of the operator $\mathscr{L}_{k}$ from \eqref{operator L_k} and of $\mathscr{F}_{k}$ from \eqref{operator F_k}, we arrive at \begin{subequations} 
	the following weak formulations
	\begin{equation}
	    \begin{aligned}
		&\int_{\Omega}2\nu(\varphi_{k})\sym{\nabla v}:\sym{\nabla\Phi}\dd{x}
		+\lambda\int_{\Omega}\frac{\rho v-\rho_{k}v_{k}}{h}\cdot\Phi 
		-(\rho_{k}v\otimes v):\nabla\Phi \dd{x} 
		\\
		&+\lambda\int_{\Omega}(\divge{J}+\frac{\rho-\rho_{k}}{h}-v\cdot\nabla\rho_{k})\frac{v}{2}\cdot\Phi 
		+(J\cdot\nabla v)\cdot\Phi 
		+\eta(\varphi_{k})S:\nabla\Phi \dd{x} 
		\\
		&=\lambda\int_{\Omega}(\mu\nabla\varphi_{k})\cdot\Phi \dd{x}
        +\langle f_{k+1}, \Phi\rangle_{H^{1}}, 
		\label{lambda weak velocity}
	\end{aligned}
	\end{equation}
	for all $\Phi\in C_{0,\mathrm{div}}^{\infty}(\Omega)$.
	\begin{equation}
	    \begin{aligned}
		&\int_{\Omega}\gamma\nabla S\threedotsbin\nabla \Psi + \gamma S:\Psi\dd{x}
        +\langle \xi^{k},\Psi\rangle_{H_{\mathrm{sym,Tr}}^{1}}
        \\
		&+\lambda\int_{\Omega}\frac{S-S_{k}}{h}:\Psi 
		+(v\cdot\nabla S):\Psi \
		+(S\skw{\nabla v}-\skw{\nabla v}S):\Psi \dd{x}
		\\
        &=\lambda\int_{\Omega}\eta(\varphi_{k})\sym{\nabla v}:\Psi 
        +\gamma S:\Psi \dd{x},	
		\label{lambda weak stress}
	\end{aligned}
	\end{equation}
	for all $\Psi\in C_{\mathrm{sym,Tr}}^{\infty}(\bar{\Omega})$.
	\begin{equation}
		\lambda\frac{\varphi-\varphi_{k}}{h}+\lambda v\cdot\nabla\varphi_{k}-\lambda\int_{\Omega}\mu \dd{x}
		=\Delta\mu-\int_{\Omega}\mu \dd{x} 
		\label{lambda first CH}
	\end{equation}
	as well as
	\begin{equation}
		\varphi+\mathrm{D} \mathcal{E}_{\mathrm{pf},\kappa}(\varphi)
		=\lambda\varphi+\lambda\mu+\lambda\kappa\frac{\varphi+\varphi_{k}}{2}. 
		\label{lambda second CH}		
	\end{equation}
	\end{subequations}
	Now due to the bounds deduced above and by using a density argument, we conclude that $\Phi=v$ and $\Psi=S$ are admissible test functions for \eqref{lambda weak velocity} and \eqref{lambda weak stress}. Moreover, testing \eqref{lambda first CH} with $\mu$ and \eqref{lambda second CH} with $\frac{1}{h}(\varphi-\varphi_{k})$ and integrating in space gives
	\begin{subequations}
	\begin{equation}
	    \begin{aligned}
		&\int_{\Omega}2\nu(\varphi_{k})\abs{\sym{\nabla v}}^{2} \dd{x}
		+\frac{\lambda}{h}\int_{\Omega}\rho\frac{\abs{v}^{2}}{2}-\rho_{k}\frac{\abs{v_{k}}^{2}}{2} 
		+\rho_{k}\frac{\abs{v-v_{k}}^{2}}{2} \dd{x}
		\\
		&=\lambda\int_{\Omega}\mu(\nabla\varphi_{k}\cdot v) \dd{x}
		-\lambda\int_{\Omega}\eta(\varphi_{k})S:\sym{\nabla v}\dd{x}
        +\langle f_{k+1}, v\rangle_{H^{1}},
		\label{lambda discrete velocity}
	\end{aligned}
	\end{equation}
	and
	\begin{equation}
	    \begin{aligned}
		&\int_{\Omega}\gamma\abs{\nabla S}^2 \dd{x}
        +\langle \xi^{k},S\rangle_{H_{\mathrm{sym,Tr}}^{1}}
        +\int_{\Omega}\gamma \abs{S}^{2}\dd{x}
		+\frac{\lambda}{h}\int_{\Omega}\frac{\abs{S}^{2}}{2}-\frac{\abs{S_{k}}^{2}}{2} 
		+\frac{\abs{S-S_{k}}^{2}}{2} \dd{x}
		\\
		&=\lambda\int_{\Omega}\eta(\varphi_{k})\sym{\nabla v}:S +\gamma \abs{S}^{2}\dd{x},
		\label{lambda discrete stress}
	\end{aligned}
	\end{equation}
	and
	\begin{equation}
		\frac{\lambda}{h}\int_{\Omega}(\varphi-\varphi_{k})\mu \dd{x}
		+\lambda\int_{\Omega}(v\cdot\nabla\varphi)\mu \dd{x}
		-(\lambda-1)\abs{\int_{\Omega}\mu \dd{x} }^{2}
		+\int_{\Omega}\abs{\nabla\mu}^{2} \dd{x} =0,
		\label{lambda discrete first CH}
	\end{equation}
	as well as
	\begin{equation}
	    \begin{aligned}
		&\frac{1}{h}\int_{\Omega}\varphi(\varphi-\varphi_{k}) \dd{x}
		+\frac{1}{h}\int_{\Omega}\nabla\varphi\cdot(\nabla\varphi-\nabla\varphi_{k})\dd{x}
		+\frac{1}{h}\int_{\Omega}W_{\kappa}^{\prime}(\varphi)(\varphi-\varphi_{k})\dd{x}
		\\
		&=\frac{\lambda}{h}\int_{\Omega}\varphi(\varphi-\varphi_{k})-\frac{\lambda}{h}\int_{\Omega}\mu(\varphi-\varphi_{k})\dd{x}
		+\frac{\lambda\kappa}{h}\int_{\Omega}\frac{\varphi^{2}-\varphi_{k}^{2}}{2}\dd{x}.
		\label{lambda discrete second CH}
	\end{aligned}
	\end{equation}
	\end{subequations}
	Summing up \eqref{lambda discrete velocity}-\eqref{lambda discrete second CH}, we obtain that
	\begin{equation*}
    \begin{aligned}
		0
		=&\int_{\Omega}2\nu(\varphi_{k})\abs{\sym{\nabla v}}^{2} \dd{x}
        -\langle f_{k+1}, v\rangle_{H^{1}}
		+\int_{\Omega}\gamma\abs{\nabla S}^{2} \dd{x}
        +\langle \xi^{k},S\rangle_{H_{\mathrm{sym,Tr}}^{1}}
		+\int_{\Omega}\abs{\nabla\mu}^{2} \dd{x}
		\\
		&+\frac{\lambda}{h}\int_{\Omega}\rho\frac{\abs{v}^2}{2}-\rho_{k}\frac{\abs{v_{k}}^2}{2} \dd{x}
		+\frac{\lambda}{h}\int_{\Omega}\rho_{k}\frac{\abs{v-v_{k}}^{2}}{2} \dd{x}
		+\frac{\lambda}{h}\int_{\Omega}\frac{\abs{S}^{2}}{2}-\frac{\abs{S_{k}}^2}{2} \dd{x}
		+\frac{\lambda}{h}\int_{\Omega}\frac{\abs{S-S_{k}}^{2}}{2} \dd{x}
		\\
		&+(1-\lambda)\int_{\Omega}\gamma\abs{S}^{2} \dd{x}
		+(1-\lambda)\abs{\int_{\Omega}\mu \dd{x}}^{2}
		+\frac{1-\lambda}{h}\int_{\Omega}\frac{\varphi^{2}}{2}-\frac{\varphi_{k}^{2}}{2} \dd{x}
		+\frac{1-\lambda}{h}\int_{\Omega}\frac{\abs{\varphi-\varphi_{k}}^{2}}{2} \dd{x}
		\\
		&+\frac{1}{h}\int_{\Omega}\frac{\abs{\nabla\varphi}^{2}}{2}-\frac{\abs{\varphi_{k}}^2}{2} \dd{x}
		+\frac{1}{h}\int_{\Omega}\frac{\abs{\varphi-\varphi_{k}}^{2}}{2} \dd{x}
		+\frac{1}{h}\int_{\Omega}W_{\kappa}^{\prime}(\varphi)(\varphi-\varphi_{k}) \dd{x}
		-\frac{\lambda}{h}\int_{\Omega}\kappa\frac{\varphi^2-\varphi_{k}^{2}}{2} \dd{x}.
	\end{aligned}
    \end{equation*}
	Recall that $\lambda\in[0,1]$ and, by convexity, there holds
	\begin{equation*}
		\int_{\Omega}W_{\kappa}^{\prime}(\varphi)(\varphi-\varphi_{k})\dd{x}
		\geq\int_{\Omega} W_{\kappa}(\varphi)-W_{\kappa}(\varphi_{k}) \dd{x}.
	\end{equation*}
	This leads to the estimate
	\begin{equation*}
	    \begin{aligned}
		&\int_{\Omega}2\nu(\varphi_{k})\abs{\sym{\nabla v}}^{2} \dd{x}
        -\langle f_{k+1}, v\rangle_{H^{1}}
		+\int_{\Omega}\gamma\abs{\nabla S}^{2} \dd{x}
        +\langle \xi^{k},S\rangle_{H_{\mathrm{sym,Tr}}^{1}}
		+\int_{\Omega}\abs{\nabla\mu}^{2} \dd{x}
        \\
		&+\frac{\lambda}{h}\int_{\Omega}\rho\frac{\abs{v}^{2}}{2}-\rho_{k}\frac{\abs{v_{k}}^{2}}{2} \dd{x}
		+\frac{\lambda}{h}\int_{\Omega}\frac{\abs{S}^{2}}{2}-\frac{\abs{S_{k}}^2}{2} \dd{x}
		+(1-\lambda)\int_{\Omega}\gamma\abs{S}^{2} \dd{x}
		+(1-\lambda)\abs{\int_{\Omega}\mu \dd{x}}^{2}
        \\
		&+\frac{1-\lambda}{h}\int_{\Omega}\frac{\abs{\varphi}^{2}}{2}-\frac{\abs{\varphi_{k}}^{2}}{2} \dd{x}
		+\frac{1}{h}\int_{\Omega}\frac{\abs{\nabla\varphi}^{2}}{2}-\frac{\abs{\nabla\varphi_{k}}^{2}}{2} \dd{x}
		+\frac{1}{h}\int_{\Omega}W_{\kappa}(\varphi)-W_{\kappa}(\varphi_{k}) \dd{x}
        \\
		&-\frac{\lambda}{h}\int_{\Omega}\kappa\frac{\varphi^{2}-\varphi_{k}^{2}}{2} \dd{x}
        \leq0.
	\end{aligned}
	\end{equation*}
	Furthermore, rearranging terms results in
	\begin{equation*}
	    \begin{aligned}
		&\int_{\Omega}2\nu(\varphi_{k})\abs{\sym{\nabla v}}^{2} \dd{x}
        -\langle f_{k+1}, v\rangle_{H^{1}}
		+\int_{\Omega}\gamma\abs{\nabla S}^{2} \dd{x}
        +\langle \xi^{k},S\rangle_{H_{\mathrm{sym,Tr}}^{1}}
		+\int_{\Omega}\abs{\nabla\mu}^{2} \dd{x}
        \\
        &+\frac{\lambda}{h}\int_{\Omega}\frac{\abs{S}^{2}}{2}\dd{x}
        +(1-\lambda)\int_{\Omega}\gamma\abs{S}^{2}\dd{x}
		+(1-\lambda)\abs{ \int_{\Omega}\mu \dd{x}}^{2}
		+\frac{1}{h}\int_{\Omega}\frac{\abs{\nabla\varphi}^{2}}{2} \dd{x}
		+\frac{1}{h}\int_{\Omega}W(\varphi) \dd{x}
		\\
		&\leq\frac{\lambda}{h}\int_{\Omega}\rho_{k}\frac{\abs{v_{k}}^{2}}{2} \dd{x}
		+\frac{\lambda}{h}\int_{\Omega}\frac{\abs{S_{k}}^{2}}{2} \dd{x}
		+\frac{1-\lambda}{h}\int_{\Omega}\frac{\abs{\varphi_{k}}^{2}}{2} \dd{x}
		+\frac{1}{h}\int_{\Omega}\frac{\abs{\nabla\varphi_{k}}}{2} \dd{x}
		+\frac{1}{h}\int_{\Omega}W_{\kappa}(\varphi_{k}) \dd{x}
        .
	\end{aligned}
	\end{equation*}
	This is equivalent to
    \begin{equation*}
        \begin{aligned}
		&h\int_{\Omega}2\nu(\varphi_{k})\abs{\sym{\nabla v}}^{2} \dd{x}
		+h\int_{\Omega}\gamma\abs{\nabla S}^{2} \dd{x}
        +h\langle \xi^{k},S\rangle_{H_{\mathrm{sym,Tr}}^{1}}
		+h\int_{\Omega}\abs{\nabla\mu}^{2} \dd{x}
		+\int_{\Omega}\frac{\abs{\nabla\varphi}^{2}}{2} \dd{x}
		\\
		&+\lambda\int_{\Omega}\frac{\abs{S}^{2}}{2}\dd{x}
        +(1-\lambda)h\int_{\Omega}\gamma\abs{S}^{2}\dd{x}
        +h(1-\lambda)\abs{ \int_{\Omega}\mu \dd{x}}^{2}
		+\int_{\Omega}W(\varphi)\dd{x}
        \\
        &\leq C_{k}+h\langle f_{k+1}, v\rangle_{H^{1}},
	\end{aligned}
    \end{equation*}
	where $C_{k}$ is a positive constant depending on $k$. Notice that by the invertibility of $\mathscr{L}_{k}$, $w=(v,S,\varphi,\mu)=\mathscr{L}_{k}^{-1}(u)\in X$ so that $\varphi\in\dom{\mathrm{D}\mathcal{E}_{\mathrm{pf},\kappa}}$. Hence $\varphi\in[-1,1]$. Moreover, due to the continuity of $W$ on $[-1,1]$, $W(\varphi)$ is bounded both from above and from below. Therefore, it can be absorbed by the constant $C_{k}$ on the right-hand side, i.e.,
	\begin{equation*}
	    \begin{aligned}
		&h\int_{\Omega}2\nu(\varphi_{k})\abs{\sym{\nabla v}}^{2} \dd{x}
		+h\int_{\Omega}\gamma\abs{\nabla S}^{2} \dd{x}
        +h\langle\xi^{k},S\rangle_{H_{\mathrm{sym,Tr}}^{1}}
		+h\int_{\Omega}\abs{\nabla\mu}^{2} \dd{x}
		+\int_{\Omega}\frac{\abs{\nabla\varphi}^{2}}{2} \dd{x}
		\\
		&+\lambda\int_{\Omega}\frac{\abs{S}^{2}}{2}\dd{x}
        +h(1-\lambda)\int_{\Omega}\gamma\abs{S}^{2}\dd{x}
        +h(1-\lambda)\abs{ \int_{\Omega}\mu \dd{x}}^{2}
        \leq C_{k}+\langle f_{k+1}, v\rangle_{H^{1}}.
	\end{aligned}
	\end{equation*}
    Notice that for $\lambda\in[0,\frac{1}{2}]$, the term $h(1-\lambda)\int_{\Omega}\gamma\abs{S}^{2}\dd{x}$ gives an estimate on $\int_{\Omega}\abs{S}^{2}\dd{x}$. Similarly, for $\lambda\in(\frac{1}{2},1]$, the term $\lambda\int_{\Omega}\frac{\abs{S}^{2}}{2}\dd{x}$ also gives an estimate on $\int_{\Omega}\abs{S}^{2}\dd{x}$. Moreover, since $\gamma>0$ and $\nu$ is bounded from below by a positive constant by \eqref{assumption on coefficients}, we finally arrive at
	\begin{equation*}
	    \begin{aligned}
		&\int_{\Omega}\abs{\nabla v}^{2} \dd{x}
		+\int_{\Omega}\abs{\nabla S}^{2} \dd{x}
        +\int_{\Omega}\abs{S}^{2} \dd{x}
        +\langle\xi^{k},S\rangle_{H_{\mathrm{sym,Tr}}^{1}}
		+\int_{\Omega}\abs{\nabla\mu}^{2} \dd{x}
		\\
        &+\int_{\Omega}\abs{\nabla\varphi}^{2} \dd{x}
        +(1-\lambda)\abs{\int_{\Omega}\mu \dd{x}}^{2}
		\leq C_{k}
        +\langle f_{k+1}, v\rangle_{H^{1}}.
	\end{aligned}
	\end{equation*}
	Combined with $\varphi\in[-1,1]$ and Poincar\'e's inequality, we obtain 
	\begin{equation}
		\lVert v\rVert_{H^1}
		+\lVert S\rVert_{H^1}
        +\widetilde{\mathcal{P}}(\varphi_{k};S)
		+\lVert\varphi\rVert_{H^{1}}
		+\lVert\nabla\mu\rVert_{L^2}
		+\sqrt{1-\lambda}\abs{ \int_{\Omega}\mu \dd{x}}\leq C_{k}. 
		\label{estimate lambda problem}
	\end{equation}
	It remains to find an $H^{2}-$estimate for $\mu$. By \eqref{estimate coming from regularity}, this is equivalent to finding an $H^{1}-$estimate. Since we already have an estimate on $\lVert\nabla\mu\rVert_{L^{2}}$, it remains to find an $L^{2}-$estimate for $\mu$. Again by Poincar\'e's inequality, we have
	\begin{equation*}
		\lVert \mu\rVert_{L^2}\leq \lVert\nabla\mu\rVert_{L^2}+\abs{\int_{\Omega}\mu \dd{x}}.
	\end{equation*}
	Hence, it is sufficient to find a bound for $\abs{\int_{\Omega}\mu \dd{x}}$. To this end, first consider $\lambda\in[0,\frac{1}{2})$. Then we directly have
	\begin{equation*}
		\frac{\sqrt{2}}{2}\abs{\int_{\Omega}\mu \dd{x}}
		\leq\sqrt{1-\lambda}\abs{\int_{\Omega}\mu \dd{x}}
		\leq C_{k}.
	\end{equation*}
	For $\lambda\in[\frac{1}{2},1]$, $\frac{1}{2}\abs{\int_{\Omega}\mu \dd{x}}\leq \lambda\abs{\int_{\Omega}\mu \dd{x}}$. Repeating the argument as in proving \eqref{estimate on the discrete version of the derivative of double well potential and mean value of mu} to equation \eqref{lambda discrete second CH} and notice $\lVert \nabla\mu\rVert_{L^2}$ and $\lVert \nabla\varphi\rVert_{L^2}$ are bounded by $C_{k}$ due to \eqref{estimate lambda problem}, we can get
	\begin{equation*}
		\abs{ \int_{\Omega}\mu \dd{x}}\leq C_{k}.
	\end{equation*}
	Therefore, \eqref{estimate lambda problem} can be improved to
	\begin{equation}
		\lVert v\rVert_{H^1}
        +\lVert S\rVert_{H^1}
        +\widetilde{\mathcal{P}}(\varphi_{k};S)
        +\lVert\varphi\rVert_{H^{1}}
        +\lVert \mu\rVert_{H^2}
        \leq C_{k}. 
	\end{equation}
	Moreover, from \eqref{lambda second CH}, we have the additional estimate
	\begin{equation*}
		\lVert \mathrm{D} \mathcal{E}_{\mathrm{pf},\kappa}(\varphi)\rVert_{L^2}\leq (\lambda+1)\lVert\varphi\rVert_{L^2}+\lambda\lVert\mu\rVert_{L^2}+\frac{\lambda\kappa}{2}\left(\lVert\varphi\rVert_{L^2}+\lVert\varphi_{k}\rVert_{L^2}\right)\leq C_{k}.
	\end{equation*}
	Altogether, we have the following estimate for $w=(v,S,\varphi,\mu)$
	\begin{align*}
		\lVert w\rVert_{\tilde{X}}+\lVert\mathrm{D} \mathcal{E}_{\mathrm{pf},\kappa}(\varphi)\rVert_{L^{2}}\leq C_{k},
	\end{align*}
    with $\tilde{X}$ from~\eqref{fixed point theorem refined preimage space}.
    \paragraph{Boundedness of $\mathscr{K}_{k}$ in $\tilde{Y}$.}
	It remains to show that $u=\mathscr{L}_{k}(w)$ is bounded in $\tilde{Y}$. Recall that $\mathscr{F}_{k}$ maps bounded sets in $\tilde{X}$ to bounded sets in $\tilde{Y}$ and that $u=\lambda\mathscr{F}_{k}(w)$. Therefore,
	\begin{align*}
		\lVert u\rVert_{\tilde{Y}}=\lVert \lambda\mathscr{F}_{k}(w)\rVert_{\tilde{Y}}\leq C_{k}\left(\lVert w\rVert_{\tilde{X}}+1\right)\leq\tilde{C}_{k}.
	\end{align*}
	where $\tilde{C}_{k}$ is a positive constant depending on $k$. 
    
    Now, since all assumptions in Schaefer's fixed point theorem have been satisfied, the existence of a fixed point of the operator $\mathscr{K}_{k}$ can be guaranteed. Hence, there exists a weak solution to \eqref{discrete velocity}-\eqref{discrete second CH}.
\end{proof}

\subsection{Existence of weak solutions}
In Section \ref{Implicit time discretization}, we have constructed a sequence of approximating solutions $(v_{h},S_{h},\varphi_{h},\mu_{h})_{h}$. In order to get a time-continuous solution, we perform the limit passage $h\to0$. We follow the ideas of \cite[][Theorem 3.4]{abelsExistenceWeakSolutions2013} where the existence of weak solutions for the Cahn--Hilliard--Navier--Stokes system was shown. In our case here, we have to deal with the extra stress tensor which will need some more attention. 
\begin{theorem}\label{existence of generalized solution for TPP}
	Let $\gamma>0$. Let $v_{0},S_{0}$ and $\varphi_{0}$ satisfy Assumption~\ref{assumptions on the initial data}. Let Assumption~\ref{TP assumption} be satisfied and let $\mathcal{P}$ fulfill Assumption~\ref{assumption on the dissipation potential}. Then, there exists a weak solution $(v,S,\varphi,\mu)$ of \eqref{two phase system alternative} in the sense of Definition \ref{gerneralized solution TPP}. Moreover, the phase-field variable $\varphi$ takes values in $(-1,1)$ a.e. in $\Omega\times(0,T)$.
\end{theorem}
\begin{proof} 
    \textbf{Step 1: Existence of discrete solutions at each time step $k=0,1\ldots,N$.}
	We would like to apply Lemma \ref{Existence of solution to the time discrete problem} in order to conclude the existence of a weak solution $(v_{k},S_{k},\varphi_{k},\mu_{k})_{k}$ at each time step. For this, notice that the initial datum $\varphi_{0}$ is required to be in $H^{2}(\Omega)$. Therefore, we have to approximate $\varphi_{0}$ by functions in $H^{2}(\Omega)$. To this end, consider the following parabolic problem
	\begin{equation*}
	    \left\{
        \begin{aligned}
             \partial_{t}u-\Delta u&=0\,&& \mbox{in}\,\Omega\times(0,T),\\
			u(0)&=\varphi_{0}\,&&\mbox{in}\,\Omega
            ,\\
			\vec{n}\cdot\nabla u&=0\,&&\mbox{on}\,\partial\Omega\times(0,T). 
        \end{aligned}\right.
	\end{equation*}
	By the regularity theorem of parabolic equations e.g.~\cite[][Chapter 7.1.3]{evansPartialDifferentialEquations2010}, there exists a solution $u\in L^{2}(0,T;H^{2}(\Omega))$, hence we can set $\varphi_{0}^{N}:=u|_{t=\frac{1}{N}}$. Moreover, the maximum principle gives that  $\abs{\varphi_{0}^{N}}\leq1$ and we also have that 
	\begin{equation}
		\varphi_{0}^{N}\to\varphi_{0}\mbox{ in }H^{1}(\Omega)
		\label{convergence of approximating initial datum of phase-field}
	\end{equation}
	(For details, see \cite[][Chapter 7.1]{evansPartialDifferentialEquations2010}). 
	Set the time-step size to be $h=\frac{T}{N}$ ($h=\frac{1}{N}$ when $T=\infty$) and $t_{k}=kh$ for all $k=0,1,\ldots,N$. Now, we can apply Lemma \ref{Existence of solution to the time discrete problem} with initial data $(v_{0},S_{0},\varphi_{0}^{N})$ to deduce the existence of approximating solutions $(v_{k},S_{k},\varphi_{k},\mu_{k})$ where $k=1,\ldots,N$.
    \paragraph{Step 2: Interpolation in time and weak formulation for the interpolations.}
	Define $F^{N}(t)$ on $[0,T)$ to be the piecewise constant interpolation i.e. 
	\begin{subequations}
	\begin{equation}
		F^{N}(t):=F_{k}\mbox{ for }t\in((k-1)h,kh]\mbox{ and }F(0)=F_{0},
		\label{piecewise constant interpolation}
	\end{equation}
	where $k\in\mathbb{N}$ and $F\in\{v,S,\varphi,\mu,\xi\}$. Also set 
	\begin{equation}
		\rho^{N}:=\frac{1}{2}(\rho_{2}+\rho_{1})+\frac{1}{2}(\rho_{2}-\rho_{1})\varphi^{N}.
		\label{piecewise constnat interpolation of density}
    \end{equation}
	Moreover, we define
	\begin{align}
		\partial_{t,h}^{+}F^{N}(t)&:=\frac{1}{h}(F^{N}(t+h)-F^{N}(t)),
        \\
		\partial_{t,h}^{-}F^{N}(t)&:=\frac{1}{h}(F^{N}(t)-F^{N}(t-h)),
        \\
		F_{h}^{N}(t)&:=F^{N}(t-h). 
		\label{one step back of interpolation}
	\end{align}
	\end{subequations}
	Notice that the approximating problem should be tested by static test functions, but, in the limit, we aim for a weak formulation which also involves time. 
	\begin{subequations}
	For this purpose, for $\Phi\in C_{0,\mathrm{div}}^{\infty}(\Omega\times[0,T))$, we define the interpolations $\tilde{\Phi}^{k+1}=\int_{kh}^{(k+1)h}\Phi \dd{t}$ and use $\tilde{\Phi}^{k+1}$ as a the test function in \eqref{discrete velocity}. Sum over $k\in\{0,1,\ldots,N\}$, we obtain
	\begin{equation}
	    \begin{aligned}
		&\int_{0}^{T}\int_{\Omega}\partial_{t,h}^{-}(\rho^{N}v^{N})\cdot\Phi \dd{x}\dd{t}
		-\int_{0}^{T}\int_{\Omega}(\rho_{h}^{N}v^{N}\otimes v^{N}):\nabla\Phi \dd{x}\dd{t} 
		\\
		+&\int_{0}^{T}\int_{\Omega}2\nu(\varphi_{h}^{N})\sym{\nabla v^N}:\sym{\nabla\Phi} \dd{x}\dd{t}
		+\int_{0}^{T}\int_{\Omega}\eta(\varphi_{h}^{N})S^{N}:\sym{\nabla\Phi} \dd{x}\dd{t} 
		\\
		-&\int_{0}^{T}\int_{\Omega}(v^{N}\otimes J^{N}):\nabla\Phi \dd{x}\dd{t}
		=\int_{0}^{T}\int_{\Omega} \mu^{N}\nabla\varphi_{h}^{N}\cdot\Phi \dd{x}\dd{t}
        +\int_{0}^{T}\langle f^{N}, \Phi\rangle_{H^{1}}\dd{t},	
		\label{weak formulation piecewise constant velocity}
	\end{aligned}
	\end{equation}
	for all $\Phi\in C_{0,\mathrm{div}}^{\infty}(\Omega\times[0,T))$. Analogously, for all $\zeta\in C_{0}^{\infty}([0,T);C^{1}(\bar{\Omega}))$, we approximate $\zeta$ in the same way as above to get $\tilde{\zeta}^{k+1}$. Hence, testing \eqref{discrete first CH} with $\tilde{\zeta}^{k+1}$ and summing over $k\in\{0,1,\ldots,N\}$ gives
	\begin{equation}
		\int_{0}^{T}\int_{\Omega}\partial_{t,h}^{-}(\varphi^{N})\zeta \dd{x}\dd{t}
		+\int_{0}^{T}\int_{\Omega}v^{N}\varphi_{h}^{N}\cdot\nabla\zeta \dd{x}\dd{t}
		=\int_{0}^{T}\int_{\Omega}\nabla\mu^{N}\cdot\nabla\zeta \dd{x}\dd{t},
		\label{weak formulation piecewise constant first CH}
	\end{equation}
	and from \eqref{discrete second CH}, we have
	\begin{equation}
		\mu^{N}+\frac{\kappa}{2}(\varphi^{N}+\varphi_{h}^{N})=-\Delta\varphi^{N}+W_{0}^{\prime}(\varphi^{N}),
		\label{weak formulation piecewise constant second CH}
	\end{equation}
	which holds almost everywhere in $\Omega\times(0,T)$.
	\end{subequations}
    \paragraph{Step 3: Energy-dissipation estimate and uniform a priori bound.}
	For all $N\in\mathbb{N}$, for all $t\in(t_{k},t_{k+1})$ and for all $k\in\{0,1,\ldots,N\}$, we define $E^{N}(t)$ to be the piecewise linear interpolation of total energy $\mathcal{E}_{\mathrm{tot}}(v_{k},S_{k},\varphi_{k})$ i.e.
	\begin{equation}
		E^{N}(t):=\frac{(k+1)h-t}{h}\mathcal{E}_{\mathrm{tot}}(v_{k},S_{k},\varphi_{k})+\frac{t-kh}{h}\mathcal{E}_{\mathrm{tot}}(v_{k+1},S_{k+1},\varphi_{k+1}),
		\label{piecewise constant interpolation total energy}
	\end{equation}
	and define $D^{N}(t)$ to be the piecewise constant dissipation for all $t\in(t_{k},t_{k+1})$, i.e.,
	\begin{equation}
		D^{N}(t):=\int_{\Omega}
		2\nu(\varphi_{h}^{N})\abs{\sym{\nabla v^{N}}}^{2}
		+\gamma\abs{\nabla S^{N}}^{2}
		+\abs{\nabla\mu^{N}}^{2}\dd{x}
        -\langle f^{N}, v^{N}\rangle_{H^{1}}
        +\mathcal{P}(\varphi_{h}^{N};S^{N}).
		\label{piecewise constant interpolation dissipation}
	\end{equation}
	From \eqref{discrete total energy estimate}, we can directly see for all $t\in(t_{k},t_{k+1})$ that 
	\begin{equation}
		-\dv{t}E^{N}(t)=\frac{E_{\mathrm{tot}}(v_{k},S_{k},\varphi_{k})-E_{\mathrm{tot}}(v_{k+1},S_{k+1},\varphi_{k+1})}{h}\geq D^{N}(t).
		\label{piecewise linear energy dissipation time derivative estimate}
	\end{equation}
	Integrating \eqref{piecewise linear energy dissipation time derivative estimate} in each time interval $(t_{k},t_{k+1})$ and summing over $k\in\{0,1,\ldots,N\}$ yields that
	\begin{equation}
	    \begin{aligned}
		\mathcal{E}_{\mathrm{tot}}(v_{0},S_{0},\varphi_{0}^{N})
		\geq&\int_{0}^{t}\int_{\Omega}2\nu(\varphi_{h}^{N})\abs{\sym{\nabla v^{N}}}^{2}
		+\gamma\abs{\nabla S^{N}}^{2}
		+\abs{\nabla\mu^{N}}^{2}\dd{x} 
        -\langle f^{N}, v^{N}\rangle_{H^{1}}\dd{\tau}
		\\
		&+\int_{0}^{t}\mathcal{P}(\varphi_{h}^{N};S^{N})\dd{\tau}+E^{N}(t),
		\label{piecewise linear energy dissipation estimate}
	\end{aligned}
	\end{equation}
    for all $t\in[0,T)$. Since $\mathcal{E}_{\mathrm{tot}}(v_{0},S_{0},\varphi_{0}^{N})$ is finite and by Assumption~\ref{assumptions on the initial data} that $f\in L_{\mathrm{loc}}^{2}([0,T);H^{-1}(\Omega)^{3})$, and using~\eqref{assumption on coefficients}, we deduce the following uniform bounds with respect to $N\in\mathbb{N}$: 
	\begin{subequations}
	\begin{align}
		&(v^{N})_{N}\mbox{ is bounded in }L^{2}(0,T^{\prime};H_{0,\mathrm{div}}^{1}(\Omega))\mbox{ and in }L^{\infty}(0,T^{\prime};L_{\mathrm{div}}^{2}(\Omega)),
		\label{bounds of discrete velocity}\\
		&(S^{N})_{N}\mbox{ is bounded in }L^{2}(0,T^{\prime};H_{\mathrm{sym,Tr}}^{1}(\Omega))\mbox{ and in }L^{\infty}(0,T^{\prime};L_{\mathrm{sym,Tr}}^{2}(\Omega)),
		\label{bounds of discrete stress}\\
		&(\varphi^{N})_{N}\mbox{ is bounded in }L^{\infty}(0,T^{\prime};H^{1}(\Omega)),
		\label{bounds of discrete phase-field}\\
		&(\nabla\mu^{N})_{N}\mbox{ is bounded in }L^{2}(0,T^{\prime};(L^{2}(\Omega))^3).
		\label{bounds of discrete chemical potenial}
	\end{align}
	\end{subequations}
	for all $0<T^{\prime}<T$. Moreover, from the uniform bounds~\eqref{bounds of discrete phase-field} on $(\nabla\varphi^{N})_{N}$ and~\eqref{bounds of discrete chemical potenial} on $(\nabla\mu^{N})_{N}$, we can deduce the following estimate on $(\mu^{N})_{N}$ with the aid of~\eqref{estimate on the discrete version of the derivative of double well potential and mean value of mu} 
	\begin{align}
		&\int_{0}^{T^{\prime}}\abs{\int_{\Omega}\mu^{N} \dd{x}}\dd{t}\leq C\cdot T^{\prime}\mbox{ for all }0<T^{\prime}<T,
		\label{bound of mean value of discrete chemical potential}
	\end{align}
	where $C>0$ is a constant.
	\paragraph{Step 4: Immediate convergence results.}
	By a classical diagonalization argument, we can extract a not relabeled subsequence and a limit quadruplet $(v,S,\varphi,\mu)$ such that for all $0<T^{\prime}<T$:
	\begin{subequations}\label{immediate convergence results for piecewise linear interpolations}
	\begin{align}
		v^{N}\rightharpoonup v&\mbox{ in }L^{2}(0,T^{\prime};H^{1}(\Omega)),
		\label{weak convergence of piecewise linear interpolation velocity space derivative}\\
		v^{N}\overset{*}{\rightharpoonup}v&\mbox{ in }L^{\infty}(0,T^{\prime};L_{\mathrm{div}}^{2}(\Omega)),
		\label{weak convergence of piecewise linear interpolation velocity}\\
		S^{N}\rightharpoonup S&\mbox{ in }L^{2}(0,T^{\prime};H^{1}(\Omega)),
		\label{weak convergence of piecewise linear interpolation stress space derivative}\\
		S^{N}\overset{*}{\rightharpoonup}S&\mbox{ in }L^{\infty}(0,T^{\prime};L_{\mathrm{sym,Tr}}^{2}(\Omega)),
		\label{weak convergence of piecewise linear interpolation stress}\\
		\varphi^{N}\overset{*}{\rightharpoonup}\varphi&\mbox{ in }L^{\infty}(0,T^{\prime};H^{1}(\Omega)),
		\label{weak convergence of piecewise linear interpolation phase-field}\\
		\mu^{N}\rightharpoonup\mu&\mbox{ in }L^{2}(0,T^{\prime};H^{1}(\Omega)).
		\label{weak convergence of piecewise linear interpolation chemical potential}
	\end{align}
	\end{subequations}
	\paragraph{Step 5: Improved convergence results.} 
	Now we derive additional strong convergence results with the help of the Aubin-Lions lemma, see \cite[][Chapter 7.3]{roubicekNonlinearPartialDifferential2013} for details. This will require estimates on the time derivatives. To this end, we define the piecewise linear interpolations, i.e.,  
	\begin{equation}
		\tilde{F}^{N}(t):=\frac{(k+1)h-t}{h}F_{k}+\frac{t-kh}{h}F_{k+1}
		\label{piecewise linear interpolation}
	\end{equation}
	for $t\in[t_{k},t_{k+1}]$ and $F\in\{\rho v,\varphi\}$; if $F=\rho v$, then $\tilde{F}^{N}:=\widetilde{\rho v}^{N}$ with the piecewise linear interpolation of $\rho v$ and $F_{k}=\rho_{k}v_{k}$ with the solutions $\rho_{k},v_{k}$ at step $k$. By definition, $\partial_{t}\tilde{F}^{N}=\partial_{t,h}^{-}F^{N}$ and the bounds on the piecewise constant interpolations can be transferred to the piecewise linear interpolations, i.e., we have
	\begin{subequations}
		\begin{align}
			&(\widetilde{\rho v}^{N})_{N}\mbox{ is bounded in }L^{2}(0,T;W_{0}^{1,\frac{3}{2}}(\Omega))\mbox{ and in }L^{\infty}(0,T;L^{2}(\Omega)),
			\label{bounds of piecewise linear interpolation momentum}\\
			&(\tilde{\varphi}^{N})_{N}\mbox{ is bounded in }L^{\infty}(0,T;H^{1}(\Omega)).
			\label{bounds of piecewise linear interpolation phase-field}
		\end{align}
	\end{subequations} Moreover, we can estimate the difference between the piecewise constant interpolation and the piecewise linear interpolation pointwise, i.e., we have
	\begin{equation}
		\abs{\tilde{F}^{N}(t,x)-F^{N}(t,x)}\leq h\abs{\partial_{t}\tilde{F}^{N}(t,x)}\mbox{ almost everywhere in }\Omega\times(0,T).
		\label{difference between piecewise constant interpolation and piecewise linear interpolation}
	\end{equation}
	
    We first derive a uniform bound on $(\partial_{t}\tilde{\varphi}^{N})_{N}$ from \eqref{weak formulation piecewise constant first CH}. Since $(v^{N}\varphi_{h}^{N})_{N}$ and $(\nabla\mu^{N})_{N}$ are bounded in $L^{2}(0,T;L^{2}(\Omega))$, the test function $\zeta$ can be chosen with $\nabla\zeta\in L^{2}(0,T;L^{2}(\Omega))$. Hence, $\zeta\in L^{2}(0,T;H^{1}(\Omega))$ is sufficient  by comparison in \eqref{weak formulation piecewise constant first CH}. This implies that $(\partial_{t}\tilde{\varphi}^{N})_{N}$ is bounded in $L^{2}(0,T;H^{-1}(\Omega))$. Besides, $(\tilde{\varphi}^{N})_{N}$ is bounded in $L^{\infty}(0,T;H^{1}(\Omega))$ from \eqref{bounds of piecewise linear interpolation phase-field}. Therefore, by Aubin-Lions Lemma, we get the strong convergence
	\begin{equation*}
		\tilde{\varphi}^{N}\to\tilde{\varphi}\mbox{ in }L^{2}(0,T^{\prime};L^{2}(\Omega))
	\end{equation*}
	as $N\to\infty$, for all $0<T^{\prime}<T$ and some $\tilde{\varphi}\in L^{\infty}(0,T;L^{2}(\Omega))$. The bound on $(\partial_{t}\tilde{\varphi}^{N})_{N}$ in $L^{2}(0,T;H^{-1}(\Omega))$ and \eqref{difference between piecewise constant interpolation and piecewise linear interpolation} give us that
	\begin{equation*}
		\tilde{\varphi}^{N}-\varphi^{N}\to0\mbox{ in }L^{2}(0,T^{\prime};H^{-1}(\Omega)).
	\end{equation*}
	as $N\to\infty$. Furthermore, we obtain
	\begin{equation*}
		\varphi^{N}\to\tilde{\varphi}\mbox{ in }L^{2}(0,T^{\prime};H^{-1}(\Omega))
	\end{equation*}
	as $N\to\infty$, for all $0<T^{\prime}<T$, which implies that $\varphi=\tilde{\varphi}$. Besides, by interpolation of Bochner spaces, see \cite[][Theorem 5.1.2]{berghInterpolationSpacesIntroduction1976} for details, the strong convergence of $(\varphi^{N})_{N}$ in $L^{2}(0,T^{\prime};H^{-1}(\Omega))$ and its bound in $L^{2}(0,T^{\prime};H^{1}(\Omega))$ imply the strong convergence 
	\begin{equation}
		\varphi^{N}\to\varphi\mbox{ in }L^{2}(0,T^{\prime};L^{2}(\Omega)).
		\label{strong convergence of piecewise constant interpolation phase-field which will be improved}
	\end{equation}
    
	Next, we want to improve the strong convergence of $(\varphi^{N})_{N}$ by interpolation inequalities and find convergence result for $(\mathrm{D}\mathcal{E}_{\mathrm{pf},\kappa}(\varphi^{N}))_{N}$. By Lemma~\ref{estimate from discrete second CH}, we obtain that $(\mathrm{D}\mathcal{E}_{\mathrm{pf},\kappa}(\varphi^{N}))_{N}$ is bounded in $L^{2}(0,T^{\prime};L^{2}(\Omega))$ for all $0<T^{\prime}<T$. Moreover, in~\eqref{weak formulation piecewise constant second CH}, the right-hand side is actually $(\mathrm{D} \mathcal{E}_{\mathrm{pf},\kappa}(\varphi^{N}))_{N}$, while the left-hand side weakly converges to $\mu+\kappa\varphi$ in $L^{2}(0,T^{\prime};L^{2}(\Omega))$. Hence, we have
	\begin{equation}
	    \begin{aligned}
		\mathrm{D} \mathcal{E}_{\mathrm{pf},\kappa}(\varphi^N)
		=-\Delta\varphi^{N}+W_{\kappa}^{\prime}(\varphi^{N})
		&=\mu^{N}+\frac{\kappa}{2}(\varphi^{N}+\varphi_{h}^{N})
		\nonumber\\
        &\rightharpoonup \mu+\kappa\varphi
		\mbox{ in }L^{2}(0,T^{\prime};L^{2}(\Omega)).
        \label{weak convergence of subdifferential of phase field energy}
	\end{aligned}
	\end{equation}
	This also implies that
	\begin{equation}
		\lVert \mathrm{D} \mathcal{E}_{\mathrm{pf},\kappa}(\varphi^{N})\rVert_{L^{2}(0,T^{\prime};L^{2}(\Omega))}\leq C
		\label{estimate on the subdifferential on the convex phase-field energy for discrete solution}
	\end{equation}
	for all $N\in\mathbb{N}$. Furthermore, by~\eqref{estimate on element in subdifferntial domain}, we have 
	\begin{equation*}
		\lVert \varphi^{N}\rVert_{H^2}^{2}
		+\lVert W_{\kappa}^{\prime}(\varphi^{N})\rVert_{L^2}^{2}
		+\int_{\Omega}W_{\kappa}^{\prime\prime}(\varphi^{N})\abs{\nabla \varphi^{N}}^2\dd{x}
		\leq C\left(\lVert \mathrm{D} \mathcal{E}_{\mathrm{pf},\kappa}(\varphi^{N})\rVert_{L^2}^{2}+\lVert \varphi^{N}\rVert_{L^2}^{2}+1\right),
	\end{equation*}
	The boundedness of $(\mathrm{D} \mathcal{E}_{\mathrm{pf},\kappa}(\varphi^{N}))_{N}$ together with the bounds on $(\varphi^{N})_{N}$ in $L^{2}(0,T^{\prime};L^{2}(\Omega))$ implies that
	\begin{equation}
		(\varphi^{N})_{N}\mbox{ is bounded in }L^{2}(0,T^{\prime};H^{2}(\Omega)).
		\label{extra bound of discrete phase-field}
	\end{equation}
	Hence, by interpolation of Bochner spaces and the strong convergence of $(\varphi^{N})_{N}$ in $L^{2}(0,T^{\prime};L^{2}(\Omega))$ from \eqref{strong convergence of piecewise constant interpolation phase-field which will be improved}, we obtain
	\begin{equation}
		\varphi^{N}\to\varphi\mbox{ in }L^{2}(0,T^{\prime};H^{1}(\Omega)).
		\label{strong convergence of piecewise constant interpolation phase-field}
	\end{equation}
	Since $(\rho^{N})_{N}$ depends in an affine linear way on $(\varphi^{N})_{N}$ by formula~\eqref{assumption on density}, we have the same convergence result for $(\rho^{N})_{N}$.
    
	Now, we derive a uniform bound on $(\partial_{t}(\widetilde{\rho v}^{N}))_{N}$ from~\eqref{weak formulation piecewise constant velocity}. For this, we argue by comparison in~\eqref{weak formulation piecewise constant velocity} and therefore, we now derive uniform estimates for all the other terms in~\eqref{weak formulation piecewise constant velocity}. In particular, for all $0<T^{\prime}<T$, we have the following estimates:\\
	Notice that $\rho_{h}^{N}v^{N}\otimes v^{N}$ contains terms of the form $\rho_{h}^{N}v_{i}^{N}v_{j}^{N}$. Since $(v^{N})_{N}$ is bounded both in $L^{\infty}(0,T^{\prime};L^{2}(\Omega))$ by~\eqref{bounds of discrete velocity} and $L^{2}(0,T^{\prime};H^{1}(\Omega))\hookrightarrow L^{2}(0,T^{\prime};L^{6}(\Omega))$ and $(\rho_{h}^{N})_{N}$ is bounded in $L^{\infty}(\Omega\times(0,T^{\prime}))$, we obtain
	\begin{equation*}
	    \begin{aligned}
		\lVert \rho_{h}^{N}v^{N}\otimes v^{N}\rVert_{L^{2}(0,T^{\prime};L^{\frac{3}{2}}(\Omega))}
		&\leq C\lVert \rho_{h}^{N}\rVert_{L^{\infty}(\Omega\times(0,T^{\prime}))}
		\lVert v^{N}\rVert_{L^{\infty}(0,T^{\prime};L^{2}(\Omega))}
		\lVert v^{N}\rVert_{L^{2}(0,T^{\prime};L^{6}(\Omega))}
		\\
		&\leq C\lVert \rho^{N}\rVert_{L^{\infty}(\Omega\times(0,T^{\prime}))}
		\lVert v^{N}\rVert_{L^{\infty}(0,T^{\prime};L^{2}(\Omega))}
		\lVert v^{N}\rVert_{L^{2}(0,T^{\prime};H^{1}(\Omega))},
    \end{aligned}
	\end{equation*}
	where we used that $\lVert \rho_{h}^{N}\rVert_{L^{\infty}(\Omega\times(0,T^{\prime}))}=\lVert \rho^{N}\rVert_{L^{\infty}(\Omega\times(0,T^{\prime}))}$.\\
	Since $(v^{N})_{N}$ is bounded in $L^{2}(0,T^{\prime};H^{1}(\Omega))$ by~\eqref{bounds of discrete velocity} and $(\nu(\varphi_{h}^{N}))_{N}$ is bounded from above by a positive constant thanks to~\eqref{assumption on coefficients}, we obtain
	\begin{equation*}
		\lVert 2\nu(\varphi_{h}^{N})\sym{\nabla v^{N}}\rVert_{L^{2}(0,T^{\prime};L^{2}(\Omega))}
		\leq C\lVert v^{N}\rVert_{L^{2}(0,T^{\prime};H^{1}(\Omega))}.
	\end{equation*}
	Since $(S^{N})_{N}$ is bounded in $L^{2}(0,T^{\prime};H^{1}(\Omega))$ by~\eqref{bounds of discrete stress} and $(\eta(\varphi_{h}^{N}))_{N}$ is bounded from above by a positive constant again by~\eqref{assumption on coefficients}, we have
	\begin{equation*}
		\lVert \eta(\varphi_{h}^{N})S^{N}\rVert_{L^{2}(0,T^{\prime};L^{2}(\Omega))}
		\leq C\lVert S^{N}\rVert_{L^{2}(0,T^{\prime};H^{1}(\Omega))}.
	\end{equation*}
	Notice that $v^{N}\nabla\mu^{N}$ consists of terms of the form $v_{i}^{N}\partial_{x_{j}}\mu^{N}$. Since $(\nabla\mu^{N})_{N}$ is bounded in $L^{2}(0,T^{\prime};L^{2}(\Omega))$ by~\eqref{bounds of discrete chemical potenial} and $(v^{N})_{N}$ is bounded in $L^{\infty}(0,T^{\prime};L^{2}(\Omega))$ and in $L^{2}(0,T^{\prime};H^{1}(\Omega))\hookrightarrow L^{2}(0,T^{\prime};L^{6}(\Omega))$ by~\eqref{bounds of discrete velocity}, we find
	\begin{equation*}
		\lVert v_{i}^{N}\partial_{x_{j}}\mu^{N}\rVert_{L^{2}(0,T^{\prime};L^{1}(\Omega))}
		\leq C\lVert v_{i}^{N}\rVert_{L^{\infty}(0,T^{\prime};L^{2}(\Omega))}
		\lVert \partial_{x_{j}}\mu^{N}\rVert_{L^{2}(0,T^{\prime};L^{2}(\Omega))}
	\end{equation*}
	and
	\begin{equation*}
		\lVert v_{i}^{N}\partial_{x_{j}}\mu^{N}\rVert_{L^{1}(0,T^{\prime};L^{\frac{3}{2}}(\Omega))}
		\leq C \lVert v_{i}^{N}\rVert_{L^{2}(0,T^{\prime};L^{6}(\Omega))}
		\lVert \partial_{x_{j}}\mu^{N}\rVert_{L^{2}(0,T^{\prime};L^{2}(\Omega))},
	\end{equation*}
	which implies that $(v_{i}^{N}\partial_{x_{j}}\mu^{N})_{N}$ is bounded in $L^{2}(0,T^{\prime};L^{1}(\Omega))$ and in $L^{1}(0,T^{\prime};L^{\frac{3}{2}}(\Omega))$. Now, by interpolation of Bochner spaces, we obtain here
	\begin{equation*}
		\left(L^{2}(0,T^{\prime};L^{1}(\Omega)),L^{1}(0,T^{\prime};L^{\frac{3}{2}}(\Omega))\right)_{\theta}=L^{\frac{8}{7}}(0,T^{\prime};L^{\frac{4}{3}}(\Omega)),
	\end{equation*}
	with $\theta=\frac{3}{4}$. This implies
	\begin{equation*}
		\lVert v^{N}\nabla\mu^{N}\rVert_{L^{\frac{8}{7}}(0,T^{\prime};L^{\frac{4}{3}}(\Omega))}
		\leq C\lVert v^{N}\nabla\mu^{N}\rVert_{L^{2}(0,T^{\prime};L^{1}(\Omega))}^{1-\theta} 
		\lVert v^{N}\nabla\mu^{N}\rVert_{L^{1}(0,T^{\prime};L^{\frac{3}{2}}(\Omega))}^{\theta}.
	\end{equation*}
	Since $(\mu^{N})_{N}$ is bounded in $L^{2}(0,T^{\prime};H^{1}(\Omega))\hookrightarrow L^{2}(0,T^{\prime};L^{6}(\Omega))$ by \eqref{bounds of discrete chemical potenial} and \eqref{bound of mean value of discrete chemical potential}. Moreover, $(\nabla\varphi_{h}^{N})_{N}$ is bounded in $L^{\infty}(0,T^{\prime};L^{2}(\Omega))$. Thus, we obtain
	\begin{equation*}
	    \begin{aligned}
		\lVert \mu^{N}\nabla\varphi_{h}^{N}\rVert_{L^{2}(0,T^{\prime};L^{\frac{3}{2}}(\Omega))}
		&\leq C\lVert \mu^{N}\rVert_{L^{2}(0,T^{\prime};L^{6}(\Omega))}
		\lVert \nabla\varphi_{h}^{N}\rVert_{L^{\infty}(0,T^{\prime};L^{2}(\Omega))}
		\\
		&\leq C\lVert \mu^{N}\rVert_{L^{2}(0,T^{\prime};H^{1}(\Omega))}
		\lVert \varphi^{N}\rVert_{L^{\infty}(0,T^{\prime};H^{1}(\Omega))},
	\end{aligned}
	\end{equation*}
	where we again used $\lVert \nabla\varphi_{h}^{N}\rVert_{L^{\infty}(0,T^{\prime};L^{2}(\Omega))}=\lVert \nabla\varphi^{N}\rVert_{L^{\infty}(0,T^{\prime};L^{2}(\Omega))}$.
    
	Altogether, all these terms are bounded in $L^{\frac{8}{7}}(0,T^{\prime};L^{\frac{4}{3}}(\Omega))$, so we can choose test function $\Phi$ in~\eqref{weak formulation piecewise constant velocity} such that 
	\begin{equation*}
		\Phi,\nabla\Phi\in \left(L^{\frac{8}{7}}(0,T^{\prime};L^{\frac{4}{3}}(\Omega))\right)^{\prime}=L^{8}(0,T^{\prime};L^{4}(\Omega)).
	\end{equation*} 
	Hence, choosing test functions $\Phi\in L^{8}(0,T^{\prime};W^{1,4}(\Omega))$ is sufficient. This implies that $(\partial_{t}(\widetilde{\rho v}^{N}))_{N}$ is bounded in $L^{\frac{8}{7}}(0,T^{\prime};W^{-1,4}(\Omega))$. Besides, $(\widetilde{\rho v}^{N})_{N}$ is bounded in $L^{2}(0,T^{\prime};W^{1,\frac{3}{2}}(\Omega))$ from~\eqref{bounds of piecewise linear interpolation momentum}. Therefore, by the Aubin-Lions Lemma, we obtain the strong convergence
	\begin{equation*}
		\widetilde{\rho v}^{N}\to \widetilde{\rho v}\text{ in }L^{2}(0,T^{\prime};L^{2}(\Omega)),
	\end{equation*}
	as $N\to\infty$ for some $\widetilde{\rho v}\in L^{\infty}(0,T^{\prime};L^{2}(\Omega))\cap L^{2}(0,T^{\prime};W^{1,\frac{3}{2}}(\Omega))$. Moreover, \eqref{difference between piecewise constant interpolation and piecewise linear interpolation} and the bound on $(\partial_{t}(\widetilde{\rho v}^N))_{N}$ in $L^{\frac{8}{7}}(0,T^{\prime};W^{-1,4}(\Omega))$ obtained above imply that
	\begin{equation*}
		\widetilde{\rho v}^{N}-\rho^{N}v^{N}\to0\mbox{ in }L^{\frac{8}{7}}(0,T^{\prime};W^{-1,4}(\Omega)).
	\end{equation*}
	as $N\to\infty$. Thus, we obtain
	\begin{equation*}
		\rho^{N}v^{N}\to\widetilde{\rho v}\mbox{ in }L^{\frac{8}{7}}(0,T^{\prime};W^{-1,4}(\Omega)),
	\end{equation*}
	as $N\to\infty$ which implies $\widetilde{\rho v}=\rho v$. Moreover, by interpolation of Bochner spaces, the strong convergence of $(\rho^{N}v^{N})_{N}$ in $L^{\frac{8}{7}}(0,T^{\prime};W^{-1,4}(\Omega))$ and the bound of $(\rho^{N}v^{N})_{N}$ in $L^{2}(0,T^{\prime};W^{1,\frac{3}{2}}(\Omega))$ imply the strong convergence in $L^{\frac{16}{11}}(0,T^{\prime};L^{\frac{24}{11}}(\Omega))$ due to the choice of $\theta=\frac{1}{2}$ for the interpolation. By continuous embedding $L^{\frac{16}{11}}(0,T^{\prime};L^{\frac{24}{11}}(\Omega))\hookrightarrow L^{\frac{16}{11}}(0,T^{\prime};L^{2}(\Omega))$, we get the strong convergence in $L^{\frac{16}{11}}(0,T^{\prime};L^{2}(\Omega))$. By repeating this argument with $L^{\frac{16}{11}}(0,T^{\prime};L^{2}(\Omega))$ and $L^{\infty}(0,T^{\prime};L^{2}(\Omega))$, we arrive at the strong convergence 
	\begin{equation}
		\rho^{N}v^{N}\to\rho v\text{ in }L^{2}(0,T^{\prime};L^{2}(\Omega)).
		\label{strong convergence of product of piecewise constant interpolation velocity and phase-field}
	\end{equation}
	
    Now, we want to obtain the strong convergence of $(v^{N})_{N}$ from~\eqref{strong convergence of product of piecewise constant interpolation velocity and phase-field}. To this end, first notice that 
	\begin{equation*}
	    \begin{aligned}
		&\abs{\int_{0}^{T^{\prime}}\int_{\Omega}\rho^{N}\abs{v^{N}}^{2} \dd{x}\dd{t}-\int_{0}^{T^{\prime}}\int_{\Omega}\rho\abs{v}^{2} \dd{x}\dd{t}} 
		\\
		\leq&\abs{\int_{0}^{T^{\prime}}\int_{\Omega}(\rho^{N}v^{N}-\rho v)v^{N} \dd{x}\dd{t}}+\abs{\int_{0}^{T^{\prime}}\int_{\Omega}\rho v(v^{N}-v) \dd{x}\dd{t}},
	\end{aligned}
	\end{equation*}
	where the first term in the last line tends to $0$ due to the strong convergence of $(\rho^{N}v^{N})_{N}$ in $L^{2}(0,T^{\prime};L^{2}(\Omega))$ by~\eqref{strong convergence of product of piecewise constant interpolation velocity and phase-field} and the second term tends to $0$, thanks to the weak convergence of $(v^{N})_{N}$ in $L^{2}(0,T^{\prime};L^{2}(\Omega))$ by \eqref{weak convergence of piecewise linear interpolation velocity}. Besides, in combination with $(\rho^{N})^{\frac{1}{2}}v^{N}\rightharpoonup(\rho)^{\frac{1}{2}}v$ in $L^{2}(0,T^{\prime};L^{2}(\Omega))$, we obtain that $(\rho^{N})^{\frac{1}{2}}v^{N}\to(\rho)^{\frac{1}{2}}v$ in $L^{2}(0,T^{\prime};L^{2}(\Omega))$. Since $(\rho^{N})_{N}\to\rho$ almost everywhere in $\Omega\times(0,T)$ and $0<\rho_{1}\leq\rho^{N}\leq\rho_{2}$, we obtain
	\begin{equation}
		v^{N}=(\rho^{N})^{-\frac{1}{2}}(\rho^{N})^{\frac{1}{2}}v^{N}\to v\mbox{ in }L^{2}(0,T^{\prime};L^{2}(\Omega)).
		\label{strong convergence of piecewise constant interpolation velocity}
	\end{equation}
	\paragraph{Step 6: Limit passage in the weak formulation~\eqref{weak formulation piecewise constant velocity}-\eqref{weak formulation piecewise constant second CH}.}
	Next we want to pass \eqref{weak formulation piecewise constant velocity}-\eqref{weak formulation piecewise constant second CH} to the limit $N\to\infty$. Notice that for all divergence-free test functions $\Phi$, we have the following relation
	\begin{equation*}
	    \begin{aligned}
		\int_{0}^{T}\int_{\Omega}\mu^{N}\nabla\varphi_{h}^{N}\cdot\Phi \dd{x}\dd{t}
		=&-\int_{0}^{T}\int_{\Omega}\nabla\mu^{N}\varphi_{h}^{N}\cdot\Phi \dd{x}\dd{t}
		\\
		\rightarrow&-\int_{0}^{T}\int_{\Omega}\nabla\mu\varphi\cdot\Phi \dd{x}\dd{t}
		=\int_{0}^{T}\int_{\Omega}\mu\nabla\varphi\cdot\Phi \dd{x}\dd{t}.
	\end{aligned}
	\end{equation*}
	Then the limit passage of~\eqref{weak formulation piecewise constant velocity}-\eqref{weak formulation piecewise constant first CH} follows from the convergence results~\eqref{convergence of approximating initial datum of phase-field}, \eqref{immediate convergence results for piecewise linear interpolations}, \eqref{strong convergence of piecewise constant interpolation phase-field}-\eqref{strong convergence of piecewise constant interpolation velocity}. 
	To pass to the limit in \eqref{weak formulation piecewise constant second CH}, we use \eqref{weak convergence of subdifferential of phase field energy} and the fact that $\mathrm{D}\mathcal{E}_{\mathrm{pf},\kappa}$ is a maximal monotone operator and apply~\cite[][Proposition IV.1.6]{Showalter1997}. Hence, we conclude~\eqref{weak second CH TPP}.
	\paragraph{Step 7: The partial energy estimates.}
    For all $N\in\mathbb{N}$, for all $t\in(t_{k},t_{k+1})$ and for all $k\in\{0,1,\ldots,N\}$, define $E_{\mathrm{k}}^{N}(t)$ to be the piecewise linear interpolation of the kinetic energy, i.e.,
    \begin{equation}
        E_{\mathrm{k}}^{N}(t):=\frac{(k+1)h-t}{h}\mathcal{E}_{\mathrm{k}}(\varphi_{k},v_{k})+\frac{t-kh}{h}\mathcal{E}_{\mathrm{k}}(\varphi_{k+1},v_{k+1})
     \end{equation}
    and define $D_{\mathrm{s}}^{N}(t)$ to be the piecewise constant dissipation as 
    \begin{equation}
        \begin{aligned}
        D_{\mathrm{s}}^{N}(t):=&
		\int_{\Omega}2\nu(\varphi_{h}^{N})\abs{\sym{\nabla v^{N}}}^{2} \dd{x}
		+\int_{\Omega}\eta(\varphi_{h}^{N})S^{N}:\nabla v^{N} \dd{x}
        \nonumber\\
		&-\int_{\Omega}\mu^{N}(\nabla\varphi_{h}^{N}\cdot v^{N}) \dd{x}
        -\langle f^{N}, v^{N}\rangle_{H^{1}}.
    \end{aligned}
    \end{equation}
    for all $t\in[t_{k},t_{k+1})$. Then from \eqref{discrete velocity energy}, we have
    \begin{equation}
        -\dv{t}E_{\mathrm{k}}^{N}(t)=-\frac{1}{h}\mathcal{E}_{\mathrm{k}}(\varphi_{k},v_{k})+\frac{1}{h}\mathcal{E}_{\mathrm{k}}(v_{k+1})\geq D_{\mathrm{s}}^{N}(t).
        \label{piecewise linear kinetic energy dissipation time derivative estimate}
    \end{equation}
    Now, we multiply both sides of~\eqref{piecewise linear kinetic energy dissipation time derivative estimate} by arbitrary $\phi\in\tilde{C}([0,T^{\prime}])$, integrate over time, and use integration by parts. This results in
    \begin{equation}
        \mathcal{E}_{k}(\varphi_{0},v_{0})
        \geq-\int_{0}^{T^{\prime}}\phi^{\prime}(t) E_{\mathrm{k}}^{N}(t)\dd{t}
        +\int_{0}^{T^{\prime}}\phi(t)D_{\mathrm{s}}^{N}(t)\dd{t}.
        \label{limit passage in discrete partial kinetic energy estimate part one}
    \end{equation}
    By the strong convergence of $(v^{N})_{N}$ in $L^{2}(0,T^{\prime};L^{2}(\Omega))$ from~\eqref{strong convergence of piecewise constant interpolation velocity} and the strong convergence of $(\varphi^{N})_{N}$ in $L^{2}(0,T^{\prime};H^{1}(\Omega))$ from~\eqref{strong convergence of piecewise constant interpolation phase-field}, we derive
    \begin{equation}
        \lim_{N\to\infty}-\int_{0}^{T^{\prime}}\phi^{\prime}(t) E_{\mathrm{k}}^{N}(t)\dd{t}\to-\int_{0}^{T^{\prime}}\phi^{\prime}(t) \mathcal{E}_{\mathrm{k}}(\varphi(t),v(t))\dd{t}.
        \label{limit passage in discrete partial kinetic energy estimate part two}
    \end{equation}
    Next, by the weak convergence of $(v^{N})_{N}$ in $L^{2}(0,T^{\prime};H^{1}(\Omega))$ from~\eqref{weak convergence of piecewise linear interpolation velocity space derivative}, the strong convergence of  $(\varphi^{N})_{N}$ in $L^{2}(0,T^{\prime};H^{1}(\Omega))$ from~\eqref{strong convergence of piecewise constant interpolation phase-field} and assumption~\eqref{assumption on coefficients}, we obtain
    \begin{equation}
        \int_{0}^{T^{\prime}}\phi\int_{\Omega}2\nu(\varphi)\abs{\sym{\nabla v}}^{2}\dd{x}\dd{t}
        \leq\liminf_{N\to\infty}\int_{0}^{T^{\prime}}\phi\int_{\Omega}2\nu(\varphi_{h}^{N})\abs{\sym{\nabla v^{N}}}^{2}\dd{x}\dd{s}.
        \label{limit passage in discrete partial kinetic energy estimate part three}
    \end{equation}
    Moreover, by the strong convergence of $(v^{N})_{N}$ in $L^{2}(0,T^{\prime};L^{2}(\Omega))$ from~\eqref{strong convergence of piecewise constant interpolation velocity}, the strong convergence of $(\varphi^{N})_{N}$ in $L^{2}(0,T^{\prime};H^{1}(\Omega))$ from~\eqref{strong convergence of piecewise constant interpolation phase-field}, the weak convergence of  $(S^{N})_{N}$ in $L^{2}(0,T^{\prime};H^{1}(\Omega))$ from~\eqref{weak convergence of piecewise linear interpolation stress space derivative} and assumption~\eqref{assumption on coefficients}, we derive 
    \begin{equation}
    \begin{aligned}
        &\lim_{N\to\infty}\int_{0}^{t}\phi\int_{\Omega}\eta(\varphi_{h}^{N})S^{N}:\nabla v^{N} \dd{x}\dd{s}
        \\
        =&\lim_{N\to\infty}\left(-
        \int_{0}^{t}\phi\int_{\Omega}\eta^{\prime}(\varphi_{h}^{N})v^{N}\otimes\nabla\varphi_{h}^{N}:S^{N} \dd{x}\dd{s}
        -\int_{0}^{t}\phi\int_{\Omega}\eta(\varphi_{h}^{N})\divge{S^{N}}\cdot v^{N} \dd{x}\dd{s}
        \right)
        \\
        =&-\int_{0}^{t}\phi\int_{\Omega}\eta^{\prime}(\varphi)v\otimes\nabla\varphi:S \dd{x}\dd{s}
        -\int_{0}^{t}\phi\int_{\Omega}\eta(\varphi)\divge{S}\cdot v \dd{x}\dd{s}
        \\
        =&\int_{0}^{t}\phi\int_{\Omega}\eta(\varphi)S:\nabla v \dd{x}\dd{s}
        \label{limit passage in discrete partial kinetic energy estimate part four}
    \end{aligned}
    \end{equation}
    with the help of integration by parts and Assumption~\eqref{assumption on coefficients}. Now, notice that we have the embedding $L^{2}(0,T^{\prime};H^{1}(\Omega))\hookrightarrow L^{2}(0,T^{\prime};L^{6}(\Omega))$, so we have the weak convergence of $(\mu^{N})_{N}$ in $L^{2}(0,T^{\prime};L^{6}(\Omega))$ from~\eqref{weak convergence of piecewise linear interpolation chemical potential}. Now, we use the interpolation of Bochner spaces to improve the strong convergence of $(v^{N})_{N}$ and $(\nabla\varphi^{N})_{N}$. By choosing $\theta=\frac{2}{3}$, the strong convergence of $(v^{N})_{N}$ in $L^{2}(0,T^{\prime};L^{2}(\Omega))$ from~\eqref{strong convergence of piecewise constant interpolation velocity} and bound of $(v^{N})_{N}$ in $L^{\infty}(0,T^{\prime};L^{2}(\Omega))$ from~\eqref{bounds of discrete velocity} result in the strong convergence of $(v^{N})_{N}$ in $L^{6}(0,T^{\prime},L^{2}(\Omega))$. Then by choosing $\theta=\frac{1}{4}$ in the interpolation of Bochner spaces, the strong convergence of $(v^{N})_{N}$ in $L^{6}(0,T^{\prime},L^{2}(\Omega))$ and the bound of $(v^{N})_{N}$ in $L^{2}(0,T^{\prime};H^{1}(\Omega))\hookrightarrow L^{2}(0,T^{\prime};L^{6}(\Omega))$ from~\eqref{bounds of discrete velocity} yields the strong convergence of $(v^{N})_{N}$ in $L^{4}(0,T^{\prime};L^{\frac{12}{5}}(\Omega))$. A similar argument gives us the strong convergence of $(\nabla\varphi^{N})_{N}$ in $L^{4}(0,T^{\prime};L^{\frac{12}{5}}(\Omega))$. Therefore, by the weak-strong-strong convergence, we obtain
    \begin{equation}
        \lim_{N\to\infty}\int_{0}^{T^{\prime}}\phi\int_{\Omega}\mu^{N}(\nabla \varphi_{h}^{N}\cdot v^{N})\dd{x}\dd{t}
        =\int_{0}^{T^{\prime}}\phi\int_{\Omega}\mu(\nabla \varphi\cdot v)\dd{x}\dd{t}
        \label{limit passage in discrete partial kinetic energy estimate part five}
    \end{equation}
    By the weak convergence of $(v^{N})_{N}$ in $L^{2}(0,T^{\prime};H^{1}(\Omega))$ from~\eqref{weak convergence of piecewise linear interpolation velocity space derivative}, we have
    \begin{equation}
        \lim_{N\to+\infty}\int_{0}^{t}\phi\langle f^{N}, v^{N}\rangle_{H^{1}}\dd{s}
        =\int_{0}^{t}\phi\langle f, v\rangle_{H^{1}}\dd{s}.
        \label{limit passage in discrete partial kinetic energy estimate part six}
    \end{equation}
    Inserting~\eqref{limit passage in discrete partial kinetic energy estimate part two}-\eqref{limit passage in discrete partial kinetic energy estimate part six} into~\eqref{limit passage in discrete partial kinetic energy estimate part one} gives~\eqref{partial energy inequality kinetic}. Similarly, we also prove~\eqref{partial energy inequality phase field}.
    \paragraph{Step 8: The evolutionary variational inequality for the internal stress.}
    Now, we show the evolutionary variational inequality for the internal stress. To this end, we derive
    \begin{equation}
        \begin{aligned}
        \mathcal{E}_{\mathrm{e}}(S_{0})\geq&
        -\int_{0}^{T^{\prime}}\phi^{\prime}\mathcal{E}_{e}(S^{N})\dd{t}
        +\int_{0}^{T^{\prime}}\phi\int_{\Omega}\gamma\abs{\nabla S^{N}}^{2}\dd{x}\dd{t}
        \\
        &+\int_{0}^{T^{\prime}}\phi\langle \xi^{N},S^{N}\rangle_{H_{\mathrm{sym,Tr}}^{1}}\dd{t}
        -\int_{0}^{T^{\prime}}\phi\int_{\Omega} \eta(\varphi_{h}^{N})S^{N}:\sym{\nabla v^{N}}\dd{x}\dd{t}
        \label{partial energy discrete stress energy}
    \end{aligned}
    \end{equation}
    from~\eqref{discrete stess energy} by using the same argument as for~\eqref{limit passage in discrete partial kinetic energy estimate part one}. Here, $\phi\in \tilde{C}([0,T^{\prime}])$ is arbitrary. Then, for all $\tilde{S}\in C_{0,\mathrm{sym,Tr}}^{\infty}(\Omega\times[0,T))$, we choose $\tilde{S}^{k+1}=-\int_{kh}^{(k+1)h}\phi\tilde{S} \dd{t}$ to be the test function in \eqref{discrete stress} and sum over $k\in\{0,1\ldots,N\}$ to derive
	\begin{equation}
	    \begin{aligned}
		-&\int_{0}^{T^{\prime}}\int_{\Omega}\partial_{t,h}^{-}(S^{N}):\phi\tilde{S} \dd{x}\dd{t}
		-\int_{0}^{T^{\prime}}\phi\int_{\Omega}(v^{N}\cdot\nabla S^{N}):\tilde{S} \dd{x}\dd{t} 
		\nonumber\\
		-&\int_{0}^{T^{\prime}}\phi\int_{\Omega}(S^{N}\skw{\nabla v^{N}}-\skw{\nabla v^{N}}S^{N}):\tilde{S} \dd{x}\dd{t}
		-\int_{0}^{T^{\prime}}\phi\langle \xi^{N}:\tilde{S} \rangle_{H_{\mathrm{sym,Tr}}^{1}}\dd{t} 
		\nonumber\\
		-&\int_{0}^{T^{\prime}}\phi\int_{\Omega}\gamma\nabla S^{N}:\nabla\tilde{S} \dd{x}\dd{t}
		 =-\int_{0}^{T^{\prime}}\phi\int_{\Omega}\eta(\varphi_{h}^{N})\sym{\nabla v^{N}}:\tilde{S} \dd{x}\dd{t}.
		\label{weak stress formulation tested with tilde s}
	\end{aligned}
	\end{equation}
    Notice that  
    \begin{equation*}
        \langle \xi^{N},S^{N}-\tilde{S}\rangle_{H_{\mathrm{sym,Tr}}^{1}}
        \geq\widetilde{\mathcal{P}}(\varphi_{h}^{N};S^{N})
        -\widetilde{\mathcal{P}}(\varphi_{h}^{N};\tilde{S})
        =\mathcal{P}(\varphi_{h}^{N};S^{N})
        -\mathcal{P}(\varphi_{h}^{N};\tilde{S}),
    \end{equation*}
    since $\xi^{N}\in\partial\widetilde{\mathcal{P}}(\varphi_{h}^{N};S^{N})$ and by~\eqref{restriction of dissipation potential}. Then, by summing \eqref{partial energy discrete stress energy} and \eqref{weak stress formulation tested with tilde s}, we obtain
    \begin{equation}\label{evoluationary inequality needed to be passed}
        \begin{aligned}
        -&\int_{0}^{T^{\prime}}\phi^{\prime}\int_{\Omega}\frac{1}{2}\abs{S^{N}}^{2}\dd{x}\dd{t}
        -\int_{0}^{T^{\prime}}\int_{\Omega}\partial_{t,h}^{-}(S^{N}):\phi\tilde{S} \dd{x}\dd{t}
        \\
        -&\int_{0}^{T^{\prime}}\phi\int_{\Omega}(v^{N}\cdot\nabla S^{N}):\tilde{S}+(S^{N}\skw{\nabla v^{N}}-\skw{\nabla v^{N}}S^{N}):\tilde{S} \dd{x}\dd{t}
        \\
        +&\int_{0}^{T^{\prime}}\phi\left(\mathcal{P}(\varphi_{h}^{N};S^{N})-\mathcal{P}(\varphi_{h}^{N};\tilde{S})\right)\dd{x}\dd{t}
        +\int_{0}^{T^{\prime}}\phi\int_{\Omega}\gamma\nabla S^{N}\threedotsbin(\nabla S^{N}-\nabla\tilde{S})\dd{x}\dd{t}
        \\
        -&\int_{0}^{T^{\prime}}\phi\int_{\Omega}\eta(\varphi_{h}^{N})\sym{\nabla v^{N}}:(S^{N}-\tilde{S})\dd{x}\dd{t}\leq\int_{\Omega}\frac{1}{2}\abs{S_{0}}^{2}\dd{x}.
    \end{aligned}
    \end{equation}
    Now, we carry out the limit passage in~\eqref{evoluationary inequality needed to be passed}. Since $\phi^{\prime}\leq0$, the functional 
    \begin{equation}
        S\mapsto-\int_{0}^{T^{\prime}}\phi^{\prime}\int_{\Omega}\frac{1}{2}\abs{S(t)}^{2}\dd{x}\dd{t}
    \end{equation}
    is convex and continuous on $L^{2}(0,T^{\prime};L^{2}(\Omega))$. Therefore, this functional is weakly lower semicontinuous on $L^{2}(0,T^{\prime};L^{2}(\Omega))$, which implies
    \begin{equation}
        -\int_{0}^{T^{\prime}}\phi^{\prime}\int_{\Omega}\frac{1}{2}\abs{S(t)}^{2}\dd{x}\dd{t}
        \leq
        \liminf_{N\to\infty}\left(-\int_{0}^{T^{\prime}}\phi^{\prime}\int_{\Omega}\frac{1}{2}\abs{S^{N}(t)}^{2}\dd{x}\dd{t}\right),
        \label{limit passage for generalized solution part one}
    \end{equation}
    thanks to the weak* convergence of $(S^{N})_{N}$ in $L^{\infty}(0,T^{\prime};L^{2}(\Omega))$ from~\eqref{weak convergence of piecewise linear interpolation stress}. Moreover, by the weak convergence of $(S^{N})_{N}$ in $L^{2}(0,T^{\prime};H^{1}(\Omega))$ from~\eqref{weak convergence of piecewise linear interpolation stress space derivative} and with the help of an integration by parts for difference quotients, we deduce
    \begin{equation}
        \begin{aligned}
        &-\int_{0}^{T^{\prime}}\int_{\Omega}\partial_{t,h}^{-}(S^{N}):\phi\tilde{S} \dd{x}\dd{t}
        \\
        =&-\int_{-h}^{0}\int_{\Omega}S_{0}:\frac{1}{h}\phi(t+h)\tilde{S}(t+h)\dd{x}\dd{t}
        +\int_{0}^{T^{\prime}}\int_{\Omega}S^{N}:\partial_{t,h}^{+}(\phi\tilde{S})\dd{x}\dd{t}
        \\
		\to&-\int_{\Omega}S_{0}:\tilde{S}(0)\dd{x}+\int_{0}^{T^{\prime}}\int_{\Omega}S:\partial_{t}(\phi\tilde{S}) \dd{x} \dd{t}.
		\label{limit passage for generalized solution part two}
    \end{aligned}
    \end{equation}
    for all $\tilde{S}\in C_{0,\mathrm{sym,Tr}}^{\infty}(\Omega\times[0,T))$. Next, by the weak convergence of $(S^{N})_{N}$ in $L^{2}(0,T^{\prime};H^{1}(\Omega))$ from \eqref{weak convergence of piecewise linear interpolation stress space derivative}, we have
	\begin{equation}
		\int_{0}^{T^{\prime}}\phi\int_{\Omega}-\nabla S^{N}\threedotsbin\nabla\tilde{S} \dd{x} \dd{t}
		\to\int_{0}^{T^{\prime}}\phi\int_{\Omega}-\nabla S^{N}\threedotsbin\nabla\tilde{S} \dd{x} \dd{t}, 
        \label{limit passage for generalized solution part three}
    \end{equation}
	as well as
	\begin{equation}
		\int_{0}^{T^{\prime}}\phi\int_{\Omega}\nabla S\threedotsbin\nabla S \dd{x} \dd{t}
        \leq\liminf_{N\to\infty}\left(\int_{0}^{T^{\prime}}\phi\int_{\Omega}\nabla S^{N}\threedotsbin\nabla S^{N} \dd{x} \dd{t}\right),
        \label{limit passage for generalized solution part four}   
    \end{equation}
	where we have used the same argument as in proving \eqref{limit passage for generalized solution part one} for the last inequality.\\
    {To see the limit passage regarding $\mathcal{P}(\varphi_{h}^{N};S^{N})$, consider the functional
    \begin{equation*}
        (\varphi,S)\mapsto \int_{0}^{T^{\prime}}\phi(t)\mathcal{P}(\varphi(t);S(t)) \dd{t}=\int_{0}^{T^{\prime}}\int_{\Omega}\phi(t)P(x,\varphi(t,x),S(t,x))\dd{x}\dd{t}.
    \end{equation*}
    By the weak* convergence of $(S^{N})_{N}$ in $L^{\infty}(0,T^{\prime};L^{2}(\Omega))$ from \eqref{weak convergence of piecewise linear interpolation stress space derivative} and strong convergence of $(\varphi^{N})_{N}$ in $L^{2}(0,T^{\prime};H^{1}(\Omega))$ from \eqref{strong convergence of piecewise constant interpolation phase-field} and applying \cite[][Theorem 3]{ADLoffe}, we arrive at
	\begin{equation}
		\int_{0}^{T^{\prime}}\phi\mathcal{P}(\varphi;S)\dd{t}
		\leq\liminf_{N\to\infty}\int_{0}^{T^{\prime}}\phi\mathcal{P}(\varphi_{h}^{N};S^{N})\dd{t}.
        \label{limit passage for generalized solution part five}   
	\end{equation}
    By the continuity of the mapping $\varphi\mapsto{P}(\varphi;S)$ for all $S\in L_{\mathrm{sym,Tr}}^{2}(\Omega)$ and the strong convergence of $(\varphi^{N})_{N}$ in $L^{2}(0,T^{\prime};H^{1}(\Omega))$ from \eqref{strong convergence of piecewise constant interpolation phase-field}, we also have
    \begin{equation}
        -\int_{0}^{T^{\prime}}\phi\mathcal{P}(\varphi;\tilde{S})\dd{t}
        \leq \liminf_{N\to\infty}\left(-\int_{0}^{T^{\prime}}\phi\mathcal{P}(\varphi_{h}^{N};\tilde{S})\dd{t}\right),
        \label{limit passage for generalized solution part six}
    \end{equation}
    with the help of Fatou's lemma and the fact that the integrand is bounded from below.} 
    Notice that by the weak convergence of $(S^{N})_{N}$ in $L^{2}(0,T^{\prime};H^{1}(\Omega))$ from \eqref{weak convergence of piecewise linear interpolation stress space derivative}, the strong convergence of $(v^{N})_{N}$ in $L^{2}(0,T^{\prime};L^{2}(\Omega))$ from \eqref{strong convergence of piecewise constant interpolation velocity}, the strong convergence of $(\varphi^{N})_{N}$ in $L^{2}(0,T^{\prime};H^{1}(\Omega))$ from \eqref{strong convergence of piecewise constant interpolation phase-field} and assumption \eqref{assumption on coefficients}, we obtain
    \begin{equation}
        \begin{aligned}
        &\int_{0}^{T^{\prime}}\phi\int_{\Omega}(v^{N}\cdot\nabla S^{N}):\tilde{S}+(S^{N}\skw{\nabla v^{N}}-\skw{\nabla v^{N}}S^{N}):\tilde{S} \dd{x}\dd{t}
        \\
        \to&\int_{0}^{T^{\prime}}\phi\int_{\Omega}(v\cdot\nabla S):\tilde{S}+(S\skw{\nabla v}-\skw{\nabla v}S):\tilde{S} \dd{x}\dd{t}
        \label{limit passage for generalized solution part seven}
    \end{aligned}
    \end{equation}
    as well as 
    \begin{equation}
        \int_{0}^{T^{\prime}}\phi\int_{\Omega}\eta(\varphi_{h}^{N})\sym{\nabla v^{N}}:(S^{N}-\tilde{S})\dd{x}\dd{t}
        \to\int_{0}^{T^{\prime}}\phi\int_{\Omega}\eta(\varphi)\sym{\nabla v}:(S-\tilde{S})\dd{x}\dd{t}
        \label{limit passage for generalized solution part eight}
    \end{equation}
    with the help of an integration by parts. Inserting~\eqref{limit passage for generalized solution part one}-\eqref{limit passage for generalized solution part eight} into~\eqref{evoluationary inequality needed to be passed} and adding both sides
    \begin{equation*}
        -\int_{0}^{T^{\prime}}\phi^{\prime}\int_{\Omega}\frac{1}{2}\abs{\tilde{S}(t)}^{2}\dd{x}\dd{t}
        -\int_{0}^{T^{\prime}}\phi\int_{\Omega}\partial_{t}\tilde{S}:\tilde{S}\dd{x}\dd{t}
        =\int_{\Omega}\frac{1}{2}\abs{\tilde{S}(0)}^{2}\dd{x},
    \end{equation*} 
    we obtain
    \begin{equation}
        \begin{aligned}
		-&\int_{0}^{T^{\prime}}\phi^{\prime}\int_{\Omega}\frac{1}{2}\abs{ S(t)-\tilde{S}(t)}^{2}\dd{x}\dd{t}
        \\
		&+\int_{0}^{T^{\prime}}\phi\int_{\Omega} \partial_{t}\tilde{S}:(S-\tilde{S}) - v\cdot\nabla S:\tilde{S} - (S\skw{\nabla v}-\skw{\nabla v}S):\tilde{S}\dd{x} \dd{t}
		\\
		&+\int_{0}^{T^{\prime}}\phi\left(\mathcal{P}(\varphi;S)-\mathcal{P}(\varphi;\tilde{S})\right) \dd{t}
		+\int_{0}^{T^{\prime}}\phi\int_{\Omega}\gamma\nabla S:\nabla(S-\tilde{S}) - \eta(\varphi)\sym{\nabla v}:(S-\tilde{S})\dd{x} \dd{t} 
		\\
		\leq&\frac{1}{2}\lVert S_{0}-\tilde{S}(0)\rVert_{L^2}^{2}
		\label{weak generalized solution stress}
	\end{aligned}
    \end{equation}
    Finally, by applying Lemma~\ref{result related to energy estimate} to~\eqref{weak generalized solution stress}, we arrive at~\eqref{generalized solution stress energy estimate TPP}.
    \paragraph{Step 9: $\varphi$ takes value in $(-1,1)$ a.e. in $\Omega\times(0,T)$.}
	By the total energy-dissipation estimate~\eqref{total energy estimate for generalized solution for TPP}, one can see that
	\begin{equation*}
		\int_{\Omega}W(\varphi(t))\dd{x}\leq C
	\end{equation*}
	for a.e. $t\in(0,T)$. With the aid of \eqref{assumption on double well potential W}, this implies that $\varphi\in[-1,1]$ a.e. in $\Omega\times(0,T)$. Now, we show further that $\varphi\in(-1,1)$ a.e. in $\Omega\times(0,T)$. To see this, recall that by \eqref{bounds of discrete chemical potenial} and \eqref{bound of mean value of discrete chemical potential}, $\mu$ is bounded in $L^{2}(0,T^{\prime};L^{2}(\Omega))$ for a.e. $0<T^{\prime}<T$. Besides, we have
	\begin{equation*}
    \begin{aligned}
		\lVert W^{\prime}(\varphi)\rVert_{L^{2}(0,T^{\prime};L^{2}(\Omega))}
		&=\lVert \mu+\Delta\varphi\rVert_{L^{2}(0,T^{\prime};L^{2}(\Omega))}
		\\
		&\leq \lVert \mu\rVert_{L^{2}(0,T^{\prime};L^{2}(\Omega))}+\lVert \Delta\varphi\rVert_{L^{2}(0,T^{\prime};L^{2}(\Omega))}.
	\end{aligned}
    \end{equation*}
	Since $\varphi$ is bounded in $L^{2}(0,T^{\prime};H^{2}(\Omega))$ by \eqref{extra bound of discrete phase-field}, $W^{\prime}(\varphi)$ is bounded in $L^{2}(0,T^{\prime};L^{2}(\Omega))$. Thanks to \eqref{assumption on double well potential W}, we derive that $\varphi\not=-1,1$ a.e. in $\Omega\times(0,T^{\prime})$. Since it holds true for a.e. $0<T^{\prime}<T$, we conclude that $\varphi\in(-1,1)$ a.e. in $\Omega\times(0,T)$.  This proves the assertion.
\end{proof}

\begin{remark}\label{estimate from weak second CH}
	By testing \eqref{weak second CH TPP} with $\varphi-\varphi_{\Omega}$ and repeating the argument as in Lemma \ref{estimate from discrete second CH}, we obtain an estimate
	\begin{equation*}
	    \begin{aligned}
		\lVert W_{\kappa}^{\prime}(\varphi)\rVert_{L^2}
		+\abs{\int_{\Omega}\mu \dd{x}}
		&\leq C\left(\lVert\nabla\mu\rVert_{L^2}+\lVert\nabla\varphi\rVert_{L^2}^2+1\right)
        \\
		\lVert\mathrm{D} \mathcal{E}_{\mathrm{pf},\kappa}(\varphi)\rVert_{L^2}&\leq C\left(\lVert\mu\rVert_{L^2}+1\right)
	\end{aligned}
	\end{equation*}
\end{remark}

\begin{remark}\label{a generalized solution obtained from previous theorem is a dissipative solution}
    By Proposition~\ref{generalized solution is dissipative solution for TPP}, one can see that a weak solution obtained from Theorem~\ref{existence of generalized solution for TPP} is also a dissipative solution for any regularity weight $\mathcal{K}$.
\end{remark}
    \section{Global existence result for the non-regularized two-phase system $\gamma=0$}\label{section Global existence result for the non-regularized two phase system}

In Section~\ref{section Global existence result for the regularized two phase system}, we have shown that, for all $\gamma>0$, there exists a dissipative solution $(v,S,\varphi,\mu)$. Now, we will prove the existence of a dissipative solution for the case $\gamma=0$ by carrying out the limit passage $\gamma\to0$ in this notion of solution.
\begin{theorem}\label{existence of dissipative solution for non-regularized system}
	Let $\gamma=0$. Let $v_{0},S_{0}$ and $\varphi_{0}$ satisfy Assumption \ref{assumptions on the initial data}. Let Assumption \ref{TP assumption} be satisfied, and let $\mathcal{P}$ satisfy Assumption \ref{assumption on the dissipation potential}. Then there exists a dissipative solution  $(v,S,\varphi,\mu)$ of type $\mathcal{K}$ to the limit system \eqref{two phase system alternative} with $\gamma=0$ in the sense of Definition \ref{dissipative solution for two phase problem}, where the regularity weight is given by
	\begin{align}
			\mathcal{K}(\tilde{S})&=\frac{k_{\Omega}^{2}}{\nu_{1}}\lVert \tilde{S} \rVert_{L^{\infty}}^{2}
            \label{regularity weight for two phase problem}
	\end{align}
    for all $(\tilde{v},\tilde{S},\tilde{\varphi})\in\mathfrak{T}$. Here, the constant $k_{\Omega}>0$ is the constant from Korn's inequality and $\nu_{1}$ is the constant from Assumption \eqref{assumption on coefficients}.
\end{theorem}
\begin{proof}
	For $\gamma>0$, there exists a dissipative solution $(v_{\gamma},S_{\gamma}.\varphi_{\gamma},\mu_{\gamma})$ of type $\mathcal{K}$ thanks to Theorem~\ref{existence of generalized solution for TPP} and Remark~\ref{a generalized solution obtained from previous theorem is a dissipative solution}. Now, we want to take the limit $\gamma\to0$. From Theorem~\ref{existence of generalized solution for TPP}, we have the energy inequality
    \begin{equation}
        \begin{aligned}
            \mathcal{E}_{\mathrm{tot}}(v_{\gamma}(t),S_{\gamma}(t),\varphi_{\gamma}(t))
		  +\int_{0}^{t}\mathcal{D}_{\mathrm{tot}}(v_{\gamma},S_{\gamma},\varphi_{\gamma},\mu_{\gamma})\dd{\tau}
		\leq \mathcal{E}_{\mathrm{tot}}(v_{0},S_{0},\varphi_{0})
        +\int_{0}^{t}\langle f, v_{\gamma}\rangle_{H^{1}}\dd{\tau}.    
        \end{aligned}
    \end{equation}
	\paragraph{Step 1: Compactness.}
	This estimate gives us the following $\gamma-$uniform bounds:
	\begin{subequations}
	\begin{align}
		&(v_{\gamma})_{\gamma}\mbox{ is bounded in }L^{2}(0,T^{\prime};H_{0,\mathrm{div}}^{1}(\Omega))\mbox{ and in }L^{\infty}(0,T^{\prime};L_{\mathrm{div}}^{2}(\Omega)),
		\label{bounds of dissipative solution to two phase velocity}\\
		&(S_{\gamma})_{\gamma}\mbox{ is bounded in }L^{\infty}(0,T^{\prime};L_{\mathrm{sym,Tr}}^{2}(\Omega)),
		\label{bounds of dissipative solution to two phase stress}\\
		&(\varphi_{\gamma})_{\gamma}\mbox{ is bounded in }L^{\infty}(0,T^{\prime};H^{1}(\Omega)),
		\label{bounds of dissipative solution to two phase phase-field}\\
		&(\nabla\mu_{\gamma})_{\gamma}\mbox{ is bounded in }L^{2}(0,T^{\prime};(L^{2}(\Omega))^3),
		\label{bounds of dissipative solution to two phase chemical potential}
	\end{align}
	\end{subequations}
	for a.e. $0<T^{\prime}<T$. Moreover, from Remark \ref{estimate from weak second CH} and the uniform bounds \eqref{bounds of dissipative solution to two phase phase-field} on $\nabla\varphi_{\gamma}$ and \eqref{bounds of dissipative solution to two phase chemical potential} on $\nabla\mu_{\gamma}$, we deduce the following estimate on $(\mu_{\gamma})_{\gamma}$ 
	\begin{align}
		&\int_{0}^{T^{\prime}}\abs{\int_{\Omega}\mu_{\gamma} \dd{x}}\dd{t}\leq C\cdot T^{\prime}\mbox{ for a.e. }0<T^{\prime}<T,
	\end{align}
	where $C>0$ is a constant. Then, by a classical diagonalization argument, we can extract a not relabeled subsequence and a limit  quadruplet $(v,S,\varphi,\mu)$ such that
	\begin{subequations}
		\begin{align}
			v_{\gamma}\rightharpoonup v&\mbox{ in }L^{2}(0,T^{\prime};H^{1}(\Omega)),
			\label{weak convergence of dissipative solution for two phase velocity space derivative}\\
			v_{\gamma}\overset{*}{\rightharpoonup}v&\mbox{ in }L^{\infty}(0,T^{\prime};L_{\mathrm{div}}^{2}(\Omega)),
			\label{weak convergence of dissipative solution for two phase velocity}\\
			S_{\gamma}\overset{*}{\rightharpoonup}S&\mbox{ in }L^{\infty}(0,T^{\prime};L_{\mathrm{sym,Tr}}^{2}(\Omega)),
			\label{weak convergence of dissipative solution for two phase stress}\\
			\varphi_{\gamma}\overset{*}{\rightharpoonup}\varphi&\mbox{ in }L^{\infty}(0,T^{\prime};H^{1}(\Omega)),
			\label{weak convergence of dissipative solution for two phase phase-field}\\
			\mu_{\gamma}\rightharpoonup\mu&\mbox{ in }L^{2}(0,T^{\prime};H^{1}(\Omega)),
			\label{weak convergence of dissipative solution for two phase chemical potential}
		\end{align}
		for a.e. $0<T^{\prime}<T$. Moreover,  thanks to \eqref{estimate on element in subdifferntial domain}, we also know that $(\varphi_{\gamma})_{\gamma}$ is bounded in $L^{2}(0,T^{\prime};H^{2}(\Omega))$. This yields
		\begin{align}
			\varphi_{\gamma}{\rightharpoonup}\varphi&\mbox{ in }L^{2}(0,T^{\prime};H^{2}(\Omega)).
			\label{weak convergence of dissipative solution for two phase phase-field second derivative}
		\end{align}
	\end{subequations}
	Besides, by repeating arguments in the proof of Theorem \ref{existence of generalized solution for TPP}, we conclude the following strong convergence results for a.e. $0<T^{\prime}<T$
	\begin{subequations}\label{strong convergence of dissipative solution for two phase}
		\begin{align}
			\varphi_{\gamma}\to\varphi&\mbox{ in }L^{2}(0,T^{\prime};H^{1}(\Omega)),
			\label{strong convergence of dissipative solution for two phase phase-field}\\
			v_{\gamma}\to v&\mbox{ in }L^{2}(0,T^{\prime};L^{2}(\Omega)).
			\label{strong convergence of dissipative solution for two phase velocity}
		\end{align} 
	\end{subequations}
	\paragraph{Step 2: Limit passage $\gamma\to0$ in the relative energy-dissipation estimate.}
	By Lemma \ref{result related to energy estimate}, we can transform the relative energy-dissipation estimate \eqref{relative energy estimate for two phase problem} into the following weak form
	\begin{equation}\label{weak formulation of relative energy dissiaption estimate gamma 0}
	    \begin{aligned}
        -&\int_{0}^{T^{\prime}}\phi^{\prime} \left(\mathcal{R}(v_{\gamma},S_{\gamma},\varphi_{\gamma}|\tilde{v},\tilde{S},\tilde{\varphi}\right) \dd{s}
		\\
		&+\int_{0}^{T^{\prime}}\phi \Big(\left\langle \mathcal{A}_{\gamma}(\tilde{v},\tilde{S},\tilde{\varphi}),
		\begin{pmatrix}
			v_{\gamma}-\tilde{v}\\
			S_{\gamma}-\tilde{S}\\
			-\Delta(\varphi_{\gamma}-\tilde{\varphi})+W^{\prime\prime}(\tilde{\varphi})(\varphi_{\gamma}-\tilde{\varphi})+\kappa(\varphi_{\gamma}-\tilde{\varphi})-\frac{\rho_{2}-\rho_{1}}{2}(v_{\gamma}-\tilde{v})\tilde{v}
		\end{pmatrix}
		\right\rangle_{\!\!\!\mathbb Y}
		\\
        &+\mathcal{P}(\varphi_{\gamma};S_{\gamma})-\mathcal{P}(\varphi_{\gamma};\tilde{S})
        +\mathcal{W}_{\gamma}^{(\mathcal{K})}(v_{\gamma},S_{\gamma},\varphi_{\gamma}|\tilde{v},\tilde{S},\tilde{\varphi})\Big)
		\exp\left(\int_{s}^{T^{\prime}} \mathcal{K}(\tilde{v},\tilde{S},\tilde{\varphi}) \dd{\tau} \right) \dd{s}
		\\
		\leq &\mathcal{R}(v_{0},S_{0},\varphi_{0}|\tilde{v}(0),\tilde{S}(0),\tilde{\varphi}(0)) 
		\exp\left(\int_{0}^{T^{\prime}} \mathcal{K}(\tilde{v},\tilde{S},\tilde{\varphi}) \dd{s} \right)
	\end{aligned}
	\end{equation}
	for all $\phi\in \tilde{C}([0,t])$. 
    
	First notice that, for the limit passage regarding the term $\mathcal{P}(\varphi_{\gamma};S_{\gamma})-\mathcal{P}(\varphi_{\gamma};\tilde{S})$ can be shown in a similar way as deriving \eqref{limit passage for generalized solution part five} and \eqref{limit passage for generalized solution part six}.
    
	Now, we investigate the limit passage in the terms involving operator $\mathcal{A}_{\gamma}$. To this end, from~\eqref{operator A for two phase problem part one}-\eqref{operator A for two phase problem part three}, one can notice that
	the terms involving $(v_{\gamma})_{\gamma}$, $(\nabla v_{\gamma})_{\gamma}$, $(S_{\gamma})_{\gamma}$, $(\varphi_{\gamma})_{\gamma}$, $(\nabla\varphi_{\gamma})_{\gamma}$ or $(\Delta\varphi_{\gamma})_{\gamma}$ depend on these quantities only linearly, so that we can pass to limit by weak convergence results \eqref{weak convergence of dissipative solution for two phase velocity space derivative}-\eqref{weak convergence of dissipative solution for two phase phase-field second derivative}. The only exception is the term
	\begin{equation*}
		\int_{\Omega}\gamma\nabla\tilde{S}(t)\threedotsbin\nabla(S_{\gamma}(t)-\tilde{S}(t))\dd{x},
	\end{equation*}
	because we do not have any weak convergence result for $(\nabla S_{\gamma})_{\gamma}$. Instead, notice that
	\begin{equation}\label{alternative estimate of limit passage for dissipative solution for two phase stress space derivative}
	    \begin{aligned}
		&\int_{0}^{T^{\prime}}\int_{\Omega}\gamma\nabla\tilde{S}(t)\threedotsbin\nabla(S_{\gamma}(t)-\tilde{S}(t))\dd{x} 
        \\
		\leq& 
        \sqrt{\gamma}\left(\int_{0}^{T^{\prime}}\lVert\nabla\tilde{S}\rVert_{L^{2}}^{2}\dd{s}\right)^{\frac{1}{2}}\left(\gamma\int_{0}^{T^{\prime}}\lVert\nabla S_{\gamma}\rVert_{L^{2}}^{2}\dd{s}\right)^{\frac{1}{2}}+\gamma\lVert \nabla\tilde{S}\rVert_{L^{2}(0,T^{\prime};L^{2}(\Omega))}^{2}
        ,
	\end{aligned}
	\end{equation}
	where $\gamma\norm{\nabla S_{\gamma}}_{L^{2}}^{2}\leq C$ uniformly in $\gamma$ thanks to \eqref{total energy estimate for generalized solution for TPP}. Thus, the right-hand side of \eqref{alternative estimate of limit passage for dissipative solution for two phase stress space derivative} tends $0$ as $\gamma\to0$.
    
	Now, we consider the limit passage of the terms in $\mathcal{R}
    (v_{\gamma},S_{\gamma},\varphi_{\gamma}|\tilde{v},\tilde{S},\tilde{\varphi})$. By the strong convergence of $(v_{\gamma})_{\gamma}$ in $L^{2}(0,T^{\prime};L^{2}(\Omega))$ from \eqref{strong convergence of dissipative solution for two phase velocity} and the strong convergence of $(\varphi_{\gamma})_{\gamma}$ in $L^{2}(0,T^{\prime};H^{1}(\Omega))$ from \eqref{strong convergence of dissipative solution for two phase phase-field}, we obtain
	\begin{equation}
		\lim\limits_{\gamma\to0}\left(-\int_{0}^{T^{\prime}}\phi^{\prime}\int_{\Omega}\rho_{\gamma}\frac{\abs{v_{\gamma}-\tilde{v}}^{2}}{2}\dd{x}\dd{s}\right)
		=-\int_{0}^{T^{\prime}}\phi^{\prime}\int_{\Omega}\rho\frac{\abs{v-\tilde{v}}^{2}}{2}\dd{x}\dd{s}.
		\label{limit passage of kinetic energy}
	\end{equation}
	Since $\phi^{\prime}\leq0$, $\tilde{\varphi}\in C_{0}^{\infty}(\Omega\times[0,T))$ and $\tilde{\varphi}\in(-1,1)$, the functional
	\begin{equation*}
		(S,\varphi)
		\mapsto
		-\int_{0}^{T^{\prime}}\phi^{\prime}\int_{\Omega}
		\frac{\abs{S-\tilde{S}}^{2}}{2}
		+\frac{\abs{\nabla\varphi-\nabla\tilde{\varphi}}^{2}}{2}
		-W^{\prime}(\tilde{\varphi})(\varphi-\tilde{\varphi})
		+\kappa\abs{\varphi-\tilde{\varphi}}^{2}\dd{x}\dd{s}
	\end{equation*} 
	is convex and continuous on $L^{2}(0,T^{\prime};L_{\mathrm{sym,Tr}}^{2}(\Omega))\times L^{2}(0,T^{\prime};H^{1}(\Omega))$. Therefore, it is weakly lower semicontinuous on $L^{2}(0,T^{\prime};L_{\mathrm{sym,Tr}}^{2}(\Omega))\times L^{2}(0,T^{\prime};H^{1}(\Omega))$. Hence, by the weak*-convergence of $(S_{\gamma})_{\gamma}$ in $L^{\infty}(0,T^{\prime};L_{\mathrm{sym,Tr}}^{2}(\Omega))$ from \eqref{weak convergence of dissipative solution for two phase stress} and weak*-convergence of $(\varphi_{\gamma})_{\gamma}$ in $L^{\infty}(0,T^{\prime};H^{1}(\Omega))$ from \eqref{weak convergence of dissipative solution for two phase phase-field}, we deduce 
	\begin{equation}\label{limit passage of convex energy}
	\begin{aligned}
		&-\int_{0}^{T^{\prime}}\phi^{\prime}\int_{\Omega}
		\frac{\abs{S-\tilde{S}}^{2}}{2}
		+\frac{\abs{\nabla\varphi-\nabla\tilde{\varphi}}^{2}}{2}
		-W^{\prime}(\tilde{\varphi})(\varphi-\tilde{\varphi})
		+\kappa\abs{\varphi-\tilde{\varphi}}^{2}\dd{x}\dd{s}
	   \\
		\leq& \liminf_{\gamma\to0}\left(-\int_{0}^{T^{\prime}}\phi^{\prime}\int_{\Omega}
		\frac{\abs{S_{\gamma}-\tilde{S}}^{2}}{2}
		+\frac{\abs{\nabla\varphi_{\gamma}-\nabla\tilde{\varphi}}^{2}}{2}
		-W^{\prime}(\tilde{\varphi})(\varphi_{\gamma}-\tilde{\varphi})
		+\kappa\abs{\varphi_{\gamma}-\tilde{\varphi}}^{2}\dd{x}\dd{s}\right).
	\end{aligned}
	\end{equation}
	Moreover, since $\varphi_{\gamma}\in (-1,1)$ a.e. in $\Omega\times(0,T^{\prime})$, by the continuity of $W$, we have $|W(\varphi_{\gamma})|\leq C$ a.e. in $\Omega\times(0,T^{\prime})$ for all $\gamma>0$. Hence, with the help of the dominated convergence theorem, we derive
	\begin{equation}
		\lim_{\gamma\to0}-\int_{0}^{T^{\prime}}\phi^{\prime}\int_{\Omega}W(\varphi_{\gamma})\dd{x}\dd{s}=-\int_{0}^{T^{\prime}}\phi^{\prime}\int_{\Omega}W(\varphi)\dd{x}\dd{s}.
		\label{limit passage of energy caused by double well potential}
	\end{equation}
	Summing up~\eqref{limit passage of kinetic energy}-\eqref{limit passage of energy caused by double well potential} yields
	\begin{equation}
		-\int_{0}^{T^{\prime}}\phi^{\prime} \mathcal{R}(v,S,\varphi|\tilde{v},\tilde{S},\tilde{\varphi}) \dd{s}
		\leq\liminf_{\gamma\to0}\left(-\int_{0}^{T^{\prime}}\phi^{\prime} \mathcal{R}(v_{\gamma},S_{\gamma},\varphi_{\gamma}|\tilde{v},\tilde{S},\tilde{\varphi}) \dd{s}\right).
	\end{equation}
    
    Now, we turn to the limit passage in $\mathcal{W}_{\gamma}^{(\mathcal{K})}(v_{\gamma},S_{\gamma},\varphi_{\gamma}|\tilde{v},\tilde{S},\tilde{\varphi})$. 
    For this, we recall from~\eqref{dissipation function like quantity for two phase} that
    \begin{equation*}
        \mathcal{W}_{\gamma}^{(\mathcal{K})}(v,S,\varphi|\tilde{v},\tilde{S},\tilde{\varphi})=\mathcal{W}_{0}^{(\mathcal{K})}(v,S,\varphi|\tilde{v},\tilde{S},\tilde{\varphi})+\mathcal{D}_{\mathrm{sd},\gamma}(S-\tilde{S})\,,
    \end{equation*}
    and we split $\mathcal{W}_{0}^{(\mathcal{K})}(v,S,\varphi|\tilde{v},\tilde{S},\tilde{\varphi})$ into two parts as follows:
	\begin{equation}
		\mathcal{W}_{0}^{(\mathcal{K})}(v_{\gamma},S_{\gamma},\varphi_{\gamma}|\tilde{v},\tilde{S},\tilde{\varphi})
		=
		\mathfrak{Q}(v_{\gamma},S_{\gamma},\varphi_{\gamma},\mu_{\gamma}|\tilde{v},\tilde{S},\tilde{\varphi},\tilde{\mu})
		+\mathfrak{R}(v_{\gamma},S_{\gamma},\varphi_{\gamma},\mu_{\gamma}|\tilde{v},\tilde{S},\tilde{\varphi},\tilde{\mu}),
	\end{equation}
	where
	\begin{equation}\label{quadratic terms}
	    \begin{aligned}
		\mathfrak{Q}(v,S,\varphi|\tilde{v},\tilde{S},\tilde{\varphi})
        :=&\int_{\Omega}2\nu_{1}\abs{\sym{\nabla v}-\sym{\nabla\tilde{v}}}^{2}+\frac{1}{2}\abs{\nabla\mu-\nabla\tilde{\mu}}^{2}+\frac{1}{2}\abs{\nabla\mu}^{2} \dd{x}
		\\
		&-\int_{\Omega}2(S-\tilde{S})\skw{\nabla v-\nabla\tilde{v}}:\tilde{S} \dd{x}
		\\
		&+\mathcal{K}(\tilde{v},\tilde{S},\tilde{\varphi})\int_{\Omega}\frac{\abs{S-\tilde{S}}^{2}}{2}+\frac{\abs{\nabla\varphi-\nabla\tilde{\varphi}}^{2}}{2}+\kappa\abs{\varphi-\tilde{\varphi}}^{2} 
        \dd{x},
	\end{aligned}
	\end{equation}
	collects all the quadratic terms with respect to $(v,S,\varphi,\mu)$, and where
	\begin{equation}\label{remaining terms}
	    \begin{aligned}
	    \mathfrak{R}(v,S,\varphi|\tilde{v},\tilde{S},\tilde{\varphi})
		:=
        &\int_{\Omega}2\nu(\varphi)\abs{\sym{\nabla v}-\sym{\nabla\tilde{v}}}^{2}
        -2\nu_{1}\abs{\sym{\nabla v}-\sym{\nabla\tilde{v}}}^{2}\dd{x}
    \\
        &+\int_{\Omega}2(\nu(\varphi)-\nu(\tilde{\varphi}))\sym{\nabla\tilde{v}}:\left({\nabla v}-{\nabla\tilde{v}}\right)-\frac{1}{2}\abs{\nabla\tilde{\mu}}^{2}\dd{x}
    \\
        &+\int_{\Omega}\Delta\tilde{\mu}\left(-\Delta(\varphi-\tilde{\varphi})+W^{\prime\prime}(\tilde{\varphi})(\varphi-\tilde{\varphi})\right)\dd{x}
    \\
        &+\int_{\Omega}(v-\tilde{v})\otimes(\rho v-\tilde{\rho}\tilde{v}+J-\tilde{J}):\nabla\tilde{v} 
        +(\rho-\tilde{\rho})(v-\tilde{v})\cdot\partial_{t}\tilde{v}\dd{x}
    \\
        &-\int_{\Omega}(\eta(\varphi)-\eta(\tilde{\varphi}))(S-\tilde{S}):\nabla\tilde{v}
		+(\eta(\varphi)-\eta(\tilde{\varphi}))\tilde{S}:({\nabla v}-{\nabla\tilde{v}})\dd{x}
	\\
		&-\int_{\Omega}(S-\tilde{S})\otimes(v-\tilde{v})\threedotsbin\nabla\tilde{S} \dd{x}
	\\
		&-\int_{\Omega}\tilde{\mu}(\nabla\varphi-\nabla\tilde{\varphi})\cdot(v-\tilde{v})
		-(\nabla\varphi-\nabla\tilde{\varphi})\otimes(\nabla\varphi-\nabla\tilde{\varphi}):\nabla\tilde{v} \dd{x}
	\\
		&+\int_{\Omega}\kappa(\nabla\mu-\nabla\tilde{\mu})\cdot(\nabla\varphi-\nabla\tilde{\varphi})
		+\kappa(v-\tilde{v})\cdot\nabla\tilde{\varphi}(\varphi-\tilde{\varphi}) \dd{x}
    \\
        &+\mathcal{K}(\tilde{v},\tilde{S},\tilde{\varphi})\int_{\Omega}\rho\frac{\abs{v-\tilde{v}}^{2}}{2}+W(\varphi)
        -W(\tilde{\varphi})
        -W^{\prime}(\tilde{\varphi})(\varphi-\tilde{\varphi}) \dd{x}
	\end{aligned}
	\end{equation}
	collects all the remaining terms.
    
	Next, we discuss the limit passage for the quadratic terms in $\mathfrak{Q}$. Recall that the non-negativity of a quadratic form implies convexity, see \cite[][Proposition 3.71]{bonnansPerturbationAnalysisOptimization2000b} for details. In order to show that $\mathfrak{Q}(v,S,\varphi|\tilde{v},\tilde{S},\tilde{\varphi})$ is non-negative, we make the following estimates:
	\begin{equation}\label{lower estimate of the quadratic terms}
	    \begin{aligned}
		\mathfrak{Q}(v,S,\varphi,\mu|\tilde{v},\tilde{S},\tilde{\varphi},\tilde{\mu})\geq&
        \int_{\Omega}2\nu_{1}\abs{\sym{\nabla v}-\sym{\nabla\tilde{v}}}^{2}+\frac{1}{2}\abs{\nabla\mu-\nabla\tilde{\mu}}^{2}
        +\frac{1}{2}\abs{\nabla\mu}^{2} \dd{x}
		\\
		&+\mathcal{K}(\tilde{v},\tilde{S},\tilde{\varphi})\int_{\Omega}\frac{\abs{S-\tilde{S}}^{2}}{2}+\frac{\abs{\nabla\varphi-\nabla\tilde{\varphi}}^{2}}{2}+\kappa\abs{\varphi-\tilde{\varphi}}^{2}+\abs{v-\tilde{v}}^{2} \dd{x}
		\\
		&-\abs{\int_{\Omega}2(S-\tilde{S})\skw{\nabla v-\nabla\tilde{v}}:\tilde{S} \dd{x}}
	\end{aligned}
	\end{equation}
	
    Next, we estimate the negative terms in~\eqref{lower estimate of the quadratic terms} using H\"older's inequality, Korn's inequality and Young inequality:
	\begin{equation}\label{bounds on the negative terms in the lower estiamte of the equadratic terms}
	    \begin{aligned}
		\int_{\Omega}2(S-\tilde{S})\skw{\nabla v-\nabla\tilde{v}}:\tilde{S} \dd{x}
        \leq& 2\lVert \tilde{S}\rVert_{L^{\infty}}
        \left(\int_{\Omega}\abs{S-\tilde{S}}^{2}\dd{x}\right)^{\frac{1}{2}}
        \left( \int_{\Omega}\abs{{\nabla v}-{\nabla \tilde{v}}}^{2}\dd{x}\right)^{\frac{1}{2}}
        \\
        \leq&2k_{\Omega}\lVert \tilde{S}\rVert_{L^{\infty}}
        \left(\int_{\Omega}\abs{S-\tilde{S}}^{2}\dd{x}\right)^{\frac{1}{2}}
        \left( \int_{\Omega}\abs{\sym{\nabla v}-\sym{\nabla \tilde{v}}}^{2}\dd{x}\right)^{\frac{1}{2}}
        \\
		\leq& \frac{k_{\Omega}^{2}}{\nu_{1}}\lVert \tilde{S}\rVert_{L^{\infty}}^{2}\int_{\Omega}\frac{\abs{S-\tilde{S}}^{2}}{2} \dd{x}
		+2\nu_{1}\int_{\Omega}\abs{\sym{\nabla v}-\sym{\nabla\tilde{v}}}^{2} \dd{x}.
	\end{aligned}
	\end{equation}
	By making use of the bounds \eqref{lower estimate of the quadratic terms} and \eqref{bounds on the negative terms in the lower estiamte of the equadratic terms}, one can see that $\mathfrak{Q}$ is non-negative and therefore convex. Besides, notice that $\mathfrak{Q}$ is also a continuous functional on the space $H_{0,\mathrm{div}}^{1}(\Omega)\times L_{\mathrm{sym,Tr}}^{2}(\Omega)\times H^{1}(\Omega)\times H^{1}(\Omega)$, so $\mathfrak{Q}$ is weakly lower semicontinuous on this space. This allows us to pass to the limit $\gamma\to0$, i.e., 
	\begin{equation}
	    \begin{aligned}
		&\int_{0}^{T^{\prime}}\phi\mathfrak{Q}(v,S,\varphi,\mu|\tilde{v},\tilde{S},\tilde{\varphi},\tilde{\mu})
		\exp\left(\int_{s}^{T^{\prime}}\mathcal{K}(\tilde{v},\tilde{S},\tilde{\varphi}) \dd{\tau}\right) \dd{s}
		\nonumber\\
		\leq&\liminf_{\gamma\to0}
		\int_{0}^{T^{\prime}}\phi\mathfrak{Q}(v_{\gamma},S_{\gamma},\varphi_{\gamma},\mu_{\gamma}|\tilde{v},\tilde{S},\tilde{\varphi},\tilde{\mu})
		\exp\left(\int_{s}^{T^{\prime}}\mathcal{K}(\tilde{v},\tilde{S},\tilde{\varphi}) \dd{\tau}\right) \dd{s}.	
	\end{aligned}
	\end{equation} 
    
	Now, we turn to the limit passage $\gamma\to0$ in $\mathfrak{R}$. Since $2\nu(\varphi_{\gamma})-2\nu_{1}\geq0$ by \eqref{assumption on coefficients}, then by the weak convergence of $(\sym{\nabla v_{\gamma}})_{\gamma}$ in $L^{2}(0,T^{\prime};L^{2}(\Omega))$ from \eqref{weak convergence of dissipative solution for two phase velocity space derivative} and the convergence of $(\varphi_{\gamma})_{\gamma}$ almost everywhere in $\Omega\times(0,T^{\prime})$ obtained from~\eqref{strong convergence of dissipative solution for two phase phase-field}, we deduce 
	\begin{equation}
	    \begin{aligned}
		&\int_{0}^{T^{\prime}}\phi\int_{\Omega}(2\nu(\varphi)-2\nu_{1})\abs{\sym{\nabla v}-\sym{\nabla\tilde{v}}}^{2} \dd{x}
		\exp(\int_{s}^{T^{\prime}}\mathcal{K}(\tilde{v},\tilde{S},\tilde{\varphi}) \dd{\tau}) \dd{s}
		\nonumber\\
		\leq&\liminf_{\gamma\to0}\int_{0}^{T^{\prime}}\phi\int_{\Omega}(2\nu(\varphi_{\gamma})-2\nu_{1})\abs{\sym{\nabla v_{\gamma}}-\sym{\nabla\tilde{v}}}^{2} \dd{x}
		\exp(\int_{s}^{T^{\prime}}\mathcal{K}(\tilde{v},\tilde{S},\tilde{\varphi}) \dd{\tau}) \dd{s}.
	\end{aligned}
	\end{equation}
	The limit passage of the remaining part of $\mathfrak{R}$ is a direct consequence of the convergence results \eqref{weak convergence of dissipative solution for two phase velocity space derivative}-\eqref{strong convergence of dissipative solution for two phase velocity} with the aid of \eqref{assumption on coefficients}. 
    We exemplarily discuss one term and note that the remaining terms can be handled with in a similar way. We estimate
	\begin{equation*}
    \begin{aligned}
		&\abs{\int_{0}^{T^{\prime}}\int_{\Omega}
        \left(
        (v_{\gamma}-\tilde{v})\otimes(\rho_{\gamma}v_{\gamma}-\tilde{\rho}\tilde{v})
        -(v-\tilde{v})\otimes(\rho v-\tilde{\rho}\tilde{v})\right):\nabla\tilde{v} 
        \dd{x} \dd{s}}
		\\
		\leq& \lVert \nabla\tilde{v}\rVert_{L^{\infty}(\Omega\times(0,T^{\prime}))}
		\int_{0}^{T^{\prime}}\int_{\Omega}
        \abs{(v_{\gamma}-\tilde{v})\otimes(\rho_{\gamma}v_{\gamma}-\tilde{\rho}\tilde{v})
        -(v-\tilde{v})\otimes(\rho v-\tilde{\rho}\tilde{v})} \dd{x} \dd{s},
	\end{aligned}
    \end{equation*}
	where we have used that $\tilde{v}\in C_{0,\mathrm{div}}^{\infty}(\Omega\times[0,T))$. The term on the right-hand side tends to zero, thanks to the strong convergence of $(v_{\gamma})_{\gamma}$ in $L^{2}(0,T^{\prime};L^{2}(\Omega))$ from \eqref{strong convergence of dissipative solution for two phase velocity} and the strong convergence of $(\varphi_{\gamma})_{\gamma}$ in $L^{2}(0,T^{\prime};H^{1}(\Omega))$ from \eqref{strong convergence of dissipative solution for two phase phase-field}.
    
	In total, taking the limit $\gamma\to0$ in \eqref{weak formulation of relative energy dissiaption estimate gamma 0} implies \eqref{relative energy estimate for two phase problem} with $\gamma=0$. Therefore, $(v,S,\varphi,\mu)$ is a dissipative solution to system \eqref{two phase system alternative} with $\gamma=0$.
\end{proof}
\begin{remark}
    By Proposition~\ref{energy-variational solution satisfies the weak formulation of Cahn-Hilliard equation} and Proposition~\ref{energy-variational solution satisfies the weak formulation of momentum balance}, one can see that a dissipative solution $(v,S,\varphi,\mu)$ obtained from Theorem~\ref{existence of dissipative solution for non-regularized system} satisfies weak formulation~\eqref{weak velocity TPP} and~\eqref{weak first CH TPP}. Moreover, $(v,S,\varphi,\mu)$ also satisfies~\eqref{weak second CH TPP} by using the same argument as in the proof of Theorem~\ref{existence of generalized solution for TPP}. In addition, $\varphi$ takes value $(-1,1)$ a.e. in $\Omega\times(0,T)$.
\end{remark}
    \paragraph{Acknowledgements: }F.C.\ and M.T.\ are grateful for the financial support by Deutsche Forschungsgemeinschaft (DFG) within CRC 1114 \emph{Scaling Cascades in Complex Systems}, Project-Number 235221301, Project B09 \emph{Materials with Discontinuities on Many Scales}. 
	\bibliographystyle{alpha}
	\addcontentsline{toc}{section}{References}
	\bibliography{Refs}
	\clearpage

\end{document}